\documentclass[11pt]{article}

\usepackage{amsfonts}
\usepackage{amssymb}
\usepackage{amsmath}
\usepackage{amsthm}
\usepackage{graphicx}
\usepackage{float}
\usepackage{bm}
\usepackage{color}
\usepackage{subcaption}
\usepackage{cite}
\usepackage{url}
\usepackage[ruled,linesnumbered]{algorithm2e}
\usepackage{mathtools}
\usepackage[shortlabels]{enumitem}

\theoremstyle{plain}

\newtheorem{theorem}{Theorem}
\newtheorem{lemma}{Lemma}

\theoremstyle{remark}
\newtheorem{remark}{Remark}

\usepackage[colorlinks=true, linkcolor=black, citecolor=black, urlcolor=black]{hyperref}

\title{Efficient time-domain scattering synthesis via\\ frequency-domain singularity subtraction}

\author{Oscar P. Bruno\footnote{Computing and Mathematical Sciences, California Institute of Technology, Pasadena, CA, 91125 USA, obruno@caltech.edu, msantana@caltech.edu)}
  \and Manuel A. Santana\footnotemark[1]
}

\usepackage{amsopn}

\def \C{\mathbb{C}}
\def \R{\mathbb{R}}
\def \br{\mathbf{r}}
\def \bp{\mathbf{p}}
\def \inv{^{-1}}

\renewcommand{\Re}{\operatorname{Re}}
\renewcommand{\Im}{\operatorname{Im}}
\DeclareMathOperator*{\argmax}{argmax}
\DeclareMathOperator*{\argmin}{argmin}

\usepackage{todonotes}
\usepackage[T1]{fontenc}
\usepackage{listings}
\date{}
\begin{document}

\maketitle

\begin{abstract}
  Fourier transform-based methods enable accurate, dispersion-free
  simulations of time-domain scattering problems by evaluating
  solutions to the Helmholtz equation at a discrete set of frequencies
  sufficient to approximate the inverse Fourier transform. However, in
  the case of scattering by trapping obstacles, the Helmholtz solution
  exhibits nearly-real complex resonances---which significantly slows
  the convergence of numerical inverse transform. To address this
  difficulty this paper introduces a frequency-domain singularity
  subtraction technique that regularizes the integrand of the inverse
  transform and efficiently computes the singularity contribution via
  a combination of a straightforward and inexpensive numerical
  technique together with a large-time asymptotic
  expansion. Crucially, all {\em relevant complex resonances} and
  their residues are determined via rational approximation of integral
  equation solutions at {\em real frequencies}. An adaptive algorithm
  is employed to ensure that all relevant complex resonances are
  properly identified.
\end{abstract}

\section{Introduction\label{intro}}
Recently developed ``Frequency-Time Hybrid'' (FTH)
Fourier-transform-based methods~\cite{abl,ablII} offer accurate and
efficient numerical techniques for solving exterior time-domain wave
scattering problems; corresponding interior problems can be tackled by
such approaches as well~\cite{bruno2024multiple,pan2025multi}.  These
algorithms solve the associated Helmholtz problems at a discrete set
of frequencies---typically using layer potential
formulations---combined with specialized techniques such as
``windowing-and-recentering'' and high-frequency integration, enabling
effective reconstruction of the time-domain solution via inverse
Fourier transformation.  This approach offers several advantages: it
produces essentially dispersion-free solutions, it enables
straightforward parallelization in both space and time, and it can
efficiently accommodate incident fields that persist
indefinitely. However, the performance of these methods is severely
impacted by the presence of trapping geometries. In such cases, the
Helmholtz solution exhibits nearly real complex resonances (poles),
which cause extremely slow convergence in the numerical evaluation of
the inverse Fourier transform. This work addresses that challenge
through a frequency-domain singularity subtraction technique that
regularizes the integrand and naturally transitions to a large-time
asymptotic expansion expressed in terms of the complex
resonances. These resonances and their residues are evaluated
efficiently using only real-frequency data, by employing a novel
Incidence Excitation (IE) adaptive algorithm that relies on rational
approximation in the frequency variable. The resulting overall FTH-SS
(Singularity-Subtraction FTH) approach enables accurate and efficient
time-domain scattering simulations, for arbitrarily long time, even
for highly-trapping scattering structures.\looseness = -1

Other methods, such as the Fourier transform
methods~\cite{douglas1993frequency,mecocci2000new} and the well known
Convolution Quadrature (CQ)
method~\cite{lubich1994multistep,banjai2010multistep,betcke2017overresolving}
have been presented that, like the FTH methods, rely on transformation
into the frequency domain.  A detailed discussion of certain
advantages offered by the FTH algorithms vis-\`a-vis other
frequency-time approaches (relating, in the case of the CQ method, to
time dispersion and existence of an infinite tail, and concerning
increasing cost per time-step as time grows in previous Fourier
transform methods), can be found in~\cite{abl} and will therefore not
otherwise be discussed here.

The difficulties faced by the FTH method in trapping configurations
can be traced to the frequency-dependent behavior of the scattering
solutions. Indeed, for obstacles which are strongly trapping (see
e.g.~\cite{anderson2020bootstrap,lafontaine2021most} for more precise
definitions of the concept of trapping scatterer), the frequency
domain scattering solutions typically exhibit nearly real poles as
functions of frequency~\cite{beale1973scattering}. These complex
resonances (all of which are located in the lower
half-plane~\cite{taylorbook}) manifest themselves as sharp, spike-like
features along the real frequency axis
\cite{betcke2011condition,pradovera2025surrogate,lafontaine2021most,darbas2013combining},
whose resolution requires a dense sampling of frequency
points---
thereby making the numerical inversion of the Fourier
transform prohibitively expensive.

An additional difficulty associated with the evaluation of time domain
scattering by trapping obstacles relates to the slow temporal decay of
the scattered field in trapping regions. This phenomenon has been the
focus of extensive analytical work, which connects the decay rate to
the geometric properties of the
scatterer~\cite{lax1969decaying,lax2005exponential,ikawa1982decay,morawetz1961decay,morawetz1966exponential,tang2000resonance,anderson2020bootstrap}. The
slow decay associated with trapping poses challenges for traditional
time-domain methods such as finite difference and finite element based
techniques, which require fine spatial and temporal resolution to
control dispersion errors over long simulation times. For long time
evaluation of the scattered field a popular alternative is to express
the scattered field as an asymptotic ``singularity
expansion''~\cite{baum2005singularity} in terms of the complex
resonances. Such expansions can be formally derived on the basis of
evaluation of the inverse Fourier-transform via contour deformation,
where for late enough times the contribution from the poles
dominate. Much work has gone into proving the validity of the
aforementioned singularity
expansions~\cite{tang2000resonance,vainberg1989asymptotic,lax1969decaying},
although a rigorous justification for their validity in the trapping
case has remained elusive. Nevertheless, singularity expansions have been
widely used in practice to model late-time scattering
phenomena~~\cite{baum2005singularity,wilks2024generalized,meylan2012complex,hazard2007singularity,heyman1985wavefront,meylan2014singularity,
  lalanne2018light}, although not without criticism of the formalism
sometimes used~\cite{heyman1985wavefront,dolph1980relationship,
  ramm1982mathematical}. The numerical examples presented in this
paper---including problems involving scattering by highly-trapping
obstacles---provide strong evidence for the validity and high accuracy
of the singularity expansion in the asymptotic regime, regardless of
the trapping character of the scatterer.

As mentioned above, this contribution proposes a method for evaluation
of the slowly decaying fields scattered by trapping obstacles, which
relies on a certain frequency-domain singularity subtraction
methodology. Subtraction of complex resonances and their residues near
the real axis regularizes the integrand of the inverse
Fourier-transform leading to greatly increased convergence of
numerical quadrature rules; the resonance contributions can then be
easily evaluated and re-incorporated to obtain the correct time-domain
solution. For efficiency, the contributions from the subtracted
complex resonances are computed for large times using an asymptotic
numerical algorithm based on expressions that resemble the singularity
expansion mentioned above (Section~\ref{sec:I2eval}). Importantly, the
singularity subtraction method does not require the validity of the
singularity expansion; in any case, numerical experiments provided in
this paper suggest that the singularity expansion is asymptotically
valid independently of the trapping character of the scattering
obstacles considered; see Section~\ref{sec:SEM}.

As noted in Remark~\ref{rmk:IEmethod}, the efficiency of the proposed
method relies critically on a certain ``Incident-Excitation'' (IE)
algorithm. In a modified form of the AAA rational-approximation
approach for resonance evaluation~\cite{bst}, the IE algorithm
identifies the {\em complex} resonances {\em excited by a given
  incident field} as poles of rational approximants to the
corresponding integral-equation densities, evaluated at {\em real}
frequency values. Crucially, the rational approximants used in the IE
method also enable the computation of field values and residues
without requiring additional (and costly) boundary-integral inversions
beyond those already required by the IE algorithm. This stands in
clear contrast with previous
methods~\cite{Beyn2012eigen,misawa2017boundary,asakura2009numerical,alves2024wave,el2020rational,guttel2017nonlinear,bst}
for evaluating complex resonances, which, in the present setting,
necessitate the evaluation of integral resolvents over significantly
larger sets of real and complex frequencies.

Soon after a preliminary version~\cite{bruno2025efficient} of the
present work was made available, an alternative strategy was proposed
in the pre-publication~\cite{wilber2025time} that also addresses the
challenges posed by the presence of complex resonances near the real
frequency axis.  In that work, the inverse Fourier transform is
evaluated via contour deformation into a rectangular contour contained
in the {\em upper-half frequency plane}---wherein no scattering poles
exist, and where the frequency-domain scattering solution is an
analytic function. The use of the new contour results in a more
regular integrand along the horizontal segment (parallel to the real
frequency axis), on account of the larger distance from that contour
portion to the polar singularities---which, as indicated above, are
located in the {\em lower half-plane}. However, a number of challenges
arise from the use of such a strategy. Section~4.2
in~\cite{wilber2025time} identifies one such difficulty, namely, the
numerical overflow caused by certain exponentially large functions of
time that emerge during the contour integration process. To address
this, the authors reduce the distance from the horizontal segment to
the real axis as time increases---a procedure that ultimately
undermines the main objective of the strategy. Further, we suggest
that a potentially more fundamental difficulty lies in the fact that
exponentially large time-dependent functions are to be
multiplicatively canceled by correspondingly exponentially small
terms that arise as the result of highly oscillatory integration with
respect to frequency. Indeed, given the extremely large upper bound of
approximately $10^{308}$ representable in IEEE double-precision
arithmetic, it is clear that cancellation errors—arising from the
multiplication of large exponentials by the results of high-frequency
integration—will occur well before reaching the overflow limit.  The
resulting cancellation errors inevitably lead to a complete loss of
accuracy as time increases---and ultimately to exponential growth in
the numerical solution, while the true physical solution remains
bounded. In contrast, the method proposed in this paper remains
accurate for arbitrarily long times. It naturally yields a long-time
asymptotic expansion—obtained via a certain {\em lower half-plane}
contour integration approach that delivers large times solution values
at negligible cost.

This paper is structured as follows. Section~\ref{sec:prelims}
introduces the time-domain scattering problem, and provides an
overview of the FTH method. The relevant frequency-domain integral
equations and their connection to complex resonances is also discussed
in that section. Section~\ref{sec:AAA_and_res_alg} presents the AAA
algorithm for rational approximation~\cite{AAA} and related variants,
leading to the introduction of the novel Incident-Excitation algorithm
for efficiently computing the resonances excited by a given incident
field. In Section~\ref{sec:RS}, the singularity subtraction procedure
is described in detail, along with a numerical technique for
evaluating the singular contributions and their asymptotic
expansion. Section~\ref{sec:Num_Imp} then provides a complete
description of the FTH-SS method. Finally, Section~\ref{sec:Num_Res}
presents numerical experiments that demonstrate the accuracy and
effectiveness of the proposed approach. These results highlight the
features of the incidence-excited resonance evaluation algorithm, the
smoothing effects of the singularity subtraction strategy, the
accuracy and convergence of the overall methodology, its capacity for
efficient long-time simulation, and compelling numerical evidence
supporting the validity of the singularity expansion in highly
trapping configurations.

%%%% END OF INTRODUCTION %%%%%%%%%%%%%%%%%%%%%%%%%%%%%%%%%%%%%%%%%%%%%%%
\section{Preliminaries}\label{sec:prelims}
We are concerned with the problem of scattering of waves governed by
the wave equation
\begin{subequations}\label{eqn:wave}
    \begin{align}
    \frac{\partial^2 u}{\partial t^2}(\br,t) &- c^2 \Delta u(\br,t) = 0, \quad \br \in \Omega^\mathrm{ext}\\
    u(\br,0) &= \frac{\partial u}{\partial t}(\br,0) = 0, \quad \br \in \Omega^\mathrm{ext}\\
    u(\br,t) &= b(\br,t) \quad \text{for} \quad (\br,t) \in \Gamma \times [0,T^{\mathrm{inc}}]\label{eqn:wavebnd}
\end{align}
\end{subequations}
where the open set $\Omega^\mathrm{ext}\subset \R^2$ is an ``exterior
domain'' with boundary $\Gamma$, which equals either the exterior of a
{\em closed} curve $\Gamma$ (such as e.g. the the unit circle) or the
complement of an {\em open} curve $\Gamma$ (such as a straight
segment, a circular section, etc.). The methods and ideas to be
developed, which concern the frequency-time duality, should be
applicable in both 2D and 3D contexts, but, for the sake of
simplicity, this paper is restricted to the 2D context only.  For a
given incident field $u^\mathrm{inc}$ defined in the exterior of
$\Gamma$, the solution $u$ of the problem~\eqref{eqn:wave} with
boundary values $b = -u^\mathrm{inc}$ is the ``scattered field''; the
total field in the exterior of $\Gamma$ is thus given by
$u^\mathrm{tot}(\br,t) = u^\mathrm{inc}(\br,t) + u(\br,t)$. With these
notations, the time domain boundary conditions~\eqref{eqn:wavebnd}
become
\begin{equation}\label{eqn:wave_bc}
    u^\mathrm{tot}(\br,t)  = u^\mathrm{inc}(\br,t) + u(\br,t) = 0, \quad \br \in \Gamma.
\end{equation}

To solve equation~\eqref{eqn:wave} we build upon the Fourier transform
based method recently introduced in~\cite{abl}, which we refer to
henceforth as the Frequency-Time Hybrid method (FTH). In brief, the FTH method utilizes the Fourier transforms $U = U(\br,\omega)$ and $B= B(\br,\omega)$ of the functions $u$
and $b$ with respect to time, respectively. It obtains $U$ as the solution of the Helmholtz equation problem 
\begin{subequations}\label{eqn:helmholtz}
    \begin{align}
    \Delta U(\br,\omega) + \kappa^2(\omega)U(\br,\omega) &= 0, \quad \br \in \Omega^\mathrm{ext}\label{eqn:helm_pde}\\
    U(\br,\omega)  &= B(\br,\omega) \quad \br \in \Gamma,\label{eqn:helm_bndc}\\
    \lim_{|\br| \to \infty}\sqrt{|\br|} \left(\frac{\partial U}{\partial |\br|} - \mathrm{i}\kappa(\omega)U\right) &= 0, \text{ uniformly in all directions } \br / |\br|
\end{align}
\end{subequations}
with linear dispersion relation $\kappa = \kappa(\omega) = \omega /c$,
where $\mathrm{i}$ denotes the complex unit. It then produces the
time-domain solution $u(\br,t)$ as the inverse Fourier transform of
$U$.

As pointed out in~\cite{abl}, a straightforward application of these
ideas presents a number of difficulties which, however, may be
effectively bypassed to yield an effective, fast an accurate
time-domain computational technique for the solution of the
problem~\eqref{eqn:wave}. A brief discussion concerning these
challenges and their resolution is presented in
Section~\ref{sec:abl_rev}. The method of boundary integral operators
for computing the solution to the Helmholtz
problem~\eqref{eqn:helmholtz} for open and closed curves are then
reviewed in Section~\ref{sec:inteqs}. Certain specialized
high-frequency quadrature rules that are used in the inverse Fourier
transform process are outlined in Section~\ref{sec:inv_four_tran}, and
finally a brief discussion of complex resonances is given in
Section~\ref{sec:complex_resonances}.

\subsection{Frequency-time hybrid method}\label{sec:abl_rev}
For definiteness throughout this paper we restrict attention to one of the
most commonly occurring boundary conditions, namely, incident fields
which impinge along a single direction $\bp$---so that
$b(\br,t) = a(t - \bp \cdot \br /c)$---but general boundary conditions
can be treated similarly~\cite{ablII}.  Further, the function $a(t)$
is assumed infinitely smooth and compactly supported in the 
interval $[0,T^\mathrm{inc}]$; see also
Remark~\ref{rmk:incident}. Under this assumption the boundary
condition function $b(\br,t)$ in~\eqref{eqn:wavebnd} may be expressed
in the form
\begin{equation}\label{eqn:singleinc}
  b(\br,t) = \frac{1}{2\pi}\int_{-\infty}^\infty A(\omega)B_\bp(\br,\omega) e^{-\mathrm{i}\kappa(\omega)ct}d\omega, \quad \text{where} \quad A(\omega) = \int_{-\infty}^{\infty} a(t)e^{\mathrm{i}\omega t}dt,
\end{equation}
and where
\begin{equation}\label{eq:BP}
B_\bp(\br,\omega) = e^{\mathrm{i}\kappa(\omega) \bp \cdot \br}
\end{equation}
---with integrals that can be produced with high accuracy by means of
the FTH specialized quadrature rules reviewed in this section and
Section~\ref{sec:inv_four_tran}.

\begin{remark}\label{rmk:incident}
 The proposed Fourier transform approach can continue to be used
  with high accuracy even when the time-domain boundary data $b(t)$
  or the function $a(t)$ do not vanish at $t = T^\mathrm{inc}$. This
  is achieved by suitably extending the given function to one that
  vanishes smoothly for some time $T > T^\mathrm{inc}$; by causality,
  the resulting solution $u$ coincides with the solution sought up to
  time $t= T^\mathrm{inc}$. Additionally, incident fields which only approximately vanish up a numerical tolerance $\tau$ at $t = T^{\mathrm{inc}}$ can also be treated effectively by the proposed approach.
\end{remark}

The direct computation of $A(\omega)$ according to
equation~\eqref{eqn:singleinc} presents certain challenges for large
values of $T^\mathrm{inc}$---that is to say, in cases for which the
incidence-field function $a(t)$ continues to take non-vanishing values
up to large times $t$. Indeed, for such large $t$ values the
exponential factor $e^{\mathrm{i}\omega t}$ in the second integral
of~\eqref{eqn:singleinc} is highly oscillatory with respect to
$\omega$. Consequently, the integral $A(\omega)$ also becomes highly
oscillatory, and, thus, its evaluation over a dense set of frequency
points is necessary to ensure adequate sampling. This, in turn,
requires the solution of frequency-domain problems for a large number of
frequencies $\omega$, resulting in significant computational costs.

To tackle this and related issues concerning high-frequency
integration, a partition-of-unity set $\{w_k(t)\ :\ k = 1,\dots,K\}$
of windowing functions is used, where each function $w_k$ is supported
in the interval $[s_k - H, s_k + H]$ for a corresponding support
center $s_k\in [0,T^\mathrm{inc}]$ and $k$-independent window size $H$, and where
the functions $w_k$ satisfy the ``partition-of-unity'' property
\begin{equation}\label{POU}
  \sum_{k = 1}^Kw_k(t) = 1\quad \mbox{for}\quad t\in [0,T^\mathrm{inc}].
\end{equation}
In this work we use the window functions $w_k(t) = w(t - s_k)$ where
\begin{equation}\label{eqn:window_func}
  w(t;H) = w(t) = \begin{cases}
    1, & |t| < \alpha H\\
    \frac{1}{2} \mathrm{erfc}\left[-\rho + 2\rho \left(\frac{|t| - \alpha H}{(1 - \alpha)H} \right)\right],& \alpha H \le |t| \le H\\
    0, & |t| > H.
           \end{cases}
\end{equation}
For a given window size $H$ the partition of unity condition~\eqref
{POU} can be made to hold by appropriately choosing the centers
$s_k$---as it follows easily from the fact that
$\mathrm{erf}(t) = 1 - \mathrm{erfc}(t)$ is an odd function of
$t$. Note that, while the functions $w_k (t)$ obtained in this fashion
are not strictly compactly supported, they do tend to zero extremely
fast as $t$ grows, and they essentially vanish at $t-s_k = H$, to any
prescribed numerical precision $\varepsilon_\mathrm{np}$, provided the
value of the parameter $\rho$ is selected appropriately. Throughout
this paper the values, $\alpha = 0.5$, and $\rho = 5.805$ were used,
for which $|w(t)| \leq \varepsilon_\mathrm{np} = 1.1\cdot 10^{-16}$
for $|t|\geq H$, with a corresponding departure from one of less than
$\varepsilon_\mathrm{np}$ for $|t| = \alpha
H$. Following~\cite[Sec. 3.1]{abl} in all cases we set
$s_k = 3(k - 1)H/2$ and $H = 10$.

Letting $a_k(t) = w_k(t)a(t)$ we write
\begin{equation}\label{eqn:BandBk}
    A(\omega) = \sum_{k=1}^K A_k(\omega) \quad \text{ where } \quad A_k(\omega) = \int_{s_k - H}^{s_k + H}a_k(t)e^{\mathrm{i}\omega t} dt.
\end{equation}
A change of variables to recenter the integration around the origin gives 
\begin{equation}\label{eqn:bkslow}
    A_k(\omega) = \int_{-H}^Ha_k(t)e^{\mathrm{i}\omega (t + s_k)}dt = e^{\mathrm{i}\omega s_k} \int_{-H}^Ha_k(t)e^{\mathrm{i}\omega t}dt = e^{\mathrm{i}\omega s_k} A_k^{\mathrm{slow}}(\omega)
\end{equation}
where, as implied by the notation in equation~\eqref{eqn:bkslow},
$A_k^{\mathrm{slow}}(\omega)$ is defined to equal the second integral
in that equation.
\begin{remark}\label{rmk:Bkslow}
  Clearly, $A_k^{\mathrm{slow}}$ is a ``slowly varying'' function of
  $\omega$, in that its derivatives with respect to $\omega$ are
  uniformly bounded for all $k$, provided $a$, and, thus, $a_k$ for
  all $k$, are bounded functions of $t$. As a result these functions
  may be represented numerically on the basis of their values at fixed
  numbers of discretization points.
\end{remark}

For $k = 1,\dots, K$ we then define windowed time domain boundary
functions
\begin{equation}
    b_k(\br,t) = \frac{1}{2\pi}\int_{-\infty}^\infty A_k(\omega)B_\bp(\br,\omega) e^{-\mathrm{i}\kappa(\omega)ct}d\omega = \frac{e^{\mathrm{i}\omega s_k}}{2\pi} \int_{-\infty}^\infty A_k^{\mathrm{slow}}(\omega) B_\bp(\br,\omega) e^{-\mathrm{i}\kappa(\omega)ct}d\omega.
\end{equation}
Then denoting by $u_k(\br,t)$ the solution to~\eqref{eqn:wave} with $b(\br,t)$ replaced with $b_k(\br,t)$ we obtain
\begin{equation}\label{eqn:u_sum_uk}
  u(\br,t) = \sum_{k=1}^K u_k(\br,t).
\end{equation}

It is important to note that, because $a(t)$ and $a_k(t)$ are both
infinitely smooth and compactly supported, they are essentially
band-limited. Indeed, a straightforward argument based on repeated
integrations-by-parts shows that the associated Fourier transforms
$A(\omega)$, $A_k(\omega)$, and $A_k^{\mathrm{slow}}(\omega)$ decay
super-algebraically fast as $|\omega| \to \infty$---that is, faster
than any negative power of $\omega$. For example, given that $a(t)$ is
compactly supported in the interval $[0,T^\mathrm{inc}]$, integrating
by parts $n$ times the second expression in~\eqref{eqn:singleinc}
yields
\begin{equation}\label{eqn:Asuperalg}
  |A(\omega)| \le  \frac{1}{|\omega|^n}\int_0^{T^\mathrm{inc}} |a^{(n)}(t)|dt\qquad \mbox{for all}\qquad n\in\mathbb{N}.
\end{equation}
It follows that, for a given interval 
\begin{equation}
  \label{eq:fr_int}
  I =I(W_1,W_2)= [W_1,W_2]\qquad\text{with} \qquad W_1< 0 < W_2
\end{equation}
we have
\begin{equation}\label{eqn:AepsBnd}
  |A(\omega)|\leq  \varepsilon(\mu(W_1,W_2))\quad\text{for} \quad \omega\not\in I,
\end{equation}
where
\begin{equation}\label{mu}
 \mu(W_1,W_2) = \min\{W_2,-W_1\},
\end{equation}
and where
\begin{equation}\label{eqn:eps_m_def}
  \varepsilon(\mu) = \varepsilon(\mu(W_1,W_2)) =\inf_{n \in \mathbb{N}} \frac{1}{|\mu(W_1,W_2)|^n}\int_0^{T^\mathrm{inc}} |a^{(n)}(t)|dt\to 0\quad\text{super-algebraically fast}
\end{equation}
as $\mu = \mu(W_1,W_2)\to +\infty$. (Note that the condition
$\mu(W_1,W_2)\to +\infty$ is equivalent to $W_1\to -\infty$ and
$W_2\to+\infty$.) Clearly, similar estimates hold for $A_k(\omega)$
and $A_k^{\mathrm{slow}}(\omega)$. In particular, we have
\begin{equation}
  \label{eq:band-limit-B}
  A(\omega) \approx 0,\quad A_k(\omega) \approx 0,\quad \mbox{and}\quad A_k^{\mathrm{slow}}(\omega)\approx 0\quad \mbox{for} \quad\omega\not\in I(W_1,W_2)\quad\mbox{as}\quad \mu(W_1,W_2)\to +\infty.
\end{equation}

Letting $U_k(\br,\omega)$ and $U_k^{\mathrm{slow}}(\br,\omega)$ denote
the solutions of equation~\eqref{eqn:helmholtz} with boundary values
$B(\br,\omega) = A_k(\omega)B_\bp(\br,\omega) $ and
$B(\br,\omega) = A_k^{\mathrm{slow}}(\omega)B_\bp(\br,\omega) $, respectively, it follows that $u_k(\br,t)$ may be expressed
in the forms
\begin{equation}\label{wind_uk}
    u_k(\br,t) = \frac{1}{2\pi}\int_{-\infty}^{\infty}U_k(\br,\omega)e^{-\mathrm{i}\omega t}d\omega = \frac{1}{2\pi}\int_{-\infty}^{\infty}U_k^{\mathrm{slow}}(\br,\omega)e^{-\mathrm{i}\omega (t - s_k)}d\omega.
\end{equation}
Further, in view of~\eqref{eqn:AepsBnd}--\eqref{eq:band-limit-B} it follows that
\begin{equation}
  \label{eq:band-limit-U}
  U(\omega) \approx 0,\quad U_k(\omega) \approx 0,\quad \mbox{and} \quad  U_k^\mathrm{slow}(\omega) \approx 0 \quad \mbox{for} \quad\omega\not\in I,
\end{equation}
and $u_k(\br,t)$ is closely approximated by a
Fourier integral supported in the fixed interval $I=I(W_1,W_2)$,
\begin{equation}\label{eqn:ukw}
    u_k(\br,t) \approx u^{I}_k(\br,t) = \frac{1}{2\pi}\int_{W_1}^{W_2}U_k^{\mathrm{slow}}(\br,\omega)e^{-\mathrm{i}\omega (t - s_k)}d\omega,
\end{equation}
with errors that decay super-algebraically fast, and uniformly for
$(\br,t)\in \Omega^\mathrm{ext}\times\mathbb{R}$ and $k\in\mathbb{N}$,
as $\mu(W_1,W_2)$ grows. The wave equation
solution~\eqref{eqn:u_sum_uk} may then be approximated by summation over $k$:
\begin{equation}\label{eqn:uab}
    u(\br,t) \approx u^{I}(\br,t) \coloneq \sum_{k=1}^K u^{I}_k(\br,t).
\end{equation}

The required frequency-domain solutions
$U_k^{\mathrm{slow}}(\br,\omega)$ may be obtained by means of any
available Helmholtz solver. In this paper we employ layer potential
methods for this purpose; the specific methods we use are reviewed in
the following section.

\subsection{Frequency-domain integral equation solutions}\label{sec:inteqs}
We consider first the case in which $\Gamma$ is a closed curve, and we
define the single-layer $\mathcal{S}_\omega[\psi](\br)$ and
double-layer potentials $\mathcal{K}_\omega[\psi](\br)$ for a certain
density function $\psi$,
\begin{equation}
    \mathcal{S}_\omega[\psi](\br) = \int_\Gamma G_\omega(\br,\br')\psi(\br',\omega)d\sigma(\br') \quad \text{and} \quad \mathcal{K}_\omega[\psi](\br)=\int_{\Gamma}\frac{\partial G_\omega(\br,\br')}{\partial n(\br')}\psi(\br',\omega)d\sigma(r'), \quad \br \in \Omega^e.
\end{equation}
Here
\begin{equation}\label{eqn:green}
    G_\omega(\br,\br') = \frac{\mathrm{i}}{4}H_0^1\left(\frac{\omega}{c}|\br - \br'| \right)
\end{equation}
denotes the 2D Helmholtz Green function, where $H_0^1$ is the
zeroth-order Hankel function of the first kind.  Then for the closed
curves $\Gamma$ under consideration, the solution $U$ of the Helmholtz
problem~\eqref{eqn:helmholtz} can be represented in the combined-field
form
\begin{equation}\label{eqn:combined_rep}
  U(\br,\omega) = \mathcal{C}_{\omega,\eta}[\psi](\br,\omega) \coloneqq \mathcal{K}_\omega[\psi](\br,\omega) - \mathrm{i}\eta \mathcal{S}_\omega[\psi](\br,\omega), \quad \br \in \Omega^{\mathrm{e}},  
\end{equation}
where $\eta\in\mathbb{R}$, $\eta \ne 0$. Using the frequency domain single-
and double-layer operators
 \begin{equation}\label{eqn:sl_dl_operator}
    (S_\omega\psi)(\br) = \int_\Gamma G_\omega(\br,\br')\psi(\br',\omega)d\sigma(r') \quad \text{and} \quad (K_\omega \psi)(\br) = \int_{\Gamma}\frac{\partial G_\omega(\br,\br')}{\partial n(\br')}\psi(\br',\omega)d\sigma(r'), \quad \br \in \Gamma,
\end{equation}
respectively, and defining the combined-field boundary integral operator
\begin{equation}\label{eqn:combined_op}
     (C_{\omega,\eta} \psi) \coloneqq \frac{1}{2} \psi(\br,\omega) + (K_\omega \psi)(\br) - \mathrm{i}\eta(S_\omega \psi)(\br), \quad \br \in \Gamma,
 \end{equation}
the density $\psi(\br',\omega)$ may be obtained as the
unique solution of the integral equation
\begin{equation}\label{eqn:combined_bie}
   (C_{\omega,\eta} \psi) = B(\br,\omega), \quad \br \in \Gamma.
\end{equation}

As is well known the representation~\eqref{eqn:combined_rep} is not
applicable in case $\Gamma$ is an open curve. In such cases, the
solution of the Helmholtz problem~\eqref{eqn:helmholtz} may instead be
represented in the form
\begin{equation}\label{eqn:slp_rep_freq}
    U(\br,\omega) = \mathcal{S}_\omega[\phi](\br,\omega), \quad \br \in \Omega^{\mathrm{e}}
\end{equation}
where $\phi$ denotes the unique solution of the boundary integral equation
\begin{equation}\label{eqn:open_bie}
  (S_\omega\phi)(\br) = B(\br,\omega), \quad \br \in \Gamma.
\end{equation}

The numerical implementations used in this paper for the closed-curve
operator~\eqref{eqn:combined_op} are based on the Nystr\"{o}m
methods~\cite[Sec. 3.6]{COLTON:2019}. The corresponding
implementations~\cite{bruno2013high} for the single-layer operator, in
turn, are used for the open-arc problem. Following the latter
reference, in particular, in the open arc case a smooth
parametrization $\br = \br(t)$ of $\Gamma$ $(-1 \le t \le 1)$ is used
to express the integral density $\phi$ in the form
\begin{equation}
  \label{eq:phi-psi}
  \phi(\br(t')) = \psi(\br(t')) / \sqrt{1 - t'^2},
\end{equation}
where $\psi$ is a smooth function, and where the square-root
denominator explicitly accounts for the singularities of the density
$\phi$ at the edges of the open curve $\Gamma$. Using $\psi$ as the
unknown we may write
\begin{equation}
  \label{eq:arc_COV}
  (S^\mathrm{arc}_\omega\psi) = (S_\omega\phi), 
\end{equation}
for a certain operator
$S^\mathrm{arc}_\omega$~\cite{lintner2015generalized,bruno2013high},
and equation~\eqref{eqn:open_bie} becomes
\begin{equation}\label{eqn:open_bie-sqrt}
  (S^\mathrm{arc}_\omega\psi)(\br) = B(\br,\omega), \quad \br \in \Gamma;
\end{equation}
the solution $U$, given by equation~\eqref{eqn:slp_rep_freq} with
$\phi$ re-expressed in terms of $\psi$ via
equation~\eqref{eq:phi-psi}, can be written as
\begin{equation}\label{eqn:slp_rep_freq_arc}
  U(\br,\omega) = \mathcal{S}^\mathrm{arc}_\omega[\psi](\br,\omega), \quad \br \in \Omega^{\mathrm{e}}
\end{equation}
for a certain operator $\mathcal{S}^\mathrm{arc}_\omega$. (An
application of the change of variables $t' = \cos(\theta')$ in the
$\br = \br(t')$ parametrized version of
equation~\eqref{eqn:open_bie-sqrt} produces a Jacobian which exactly
cancels the (explicit) edge singularity, and, further, it enables the
representation of the singular function $\phi$ in terms of a rapidly
convergent cosine series for the smooth density $\psi$. The open-curve
algorithm is completed~\cite{bruno2013high} by exploiting a quadrature
rule that leverages exact integration of the product of cosine Fourier
basis functions and a logarithmic kernel.)  Throughout this paper the
symbol $H_\omega$ is used to denote either $C_{\omega,\eta}$ or
$S^\mathrm{arc}_\omega$, depending on context:
\begin{equation}
  \label{eq:C_or_S}
  H_\omega = C_{\omega,\eta}\quad\mbox{for closed curve problems, and}\quad H_\omega = S^\mathrm{arc}_\omega\quad \mbox{for open arc problems.}
\end{equation}

Clearly, the function $U_k^{\mathrm{slow}}$ in~\eqref{eqn:ukw}, that
is required by the FTH method to produce the $k$-th time-domain
solution $u_k$, is the solution of the Helmholtz
equation~\eqref{eqn:helmholtz} with boundary values related to
$B_\bp(\br,\omega)$ in~\eqref{eq:BP}: 
\begin{equation}
  \label{eq:bnd_cond_omegaj}
  B(\br,\omega) = A_k^{\mathrm{slow}}(\omega)B_\bp(\br,\omega) = A_k^{\mathrm{slow}}(\omega)e^{\mathrm{i}\kappa(\omega)\bp
  \cdot \br}.
\end{equation}
As indicated in what follows, the functions $U_k^{\mathrm{slow}}$ may
be efficiently obtained, for both closed- and open-curve problems, on
the basis of the integral formulations just described.  To achieve
this, we define the $k$-independent set 
\begin{equation}\label{eqn:F_freqs}
    \mathcal{F} = \{\omega_1, \dots, \omega_J\} \subset I
\end{equation}
(see~\eqref{eq:fr_int}) containing $J$ equispaced
frequencies, which is assumed to be sufficient for evaluating the
integral in~\eqref{eqn:ukw}---for all $k$ and within a given error
tolerance---once again by means of the FTH specialized quadrature
rules based on windowing and recentering described in
Section~\ref{sec:inv_four_tran}.  Further, letting
\begin{equation}
  \label{eq:psi_p_sol}
  \psi_\bp(\br',\omega)\quad = \quad\mbox{``Solution $\psi =\psi(\br',\omega)$ of~\eqref{eqn:combined_bie} or~\eqref{eqn:open_bie-sqrt}, as applicable, with $B(\br,\omega) = B_\bp(\br,\omega)$''}.
\end{equation}
we define the $k$-independent density-solution set
\begin{equation}\label{eqn:dens_set}
  \mathcal{D}_\mathcal{F} =\{\psi_\bp(\cdot,\omega_j)\, :\, 1 \le j \le
  J\}.
\end{equation}
Calling
\begin{equation}
  \label{eq:Up}
  U_\bp(\br,\omega)\mbox{ the 
    Helmholtz solution given  by~\eqref{eqn:combined_rep} or~\eqref{eqn:slp_rep_freq_arc}, as applicable, with  density
    $\psi = \psi_\bp$,}
\end{equation}
(wherein, clearly, $U_\bp$ is independent of $k$, and where $\psi_\bp$
is expressed in terms of a corresponding density $\phi_\bp$, in
accordance with~\eqref{eq:phi-psi}, in the open-arc case), the $k$-th
Helmholtz solution $U_{k}^{\mathrm{slow}}(\br,\omega_j)$
($1\leq j\leq J$), which takes on the boundary
values~\eqref{eq:bnd_cond_omegaj} for $\br\in\Gamma$, is given by
\begin{equation}
  \label{eq:ukslow}
  U_k^{\mathrm{slow}}(\br,\omega_j) =
  A_{k}^{\mathrm{slow}}(\omega_j)U_\bp(\br,\omega_j).
\end{equation}
Thus, as suggested above, the
$k$-independent set of solutions
$\mathcal{D}_\mathcal{F}$ suffices to evaluate the necessary functions
$U_k^{\mathrm{slow}}$ for all $k$.

In what follows the discrete forms (given
in~\cite[Sec. 3.6]{COLTON:2019} and~\cite{lintner2015generalized},
respectively) of the operators $C_{\omega,\eta}$ and
$S^\mathrm{arc}_\omega$ in~\eqref{eqn:combined_op}
and~\eqref{eq:arc_COV} on an $N$-point discretization
$\{\br_1,\dots,\br_N \}$ of $\Gamma$ are respectively denoted by
\begin{equation}
  \label{eq:discr_opers}
  C_{\omega,\eta}^N \quad\mbox{and}\quad S_\omega^{\mathrm{arc},N}.
\end{equation}
Paralleling~\eqref{eq:C_or_S} we let
\begin{equation}
  \label{eq:C_or_S-num}
  \widetilde{H}_\omega = C_{\omega,\eta}^N\quad\mbox{for closed curve problems, and}\quad \widetilde{H}_\omega = S_\omega^{\mathrm{arc},N}\quad \mbox{for open arc problems.}
\end{equation}
The corresponding numerical solutions of~\eqref{eqn:combined_bie}
or~\eqref{eqn:open_bie-sqrt}, as applicable, with discrete boundary
values
\begin{equation}
  \label{eq:BP-disc}
   \widetilde{B}_\bp(\omega) = (\widetilde{B}_{\bp,1}(\omega)),\dots,\widetilde{B}_{\bp,N}(\omega))  = (e^{\mathrm{i}\kappa(\omega)\bp \cdot \br_i})_{i=1}^N
\end{equation}
(cf.~\eqref{eq:BP}), are denoted by
\begin{equation}\label{disc_vrsn}
  \widetilde\psi_{\bp}(\omega)  = \widetilde{H}_{\omega}^{-1} \widetilde{B}_\bp(\cdot ,\omega);\quad  \widetilde\psi_{\bp}=\widetilde\psi_{\bp}(\omega) = (\widetilde\psi_{\bp,1},\dots,\widetilde\psi_{\bp,N})^T\approx (\psi_\bp(\br_1,\omega),\dots,\psi_{\bp}(\br_N,\omega))^T.
\end{equation}
Once the numerical density vector $\widetilde\psi_{\bp}(\omega)$ has
been obtained, the numerical approximation
$\widetilde{U}_\bp(\mathbf{r},\omega)$ at any given point
$\mathbf{r}\in\Omega^\mathrm{ext}$ is obtained for closed curve and
open are problems via
\begin{equation}\label{eqn:Unum}
 \widetilde{U}_\bp(\br, \omega) = \mathcal{C}_{\rho_n,\eta}^N[\widetilde\psi_{\bp}](\br)\quad \text{and} \quad \widetilde{U}_\bp(\br, \omega) =  \mathcal{S}^{\mathrm{arc},N}_{\omega}[ \widetilde\psi_{\bp}](\br),
\end{equation}
respectively, where
\begin{equation}\label{eqn:disc_prop_ops}
    \mathcal{C}_{\rho_n,\eta}^N \quad \text{and} \quad  \mathcal{S}^{\mathrm{arc},N}_{\omega}
\end{equation}
denote discrete versions of the continuous operators
$\mathcal{C}_{\omega,\eta}$~\eqref{eqn:combined_rep} and
$\mathcal{S}^\mathrm{arc}_\omega$~\eqref{eqn:slp_rep_freq_arc}
respectively~\cite{COLTON:2019,lintner2015generalized}.

\subsection{$O(1)$-cost Fourier transform at large times}\label{sec:inv_four_tran}
While, as indicated in Remark~\ref{rmk:Bkslow}, $A_k^{\mathrm{slow}}$
is a slowly oscillatory function of $\omega$, the numerical evaluation
of the integral in~\eqref{eqn:ukw} still requires the use of
increasingly finer discretization meshes as $t$ grows, on account of
the fast oscillations exhibited by the exponential term
$e^{-\mathrm{i}\omega (t - s_k)}$, as function of $\omega$ for large
$t$. As a result, the evaluation of the quantity $u^{I}_k(\br,t)$ by
means of classical quadrature rules requires increasingly fine
frequency meshes, and, thus, increasing numbers of expensive
Helmholtz-equation solutions, as $t$ increases. The FTH
algorithm~\cite{abl} addresses this challenge by employing a
$O(1)$-cost high-frequency quadrature rule for the integrals
in~\eqref{eqn:ukw}, which is reviewed in what follows.

The quadrature rule in~\cite{abl} relies on a truncated Fourier
expansion of the function $U_k^\mathrm{slow}(\br,\omega)$ in
equation~\eqref{eqn:ukw} (cf. equation~\eqref{eqn:F_trig_approx}
below) for $\omega \in I$~\eqref{eq:fr_int}. This approach is
effective because (i)~$U_k^\mathrm{slow}$ is a smooth function of
$\omega$ for $\omega\in I$ (see Remark~\ref{rmk:zer_freq}); and,
(ii)~$U_k^\mathrm{slow}$ can be closely approximated in the interval
$I$ by a smooth and periodic function of period $W_2-W_1$---as it
follows from~\eqref{eq:band-limit-U} and similar relations on the
derivatives of $U_k^\mathrm{slow}$. The high-frequency rule is
presented below in the general setting of
equation~\eqref{eqn:gen_four_int}. In the context of this paper it is
important to note that, for trapping obstacles, nearly-real complex
resonances emerge (see Section~\ref{sec:complex_resonances}), which
leads to very slow Fourier-series convergence . A strategy for
overcoming this difficulty, which is a central contribution of this
paper, is provided in Section~\ref{sec:RS}.

\begin{remark}\label{rmk:zer_freq} 
  As is well known, the solutions to the 2D Helmholtz equation as
  functions of frequency $\omega$ exhibit logarithmic singularities at
  $\omega = 0$~\cite{maccamy1965low, werner1986low} whenever the
  incident field as a function of $\omega$ (equal to the Fourier
  transform of the given temporal excitation) does not vanish in a
  neighborhood of $\omega = 0$. (Such singular behavior does not occur
  in the 3D case.) Within the framework of the 2D FTH method, temporal
  excitations with such nontrivial zero-frequency content give rise to
  frequency-domain functions $F(\omega)$ in~\eqref{eqn:gen_four_int}
  containing logarithmic singularities at $\omega = 0$. Such
  singularities, in turn, lead to slow temporal decay of the Fourier
  transform $\mathcal{I}^{[a,b]}[F](t)$, and, thus, of each solution
  $u_k$~\eqref{wind_uk}. Nonetheless, the FTH method remains valid in
  the presence of non-vanishing zero-frequency content~\cite{abl}, and
  the associated slow decay can be effectively addressed using
  suitable asymptotic expansions~\cite{ablII}. Consequently, the
  techniques proposed in this paper are extensible to such cases. For
  simplicity and definiteness, however, this paper restricts attention
  to incident excitations whose Fourier transforms vanish in a
  neighborhood of $\omega = 0$.
\end{remark}

To present the FTH integration method, we consider integrals of the form
\begin{equation}\label{eqn:gen_four_int}
    \mathcal{I}^{[a,b]}[F](t) = \int_{a}^b F(\omega)e^{-\mathrm{i}\omega t}d\omega,
\end{equation}
where $F(\omega)$ is a smooth and periodic function of period $(b-a)$
(cf. points~(i) and~(ii) above). In order to efficiently evaluate the
integral~\eqref{eqn:gen_four_int} for arbitrarily large values of $t$
we first re-express that integral in the form
\begin{equation}\label{eqn:FT_int_no_zf}
    \mathcal{I}^{[a,b]}[F](t) = e^{-\mathrm{i}\delta t}\int_{-W}^W F(\delta + \omega)e^{-\mathrm{i}\omega t}d\omega \quad \text{where} \quad W = \frac{b - a}{2} \quad \text{and} \quad \delta = \frac{b + a}{2}.
\end{equation}
The function $F(\delta + \omega)$ is then approximated by a trigonometric polynomial of the form
\begin{equation}\label{eqn:F_trig_approx}
    F(\delta + \omega) \approx \sum_{m = -M/2}^{M/2 - 1} c_me^{\mathrm{i}\frac{2\pi}{P}m\omega},
  \end{equation}
where $P = 2W$. To complete the quadrature rule we write $\alpha = \frac{P}{2\pi}$,  substitute~\eqref{eqn:F_trig_approx} into~\eqref{eqn:FT_int_no_zf} and integrate termwise---which results in the highly accurate approximation
\begin{equation}\label{NEFFT}
  \mathcal{I}^{[a,b]}[F](t) \approx e^{-\mathrm{i}\delta t} \sum_{m = -M/2}^{M/2 - 1} c_m \frac{P}{\pi(\alpha t - m)}\sin\left(\pi(\alpha t - m)\right)
\end{equation}
---which may be evaluated at fixed cost for arbitrarily large values
of $t$. As noted in~\cite{abl}, finally, the expression~\eqref{NEFFT}
may be produced over prescribed equispaced sets of times $t$ at FFT
speeds by employing the fractional Fourier transform.

\subsection{Complex resonances}\label{sec:complex_resonances}
Let $\mathcal{U}=\mathcal{U}_\omega$ denote the solution operator to
the Helmholtz problem for either open or closed curves at a given
frequency $\omega$: for a given ``boundary-values'' function
$B = B(\mathbf{r})$ in the space $H$ equal to
$H^{1/2}(\Gamma)$ for closed curves and equal to
$\widetilde{H}^{-1/2}(\Gamma)$ for open curves, we have
\begin{equation}
  \label{eq:sol_op}
  \mathcal{U}[B] = U,\quad \mbox{where}\quad U\quad \mbox{solves the problem~\eqref{eqn:helmholtz}}.
\end{equation}
(For detailed definitions of the spaces $H^{1/2}(\Gamma)$ and
$\widetilde{H}^{-1/2}(\Gamma)$ see~\cite{mclean2000strongly}
and~\cite{stephan1984augmented,lintner2015generalized}, respectively.)
Using the notation $B=B_\omega$ to explicitly display the
$\omega$-dependence of the given boundary data we also write e.g.
\begin{equation}
  \label{eq:abuse}
  \mathcal{U}=\mathcal{U}_\omega\quad
  U= \mathcal{U}_\omega\big[B\big]=\mathcal{U}_\omega\big[B_\omega\big]\quad\mbox{and}\quad U(\br,\omega) = \mathcal{U_\omega}\big[B_\omega\big](\br)
\end{equation}
for the solution operator, the solution $U$ and the values of the
solution $U$ for given $\omega$ and $\br$, as needed.

As shown in~\cite{taylorbook} and~\cite{bst} in the closed-curve and
open-arc contexts, respectively, the operator $\mathcal{U}_\omega$,
which is defined for all real values of $\omega$, admits a meromorphic
continuation into the complex $\omega$ plane, with all poles of
$\mathcal{U}_\omega$ contained in the lower half-plane. (Henceforth, we refer to the
poles of $\mathcal{U}_\omega$ as complex resonances.) In the 2D case
a branch cut must also be introduced in order to account for a
logarithmic singularity at $\omega = 0$. The analytic
continuation is performed by expressing $\mathcal{U}_\omega$ in terms
of integral operators.  In the open-arc case, for example, denoting by
$\mathcal{U}^\mathrm{o}_\omega$ the solution operator for open arcs
we may write~\eqref{eqn:slp_rep_freq}
\begin{equation}\label{eqn:sol_op_oa}
    \mathcal{U}^\mathrm{o}_\omega[B] = \mathcal{S}_\omega[(S_\omega)\inv B].
\end{equation}
Letting $\mathcal{U}^\mathrm{c}_\omega$ denote the operator for closed
curves, in turn, we  have~\eqref{eqn:combined_rep}
\begin{equation}\label{eqn:sol_op_comb}
  \mathcal{U}^\mathrm{c}_\omega[B] =  \mathcal{K}_\omega[(C_{\omega,\eta})\inv B](\omega) - \mathrm{i}\eta \mathcal{S}_\omega[(C_{\omega,\eta})\inv B].
\end{equation}

The $\omega$-dependent boundary integral operators utilized to derive
the analytic continuation of $\mathcal{U}(\omega)$ also enable the
numerical computation of complex resonances for both closed
curves~\cite{steinbach2017combined} and open arcs~\cite[Sec.
2]{bst}. In detail, in view of~\eqref{eqn:sol_op_oa}, for open curves
the poles of the inverse Single Layer Potential $(S_\omega)^{-1}$ in
the lower half-plane correspond to the complex resonances of the
solution operator, and, thus the complex resonances may be
approximated numerically as the poles of the discrete operator
$(S_\omega^N)^{-1}$ in~\eqref{eq:discr_opers}. In the case of closed
curves, however, special care is required in selecting the
combined-field coupling parameter $\eta$ in~\eqref{eqn:combined_rep}
(cf.~\cite{taylorbook}, where, for a different purpose, a choice is
made that differs from the one introduced below which is not suitable
for our setting).  Indeed, as established
in~\cite{steinbach2017combined} for a certain combined field operator
$\widetilde{C}_{\omega,\eta}$ associated with the Neumann problem,
choosing a coupling parameter $\eta > 0$ causes the inverse
$(\widetilde{C}_{\omega,\eta})^{-1}$ to have poles in the lower
half-plane that do not correspond to complex resonances. However, for
$\eta < 0$, these issues do not arise, and the poles of
$(\widetilde{C}_{\omega,\eta})^{-1}$ in the lower half-plane exactly
match the Neumann complex resonances. Appendix~\ref{app:combined_form}
presents a corresponding discussion concerning the operator
$C_{\omega,\eta}$ (eq.~\eqref{eqn:combined_op}) associated with the
Dirichlet problem, showing that the poles of $(C_{\omega,\eta})^{-1}$
with $\eta < 0$ exactly match the complex resonances of
$\mathcal{U}(\omega)$.
 
The most physically relevant complex resonances are those near the
real axis, as they induce near-singular behavior in the integrand of
the inverse Fourier transform integral~\eqref{eqn:ukw}, making
numerical evaluation challenging. Section~\ref{sec:RS} presents an
approach that overcomes these difficulties through a singularity
subtraction technique, and it establishes a key connection of that
method with the long-time asymptotics of time-domain scattering
solutions. To enable this approach, Section~\ref{sec:AAA_and_res_alg}
introduces a novel algorithm for computing the required
``incidence-excited'' complex resonances for the Dirichlet
problem. Notably, this algorithm relies exclusively on evaluations of
$H_{\omega}^{-1}$~\eqref{eq:C_or_S} at real frequencies $\omega$.

\begin{remark}\label{rmk:simple_reso}
  \label{simple} For simplicity, this paper focuses on the generic
  case~\cite{albert1975genericity, klopp1995generic}, where complex
  resonances are simple poles of $\mathcal{U}_\omega$.
\end{remark}

\section{Incidence-excited resonances from real-frequency data}\label{sec:AAA_and_res_alg}
Well-known
methods~\cite{Beyn2012eigen,asakura2009numerical,baum2005singularity}
for the evaluation of complex resonances within a contour
$\mathcal{C}\subset \mathbb{C}$ require inversion of boundary integral
operators along $\mathcal{C}$. However, most relevant to this work are
the complex resonances which both lie near the real axis and are
excited by a given incident field. This section proposes a
Incidence-Excitation (IE) method that, relying on {\em a new real-axis
  adaptive} rational approximation strategy, obtains the relevant
incidence-excited complex resonances on the basis of inversion of the
boundary integral operators {\em at real frequencies only}.  The
excited frequencies thus obtained can then be utilized in a seamless
manner in conjunction with the FTH method to produce the solution of a
given time-domain problem.

The groundwork of the method is laid in Section~\ref{sec:AAA},
starting with a brief presentation of the AAA algorithm for rational
approximation~\cite{AAA} and relevant
variants. Section~\ref{sec:adaptiveres} then motivates the proposed
algorithm by reviewing a recently introduced adaptive
algorithm~\cite{bst} for the evaluation of real and complex
resonances. Section~\ref{sec:realleftscal} then details the proposed
IE adaptive algorithm for the evaluation of the excited resonances and
the corresponding residues,---which, as is demonstrated
Section~\ref{sec:RS}, form a basis for the evaluation of time-domain
fields at all times with minimal computational cost.

\subsection{Scalar-valued and random-sketching vector-valued rational approximation}\label{sec:AAA}
The AAA algorithm~\cite{AAA} is an efficient method for constructing
rational approximants $r^m(\omega) \approx f(\omega)$ for a
complex-variable function $f$ on the basis of the values of $f$ at an
$M$-point set $Z \subset \mathbb{C}$. The algorithm proceeds by
inductively constructing, for $m=1,2,\dots$, certain sets of ``support
points'' $Z_{m} = \{\omega_1,\ldots,\omega_{m}\}, Z_{m} \subset{Z}$
($\omega_p\ne \omega_q$ for $p\ne q$),  ``weights''
$v^{m} = \{v^{m}_1,\ldots,v^{m}_m\}\subset \mathbb{C}$, and associated
rational approximants
\begin{equation}\label{eqn:rat_approx}
    r^m(\omega) = \sum_{j=1}^m \frac{v^m_j f(\omega_j)}{\omega - \omega_j} \bigg/   \sum_{j=1}^{m} \frac{v^{m}_j}{\omega - \omega_j}.
  \end{equation}
Letting $\widetilde{Z}_0=\emptyset$, calling
$\widetilde{Z}_m \coloneq Z \setminus Z_{m}$, and using the enumeration
$\widetilde{Z}_m = \{\omega^m_1,\ldots,\omega^m_{M-m}\}$,
the inductive process proceeds by first constructing the $m$-th support point
\begin{equation*}\label{eqn:support_max}
 \omega_{m} =  \argmax_{\omega\in \widetilde{Z}_{m-1}} |r^{m-1}(\omega) - f(\omega)|,
\end{equation*}
and setting $Z_{m} = Z_{m-1} \cup\{\omega_{m}\}$. The inductive step
is then completed by obtaing the $m$-th ``weight vector'' $v^{m}$ as
the solution of the least squares problem
\begin{equation}\label{eqn:A_min}
    v^{m} =  \argmin_{v\in \mathbb{C}^m,\; \parallel v\parallel_m = 1}\parallel A_{m}[f]v\parallel_{M-m},
\end{equation}
where, for a positive integer $n$, $\parallel\cdot\parallel_n$ denotes the Euclidean norm in $\mathbb{C}^n$ and $A_{m}[f]$ denotes the $(M - m) \times m$ Loewner matrix, whose $(ij)$-th entry is given by
\begin{equation*}
  (A_{m}[f])_{ij} = \frac{f(\omega^m_i) - f(\omega_j)}{\omega^m_i - \omega_j}.
\end{equation*}
(The minimization problem~\eqref{eqn:A_min} amounts to minimizing $f$
times the denominator minus the numerator in the barycentric
formulation~\eqref{eqn:rat_approx} of the rational approximant
(see~\cite[Eq.~3.4]{AAA}).)  Using the support points in the set $Z_m$
and the weight vector $v^m$, the rational approximant $r^m$ is
constructed. The algorithm terminates when
$\max_{Z}|r^{m}(\omega) - f(\omega)|$ is less than a user-specified
error tolerance $\varepsilon_\mathrm{tol}$. In practice, the
algorithm is also terminated if $m$ exceeds a suitably chosen value
$m_\mathrm{max}$, typically set to $100$. Meeting this stopping
criterion before the user-prescribed tolerance
$\varepsilon_\mathrm{tol}$ is achieved provides an indication that $Z$
does not adequately represent $f$ or that $f$ has too many poles near
the points in $Z$. In either case the sample set should be adequately
refined as is done e.g. as part of the adaptive algorithm described in
Section~\ref{sec:adaptiveres}.

A vector AAA algorithm, which produces a vector-valued rational
approximant to a complex vector-valued function $\widehat{f}:\C \to \C^N$, was
introduced in~\cite{lietaert2022automatic}. As in the scalar case, the
vector algorithm constructs a set of support points $Z_m$ and a
vector $v^m$ of weights to produce a vector valued
rational approximant $R^m(\omega) \approx \widehat{f}(\omega)$. Analogous to
the scalar case, the vector-valued rational approximant $R^m(\omega)$ is
constructed inductively in the form of the  quotient
\begin{equation}\label{eqn:rat_approx_vec}
    R^m(\omega) = \sum_{j=1}^m \frac{v^m_j \widehat{f}(\omega_j)}{\omega - \omega_j} \bigg/   \sum_{j=1}^{m} \frac{v^m_j}{\omega - \omega_j}.
\end{equation}
 Letting $\widehat{f}_n(\omega)$, $R^m_n(\omega)$, $(1 \le n \le N)$ denote the components of $\widehat{f}$ and $R^m$ respectively, the support points $\omega_m$ at each step are computed as
\begin{equation*}
   \omega_m = \argmax_{\omega\in \widetilde{Z}_{m-1},\, 1 \le n \le N} |R_n^{m-1}(\omega) - \widehat{f}_n(\omega)|
\end{equation*} 
To compute the weights a problem of the same form as~\eqref{eqn:A_min}
is solved, namely
\begin{equation}\label{eqn:sv_ls}
     v^{m} =  \argmin_{v\in \mathbb{C}^m,\; \parallel v\parallel_m = 1}\parallel B_m[\widehat{f}]v\parallel_{N(M-m)},
\end{equation}
where $B_m[\widehat{f}]$, is a block matrix given by $B_m[\widehat{f}]= [A_m[\widehat{f}_1] \dots A_m[\widehat{f}_N]]^T$. 

The least-squares problem in the vector-valued AAA
algorithm~\eqref{eqn:sv_ls} becomes prohibitively expensive as $N$
grows, even if an efficient implementation such as suggested
in~\cite{lietaert2022automatic} is used. To address this issue a
random sketching approach was recently proposed
in~\cite{guttel2024randomized}, where for a random matrix
$V \in \C^{N\times \ell}$ with $\ell \ll N$, the vector-valued AAA
algorithm is applied to the function $\widehat{g}(\omega) = V^T\widehat{f}(\omega)$ to compute the rational approximant
\begin{equation}\label{eqn:Gell}
    R^m(\omega) = \sum_{j=1}^m \frac{v^m_j \widehat{g}(\omega_j)}{\omega - \omega_j} \bigg/   \sum_{j=1}^{m} \frac{v^m_j}{\omega - \omega_j}.
  \end{equation}
A rational approximant for $\widehat{f}$ is then constructed by substituting the values $\widehat{g}_\ell(\omega_j)$ by $\widehat{f}(\omega_j)$ in equation~\eqref{eqn:Gell}. The numerical experiments and
theoretical analysis in~\cite{guttel2024randomized} show that even for
small values of $\ell$ (e.g. $\ell = 2,4,8)$ the random sketching
rational approximation algorithm produces a highly accurate rational
approximant.

The numerical implementations presented in this paper utilize the AAA
and vector-AAA algorithms provided in~\cite{chebfun}
and~\cite{lietaert2022automatic}, respectively. For the vector-AAA algorithm we have additionally implemented a spurious pole ``clean up'' algorithm based on the prescription given in~\cite{AAA}.
\begin{remark}\label{AAA_notat_usage}
  The notations $r^m$, $R^m$ and $v^m_j$ used in this section for the
  AAA approximants and weights is useful due to the inductive nature
  of the AAA construction. However, this notation is not necessary for
  the practical use of the approximants. Therefore, in the remainder
  of the paper, we will omit it and instead express the AAA
  approximant produced by the algorithm without explicitly
  incorporating the numerator/denominator degree $m$ (except in the
  upper summation limit):
  \begin{equation}\label{eqn:rat_approx_vec_no-exp}
    R(\omega) = \sum_{j=1}^m \frac{v_j \widehat{f}(\omega_j)}{\omega - \omega_j} \bigg/   \sum_{j=1}^{m} \frac{v_j}{\omega - \omega_j}.
\end{equation}
\end{remark}

\subsection{Adaptive random-excitation (RE) resonance evaluation}\label{sec:adaptiveres}
To motivate the adaptive algorithm for evaluation of incidence-excited
resonances presented in Section~\ref{sec:realleftscal}, this section
briefly reviews the adaptive algorithm~\cite[Algorithm 2]{bst} for
evaluation of all resonances in a given domain in the complex plane,
which is based on the use of random excitations (RE). The adaptive RE
algorithm has proven effective for evaluation of complex resonances,
even in presence of large numbers of resonances and/or high
frequencies.

To evaluate the resonances associated with the frequency-domain
solution operator
$\mathcal{U}_\omega$~\eqref{eq:sol_op}--\eqref{eq:abuse} (which per
Section~\ref{sec:complex_resonances}, coincide with the poles of the
resolvent $(H_{\omega})^{-1}$), the RE method computes the poles of
the corresponding matrix-valued numerical approximation
$\widetilde{H}_\omega^{-1}$~\eqref{eq:C_or_S-num};
see~\cite[Rem. 2]{bst}. The poles of $\widetilde{H}_\omega^{-1}$, in
turn, are produced by seeking poles of the randomly scalarized
resolvent
\begin{equation*}
    s(\omega) = u^*\widetilde{H}_\omega^{-1}v \quad \text{where} \quad {u,v \in \C^N} \quad \text{are fixed random vectors}
\end{equation*}
---since, as shown in~\cite{bst}, the poles $s(\omega)$ coincide, with
probability $1$, with the poles of $\widetilde{H}_\omega^{-1}$.

To compute the poles of $s(\omega)$---that is, the numerical
approximations of the resonances of $\mathcal{U}_\omega$---lying
within a set $D\subset \mathbb{C}$, including its boundary, the RE
method applies the AAA algorithm to construct a rational approximant
$r(\omega)\approx s(\omega)$ from samples along the boundary of
$D$. This approximant is useful in that, provided the set $D$ is
``sufficiently small'' and an adequate number of roughly equispaced
sampling points are used along the boundary of $D$~\cite[Rem. 4]{bst},
the poles of $r(\omega)$ within $D$ provide close approximations of
the poles of $s(\omega)$. In order to tackle the generic case in which
a proposed set $D$ may not be sufficiently small, the RE method
employs an adaptive search technique by partitioning $D$ into
sub-regions and computing the poles within each subregion, typically
using rectangular domains $D$ which are subsequently dyadically
partitioned, in an iterative fashion, into smaller rectangular
subregions---as detailed in~\cite{bst}. The algorithm terminates when
no new poles are found in each subregion, upon which a certain secant
method-based termination stage is used to significantly enhance
accuracy and to filter out spurious poles that may (rarely) be
produced by the AAA algorithm.

\subsection{Incidence-excitation  (IE) resonance evaluation for time-domain problems}\label{sec:realleftscal}

The complex resonances most relevant to the FTH method reviewed in
Section~\ref{sec:abl_rev} are those whose real parts lie within the
incident-field interval $I = I(W_1,W_2)$~\eqref{eq:fr_int}, which are
located near the real axis, and whose residue is not numerically
insignificant. Indeed, such complex resonances lead to sharp spikes in
$U_k^\mathrm{slow}(\br,\omega)$ along the real frequency axis, as
illustrated by the numerical experiments in
Section~\ref{sec:SingSubEx}, and, therefore, the accurate evaluation
of the Fourier transform~\eqref{eqn:ukw} via the FTH high-frequency
integration method often requires the use of extremely fine meshes. In
order to avoid this difficulty a certain complex resonance singularity
subtraction technique is proposed in Section~\ref{sec:RS}, which, in
particular requires as input the positions and residues of all
relevant near-real complex resonances. This section presents the
``Incident Excitation'' (IE) algorithm which, in contrast with the RE
algorithm presented in the previous section (which obtains complex
resonances $\omega$ within a given region in the complex plane on the
basis of the scalarization of the resolvent
$\widetilde{H}_\omega^{-1}$ via pre- and post-multiplication by a pair
of {\em random} vectors at a number of frequencies $\omega$ in the
{\em complex} plane) aims to compute all complex resonances
responsible for the spikes in $U_k^\mathrm{slow}(\br,\omega)$ on the
sole basis of the action of the discrete version
$\widetilde{H}_\omega^{-1}$ of the resolvent $H_\omega^{-1}$ on {\em
  incident field data}~\eqref{eq:BP}
(eqs. \eqref{eq:C_or_S-num}-\eqref{disc_vrsn}) at {\em real}
frequencies $\omega$---where either $H_\omega = C_{\omega,\eta}$ or
$H_\omega = S_\omega$ (equations~\eqref{eqn:combined_bie}
and~\eqref{eqn:open_bie-sqrt}), as applicable.

\begin{remark}\label{rmk:IEmethod}
  Unlike other methods for evaluating complex resonances—such as the
  RE method reviewed in Section~\ref{sec:adaptiveres}, the contour
  integration
  methods~\cite{Beyn2012eigen,misawa2017boundary,asakura2009numerical},
  and the root-finding methods~\cite{guttel2017nonlinear}---the IE
  approach introduced in this section identifies near-real complex
  resonances as the frequency poles of the integral density
  $\psi_\bp(\br',\omega) = \left(H_\omega^{-1}B_\bp(\cdot,\omega)\right)(\br',\omega)$ associated with the incident field
  $B_\bp$. A byproduct of this procedure is the construction of
  rational approximants~\eqref{eqn:Rat_approx_set} for the density
  functions $\psi_\bp(\br',\omega)$ themselves. These approximants can
  then be reused to inexpensively obtain, without further resolvent
  evaluations: i) The residues of the field at a spatial point $\br$
  (as is done in Section~\ref{sec:res_comp}), which are required for
  the subtraction procedure and for the large-time evaluation of polar
  contributions (see Section~\ref{sec:RS}); and ii) The density values
  at arbitrary sets of frequencies, such as equispaced frequency sets
  which in Section~~\ref{sec:Num_Imp} provide the regularized
  frequency-domain data needed in the high-frequency integration step
  of the FTH method. This reuse of rational approximants is a key
  element of the overall FTH-SS methodology, as it substantially
  reduces the number of costly resolvent evaluations required by the
  algorithm.
\end{remark} 

To introduce the IE algorithm we first observe that, in view
of~\eqref{eq:psi_p_sol},~\eqref{eq:Up} and~\eqref{eq:ukslow} together
with \eqref{eqn:combined_rep} or~\eqref{eqn:slp_rep_freq_arc}, as
applicable, the complex resonances of
$U_k^{\mathrm{slow}}(\br,\omega)$ coincide with the complex poles of
$U_\bp(\br,\omega)$, and thus, with the complex poles of
$\psi_\bp(\br',\omega) =
\left(H_\omega^{-1}B_\bp(\cdot,\omega)\right)(\br,\omega)$ (which
therefore are, in particular, independent of both $\br$ and $k$). The
IE algorithm thus seeks to evaluate all complex poles of the
incidence-excited resolvent (IE)
\begin{equation}
  \label{eq:inc_exc}
  \left(H_\omega^{-1}B_\bp(\cdot,\omega)\right)(\br,\omega)
\end{equation}
in the box
\begin{equation}
  \label{eq:box}
  \mathcal{M}^I_h \coloneqq \{\omega \in \C \ | \ \Re(\omega) \in I,
  \Im(\omega) \in [-h,0]\}\quad\mbox{for some prescribed parameter}\quad h > 0,
\end{equation}
and associated residues. Using the enumerations
\begin{equation}
  \label{eq:nih}
  \sigma_n,\quad 1\leq n\leq N_h^{I},\qquad\mbox{and}\qquad \rho_n,\quad
  1 \le n \le N_h^{I,\mathrm{e}},
\end{equation}
which list all poles $\sigma_n$ of the Helmholtz solution
operator~\eqref{eq:sol_op} contained in the box $\mathcal{M}^I_h$, and
all poles $\rho_n$ in the same box obtained through the
Incidence–Excitation algorithm introduced in Section~\ref{adaptive},
we denote the corresponding sets of resonances by
\begin{equation}\label{eqn:inc_res}
  P^I_h = \{\sigma_1,\dots,\sigma_{N_h^{I}}\}\subset \mathcal{M}^I_h \qquad \text{and}\qquad P_{h}^{I,\mathrm{e}} = \{\rho_1,\dots,\rho_{N_h^{I,\mathrm{e}}}\}\subset \mathcal{M}^I_h.
\end{equation}
In view of~\eqref{eq:Up}, the corresponding residues, which are
denoted $c_{\bp,n}(\br)$ ($1\leq n\leq N_h^{I}$) and
$d_{\bp,n}(\br)$ ($1\leq n\leq N_h^{I,\mathrm{e}}$), respectively, are given by
\begin{equation}\label{eqn:res}
   c_{\bp,n}(\br) = \frac{1}{2\pi i}\int_{C_n} U_\bp(\br,\omega)d\omega,\quad 1\leq n\leq N_h^{I} \quad\text{and}\quad d_{\bp,n}(\br) = \frac{1}{2\pi i}\int_{C_n^\mathrm{e}} \widetilde{U}_\bp(\br,\omega)d\omega\quad 1\leq n\leq N_h^{I,\mathrm{e}},
\end{equation}
where $C_n$ (resp. $C_n^\mathrm{e}$) denotes a contour enclosing
$\sigma_n$ (resp.  $\rho_n$) but no other poles $\sigma_j$, $j\ne n$
(resp. no other poles $\rho_j$, $j\ne n$). For readability the
dependence of $c_{\bp,n}$ and $d_{\bp,n}$ on $I$ and $h$ is not made
explicit in the notation used.
\begin{remark}\label{rmk:phIe_is_numer}
  The RE algorithm described in Section~\ref{sec:adaptiveres} produces
  a numerical approximation of the set $P^I_h$ of complex resonances
  contained in $\mathcal{M}^I_h$. As noted in that section, owing in
  part to its adaptive character, that method generically captures all
  the singularities within $\mathcal{M}^I_h$ subject to the error
  tolerance associated with the numerical approximations used.  In
  contrast to $P^I_h$, the set $P_{h}^{I,\mathrm{e}}$ of
  incidence-excited poles is a numerical construct that is not defined
  independently of the algorithm used to compute it. Although this
  dependence is not explicitly reflected in the notation,
  $P_{h}^{I,\mathrm{e}}$ is determined by the choice of algorithm, the
  prescribed error tolerance, and the specified incident field. The
  algorithm proposed in this paper for producing the set
  $P_{h}^{I,\mathrm{e}}$, which also utilizes the AAA method
  adaptively, is presented in Section~\ref{adaptive}. As discussed
  below, the contribution to the field $U(\br,\omega)$ by a pole
  $\sigma_n$ is a quantity of the order of the quotient of the
  ($\mathbf{r}'$-dependent) residue of the integral equation
  density~\eqref{dens_resid} and the distance of the pole to the real
  frequency axis. Our numerical experiments indicate that the set
  $P_{h}^{I,\mathrm{e}}$ produced by the proposed IE algorithm
  generally coincides with the subset of $P^I_h$ for which the norm
  (in $L^2(\Gamma)$) of the residue of the integral equation density
  \emph{resulting from the given incident field} exceeds a value of
  the order of the associated AAA tolerance
  $\varepsilon_\mathrm{tol}$.  Thus, the proposed IE algorithm (which,
  in particular, incorporates a version of the AAA algorithm that
  includes the ``cleanup'' procedure described in~\cite[Sec. 5]{AAA}),
  ``disregards'' poles with negligible residues. As indicated
  in~\cite{AAA}, the cleanup procedure generically discards poles
  whose associated residues fall below the numerical tolerance
  $\varepsilon_\mathrm{tol}$ that is also used as part of the
  termination criterion for the AAA algorithm
  (cf. Section~\ref{sec:AAA}).
  \end{remark}

 It is important to note that, in view of~\eqref{eq:psi_p_sol}
and~\eqref{eq:Up}, the residues $c_{\bp,n}(\br)$ may be expressed in
terms of the density residues
\begin{equation}\label{dens_resid}
    \widehat{c}_{\bp,n}(\br') = \frac{1}{2\pi \mathrm{i}}\int_{C_n} \psi_\bp(\br',\omega) d\omega, \quad \br' \in \Gamma.
\end{equation}
Indeed, on account of Remark~\ref{rmk:simple_reso} we have
$\widehat{c}_{\bp,n}(\br) = \lim_{\omega \to \sigma_n} (\omega - \sigma_n)
\psi_\bp(\br,\omega)$. Thus, using the single-layer and combined-field
field representations $\mathcal{S}_\omega$ and
$\mathcal{C}_{\omega, \eta}$ (equations~\eqref{eqn:slp_rep_freq_arc}
and~\eqref{eqn:combined_rep}, respectively) with $\omega = \sigma_n$,
together with the dominated convergence theorem, we obtain
% The dominated convergence comes because of the maximum principle. The function g(\omega) = (\omega - \rho_n) \psi_\bp(\br,\omega) is analytic within the contour C so it takes a max on C.
\begin{equation}\label{res_alt}
  c_{\bp,n}(\br) = \mathcal{S}^\mathrm{arc}_{\sigma_n}[\widehat{c}_{\bp,n}](\br) \quad \mbox{or} \quad c_{\bp,n}(\br) = \mathcal{C}_{\sigma_n,\eta}[\widehat{c}_{\bp,n}](\br) \quad \br \in \Omega^\mathrm{ext},\quad\mbox{as applicable}
\end{equation}
in the open and closed curve cases, respectively.  Utilizing the
Cauchy–Schwarz inequality we then obtain
\begin{equation}\label{eqn:spat_res_dens_bound}
  |c_{\bp,n}(\br)| \le M_\mathcal{S}(\br) \|\widehat{c}_{\bp, n}\|_{L^2(\Gamma)} \quad \mbox{and} \quad |c_{\bp,n}(\br)| \le  M_\mathcal{C}(\br) \|\widehat{c}_{\bp, n}\|_{L^2(\Gamma)}
\end{equation}
where, letting $\|\cdot\|_{L^2(\Gamma)}$ denote the $L^2$ norm on the
curve $\Gamma$, we have set
\[
  M_\mathcal{S}(\br) =
  \|G_{\sigma_n}(\br,\cdot)\|_{L^2(\Gamma)}\quad\mbox{and}\quad
  M_\mathcal{C}(\br) =\| \frac{\partial G_{\sigma_n}(\br,\cdot)}{\partial
    n} - \mathrm{i} \eta G_{\sigma_n}(\br,\cdot)\|_{L^2(\Gamma)}.
\]
This tells us that, for $\omega \in I$, the contributions
$c_{\bp, n}/(\omega - \sigma_n)$ to the field $U(\br,\omega)$ which
result from a given $\br$-dependent resonance
$(\sigma_n, c_{\bp,n}(\br))$ are negligibly small---of the order
$\varepsilon_\mathrm{tol}$,
cf. Remark~\ref{rmk:phIe_is_numer}---provided that the corresponding
$\br$-independent resonance $(\sigma_n, \widehat{c}_{\bp,n})$ itself
produces $\psi_\bp(\br,\omega)$ contributions
$\widehat{c}_{\bp,n}/(\omega - \sigma_n)$ of order
$\varepsilon_\mathrm{tol}$ for all $\omega \in I$.  Briefly, then, the
complex resonances $\sigma_n$ which cause spikes in
$U_k^\mathrm{slow}(\br,\omega)$ are the complex resonances near the
real axis whose density residue $\widehat{c}_{\bp,n}(\br')$ is not
negligible for $\br' \in \Gamma$. In order to produce all such
resonances the proposed IE method resorts to computing a rational
approximant to $\psi_\bp(\br',\omega)$ by means of the random
sketching algorithm described in Section~\ref{sec:AAA}. In practice,
however, a direct application of the random sketching
rational approximation approach to the
entire interval $I$ generally fails to approximate all relevant
resonance pairs $(\sigma_n, \widehat{c}_{\bp,n})$. In order to tackle
this difficulty an adaptive approach (analogous to but different from
the one utilized by RE method) is proposed in Section~\ref{adaptive}
for the evaluation of the relevant pole locations, producing the set
$P^{I,\mathrm{e}}_h$.  Section~\ref{sec:res_comp} then describes a
method to compute the corresponding residues
$d_{\bp, n}$~\eqref{eqn:res} by re-using the rational approximants to
$\psi_\bp(\br,\omega)$ that are generated as part of by the adaptive
IE-based pole-search algorithm.

\vspace{0.2cm}
\begin{algorithm}[H]
    \SetKwFunction{RealLineAdaptive}{reallineadaptive}
    \SetKwProg{Fn}{Function}{}{end}
    \DontPrintSemicolon 
    \KwIn{An interval $I = [W_1,W_2]$, a number $J$ of frequencies to use in each interval, AAA stopping parameters $m_\mathrm{max}$ and $\varepsilon_\mathrm{tol}$, and singularity-box depth $h>0$.}
    \Fn{\RealLineAdaptive{$I$, $J$}}{
        Evaluate $\widetilde{\psi}_p$~\eqref{eq:F-psi} at $J$ equally spaced frequencies in the interval $I$.\\
        Compute a rational approximant $R(\omega)$ of the incidence-excited resolvent~\eqref{eq:tilde_inc_ex} evaluated by the random sketching
rational approximation  algorithm applied to the values $\widetilde{\psi}_p$ computed at the previous step.\\
        \eIf{\text{the random sketching
rational approximation error converged within} $\varepsilon_\mathrm{tol}$}{
        Set $R^{I}(\omega) = R(\omega)$ and compute $P_{h}^{I,\mathrm{e}}$.\\
        \Return $R^{I}(\omega), P_{h}^{I,\mathrm{e}}$.\\
        }{
        Compute the midpoint $W_3 = (W_2 + W_1) /2$, and set $I_\mathrm{left} = [W_1,W_3]$ and $I_\mathrm{right} = [W_3,W_2]$.\\
        \Return \RealLineAdaptive{$I_\mathrm{left}$, $J$} and \RealLineAdaptive{$I_\mathrm{right}$, $J$}
        }
    }
\caption{Adaptive Incidence-Excitation Resonance Evaluation} \label{alg:real-line-adaptive}
\end{algorithm}
\vspace{0.2cm}

\subsubsection{Adaptive IE resonance evaluation\label{adaptive}}
For a given incident field whose frequency content vanishes outside an
interval $I = I(W_1,W_2)$ (so that, e.g., in the context set up in
Section~\ref{sec:abl_rev}, $B(\omega)$ either vanishes or is
negligible for $\omega\not\in I$) the adaptive IE method computes the
set $P_{h}^{I,\mathrm{e}}$~\eqref{eqn:inc_res} of IE complex
resonances in the box~\eqref{eq:box}, for some prescribed parameter
$h > 0$ (that may be selected as discussed below in this section). The
IE method accomplishes this on the basis of the discrete version
\begin{equation}
  \label{eq:F-psi}
  \widehat{f}(\omega) = (\widehat{f}_1(\omega),\dots,\widehat{f}_N(\omega)) =
  (\widetilde\psi_{\bp,1}(\omega),\dots,\widetilde\psi_{\bp,N}(\omega)) = \widetilde\psi_{\bp}(\omega)
\end{equation}
(equation~\eqref{disc_vrsn}) of the density solutions
$\psi_{\mathrm{\bp}}(\cdot,\omega) =
\left(H_{\omega}\right)^{-1}B_\bp(\cdot,\omega)$, with $\omega$ in a
discrete set of (generally non-equispaced) adaptively-selected
frequencies within the interval $I$, as described in what follows. (The
set $P_h^{I,\mathrm{e}}$ then serves as the input to the singularity subtraction
algorithm introduced in Section~\ref{sec:RS}.)

In detail, starting with a set
$\mathcal{F} = \{\omega_1, \dots, \omega_J\}$ of $J$ equispaced
frequencies in the interval $I$, the IE algorithm first seeks to
compute a vector-valued rational approximant $R(\omega)= R^m(\omega)$
of the vector
\begin{equation}
\label{eq:tilde_inc_ex}
\mbox{$\widetilde{H}_{\omega}^{-1}\widetilde{B}_\bp(\omega)$, on the basis of its values for $\omega \in \mathcal{F}$,}
\end{equation}
by applying the random sketching
rational approximation method for vector-valued functions of
$\omega$ which is described in Section~\ref{sec:AAA}.  (Per
Remark~\ref{AAA_notat_usage} the superindex $m$ in the notation for
the rational approximant $R = R^m$ is suppressed here and in what
follows.) If the random sketching
rational approximation method converges within the
prescribed error tolerance $\varepsilon_\mathrm{tol}$ for some value
of $m \leq m_\mathrm{max}$, then the algorithm is completed in its
$L=1$ step by setting $L=1$, $I_L = I_1 =I$,
% $\mathcal{F}_L =\mathcal{F}_1 = \mathcal{F}$
$R^{I_L} = R^{I_1} = R$, and by producing the set
$P_{h}^{I,\mathrm{e}}= P_{h}^{I_1,\mathrm{e}}$ that comprises all
poles of $R^{I_1}(\omega)$ contained in
$\mathcal{M}_h^{I_1}$~\eqref{eq:box}. Otherwise the interval $I$ is
divided at the midpoint $W_3 = (W_2 + W_1) /2$ into two subintervals
$I_\mathrm{left} = [W_1, W_3]$ and $I_\mathrm{right} = [W_3, W_2]$,
and the same procedure is recursively applied to the intervals
$I_\mathrm{left}$ and $I_\mathrm{right}$, using a set of $J$
equispaced frequencies in each case. Each time the random sketching
rational approximation
method converges within the error tolerance $\varepsilon_\mathrm{tol}$
for some value of $m \leq m_\mathrm{max}$ and for some subinterval
$I_L\subset I$, the index $L$ is increased by 1, and upon each such
completion, the corresponding interval $I_L$, rational approximant
$R^{I_L} $, and set $P_{h}^{I_L,\mathrm{e}}$ of poles of $R^{I_L} $ in
the box $\mathcal{M}_h^{I_L}$ are recorded (and the poles of
$R^{I_L} $ outside $\mathcal{M}_h^{I_L}$ are discarded). The processes
terminates when the random sketching
rational approximation method has converged, within the
tolerance $\varepsilon_\mathrm{tol}$ and for some value of
$m \leq m_\mathrm{max}$, on every subinterval in a partition
$I_1,\dots,I_L$ of the interval $I$:
\begin{equation}
  \label{eq:partit}
I=\bigcup_{\ell =1}^L I_\ell.
\end{equation}
Thus, upon completion, the IE algorithm results in a partition
$I_1,\dots,I_L$ of $I$, a set of rational approximants
\begin{equation}\label{eqn:Rat_approx_set}
    R^{I_1},\dots, R^{I_L}
\end{equation}
applicable over the
corresponding sets $\mathcal{M}^{I_1}_h,\dots, \mathcal{M}^{I_L}_h$,
respectively, and associated sets of poles
$P_{h}^{I_1,\mathrm{e}}\subset \mathcal{M}^{I_1}_h,\dots,P_{h}^{I_L,\mathrm{e}}\subset
\mathcal{M}^{I_L}_h$. The set of all poles~\eqref{eqn:inc_res} thus obtained, in particular, is given by
\begin{equation}
  \label{eq:all_poles}
  P_{h}^{I,\mathrm{e}} = \bigcup_{\ell=1}^L P_{h}^{I_\ell,\mathrm{e}}\subset \mathcal{M}^{I}_h,\quad\text{where}\quad I = I(W_1,W_2)\quad \text{is given by~\eqref{eq:fr_int}}.
\end{equation}
A pseudocode description of the method is presented in Algorithm~\ref{alg:real-line-adaptive}.

On the basis of a broad set of numerical experiments we have found
that using the parameters $m_\mathrm{max} = 100$, AAA tolerance
$\varepsilon_\mathrm{tol} = 10^{-10}$, with $J \in [200,400]$ in each
relevant interval, Algorithm~\ref{alg:real-line-adaptive} is effective
at capturing all resonances with real part in the interval $I$ that
are relevant in the context of singularity-subtraction method. (For
computational efficiency, values of~\eqref{eq:tilde_inc_ex} computed
on each level of the adaptivity recursion can be stored and used in
subsequent levels.) The parameter $h$ should be selected so as to lead
to rapid convergence of the singularity subtraction-based FTH
high-frequency Fourier transform algorithm introduced in
Section~\ref{sec:RS}. Across a wide range of experiments involving
frequency intervals of the type considered in this paper, it has been
found that selecting $0.2\leq h \leq 0.5$ results in a set
$P_{h}^{I,\mathrm{e}}$ that leads to smooth singularity-subtracted
frequency dependence, and which prevents the inclusion of spurious
poles that lie far from the real axis and which, while increasing the
computational cost, have little effect on the convergence rate of the
FTH integration.
Additionally, in situations for which only a low
number of integral equation inversions are to be used, (say less than
$200$), we have found that it is preferable to use
Algorithm~\ref{alg:real-line-adaptive} without adaptive steps
(resulting in a rational approximant with degree
$ m \leq m_\mathrm{max} = J/2$ (cf.~\cite[Sec. 3]{AAA})), rather than
using a smaller value of $J$ (such as e.g. $J = 20$) and applying the
full adaptive algorithm with a large value of the tolerance
$\varepsilon_\mathrm{tol}$. Section~\ref{sec:adaptivealgdemo}
demonstrates the effectiveness of the overall algorithm, and it
analyzes the significance---or lack thereof---of any resonances not
captured by the method for a given value of $h$.

% This remark is mentioned in Section 6.3
\begin{remark}\label{rmk:Interval_Selection_Alg}
  The application of Algorithm~\ref{alg:real-line-adaptive} can become
  expensive in the presence of a large number of IE resonances with
  real parts in the interval $I$---since, in such cases, many adaptive
  steps may be required. To address this difficulty, a modified
  version of the algorithm could be considered in which the search
  for IE resonances is initiated over a collection of subintervals
  forming a partition of the interval $I$. Since the number of
  relevant complex resonances in any given subinterval of $I$ is not
  known a priori, however, a more general ``sampling'' algorithm has
  been devised in which the interval $I$ is {\em sampled} by means of
  a small number of small subintervals on which
  Algorithm~\ref{alg:real-line-adaptive} can be cheaply applied, and
  thus produce useful estimates of the variation of the density of IE
  resonances with real part in the interval $I$. Using such estimates,
  a partition~\eqref{eq:partit} of adequately varying sizes may be
  produced such that the application of
  Algorithm~\ref{alg:real-line-adaptive} on $I_j$ with
  $m_\mathrm{max} = 100$ produces all the resonances relevant to
  $I_j$. A complete exploration and demonstration of the sampling
  algorithm is beyond the scope of this paper, however, and is left
  for future work.
\end{remark}

\subsubsection{Computation of the residues $d_{\bp,n}$~\eqref{eqn:res}}\label{sec:res_comp}
Once the set of IE complex resonances~\eqref{eq:all_poles} has been
identified, the associated spatially dependent residues
$d_{\bp,n}(\br)$~\eqref{eqn:res} ($\mathbf{r}\in\Omega^\mathrm{ext}$)
can be computed using the residue vector
$\widehat{d}_{\bp,n} = (\widehat{d}_{\bp,n,1},\dots\widehat{d}_{\bp,n,N}) \in \C^N$,
\begin{equation}\label{eqn:num_dens_res}
  \widehat{d}_{\bp,n,i} = \int_{C_n^\mathrm{e}} \widetilde{\psi}_{\bp,i}(\omega)\, d\omega, \quad \text{with} \quad \widetilde{\psi}_{\bp,i}(\omega),\quad 1\leq i\leq N, \quad \text{given by~\eqref{disc_vrsn}},
\end{equation}
of the numerical density $\widetilde{\psi}_{\bp,i}(\omega)$ at the
point $\mathbf{r}_i\in\Gamma$, where $C_n^\mathrm{e}$ is defined as
in~\eqref{eqn:res}.  The spatially dependent residue is then computed
using the discretized operators~\eqref{eqn:disc_prop_ops} via the
relations
\begin{equation*}
 d_{\bp,n}(\br) = \mathcal{S}^{\mathrm{arc},N}_{\rho_n}[\widehat{d}_{\bp,n}](\br) \quad \mbox{or} \quad d_{\bp,n}(\br) = \mathcal{C}_{\rho_n,\eta}^N[\widehat{d}_{\bp,n}](\br) \quad \br \in \Omega^\mathrm{ext},\quad\mbox{as applicable}
\end{equation*}
(cf.~\eqref{res_alt}). Evaluating $\widehat{d}_{\bp,n}(\br')$ via
contour integration requires knowledge of the density solution
$\widetilde{\psi}_\bp(\omega) =
(\widetilde{\psi}_{\bp,1}(\omega),\dots,\widetilde{\psi}_{\bp,N}(\omega))\in
\C^N$ in~\eqref{disc_vrsn} for a sufficient number of values of
$\omega\in C_n^\mathrm{e}$. (In all of the numerical examples
presented in this paper, the contour $C_n^\mathrm{e}$ was chosen as
the circle centered at $\rho_n$ of radius $10^{-5}$, and the necessary
integrals around $C_n^\mathrm{e}$ were evaluated using the trapezoidal
rule applied to~\eqref{dens_resid} with a number $J_C = 10$ of
integration points $\omega_1,\dots,\omega_{J_C}$.)  To avoid the
costly inversion of the boundary integral operators for frequency
points $\omega_j\in C_n^\mathrm{e}$, $1\leq j\leq J_C$, and for
$1\leq n\leq N_h^{I,\mathrm{e}}$, the densities
$\widetilde{\psi}_\bp(\omega_j)$ are obtained from the rational
approximants produced as part of the IE algorithm’s resonance
evaluation---specifically, for $\rho_n \in P_{h,e}^{I_\ell}$, the
corresponding IE rational approximant
$R^{I_\ell}$~\eqref{eqn:Rat_approx_set} is used. Thus, for
$\rho_n \in P_{h,e}^{I_\ell}$ and for each $1 \leq j \leq J_C$, an
accurate numerical approximation of $\psi_\bp(\br',\omega_j)$ is
obtained by exploiting the relation
\begin{equation}\label{psi-R}
\widetilde{\psi}_\bp(\omega_j) \approx R^{I_\ell}(\omega_j),
\end{equation}
where $R^{I_\ell}(\omega_j)$ denotes the $N$-dimensional vector
rational approximant~\eqref{eqn:rat_approx_vec_no-exp} of
$\widetilde{\psi}_\bp(\omega_j)$. 

\section{Frequency-domain singularity subtraction}\label{sec:RS}
The FTH method reviewed in Section~\ref{sec:abl_rev} enables
essentially dispersion-free simulation of~\eqref{eqn:wave} over
arbitrarily long times, but it encounters difficulties when complex
resonances lie close to the real axis---a situation that typically
occurs for strongly trapping scattering obstacles.  In such cases, the
integrand of the inverse Fourier transform~\eqref{eqn:ukw} becomes
nearly singular, leading to slow convergence of the Fourier transform
method described in Section~\ref{sec:inv_four_tran} as the number $J$
of quadrature points is increased---as would occur with any
Fourier-based approach that does not explicitly account for such
near-singularities. This difficulty is particularly pronounced in the
presence of strongly trapping obstacles under wideband or
high-frequency regimes, where hundreds or even thousands of resonance
poles may cluster near the real-frequency integration path.

To address this issue, Section~\ref{sec:SingSubIntrod} introduces a
complex resonance subtraction method that decomposes the near-singular
integrand into two components: a smooth, singularity-subtracted term
$U_{\bp,h}^\mathrm{s}(\br,\omega)$, which is free of sharp
frequency-domain features, and a second term which equals the sum of a
finite number of isolated polar singularities. The resulting Fourier integrals
over the frequency interval~\eqref{eq:fr_int} are denoted by
$\mathcal{I}_1(\br,t)$ and $\mathcal{I}_2(\br,t)$, respectively.

Section~\ref{sec:I2eval} then highlights the straightforward
evaluation of $\mathcal{I}_1(\br,t)$ via the quadrature rule described
in Section~\ref{sec:inv_four_tran} and it presents an efficient method
for computing $\mathcal{I}_2(\br,t)$ with a computational cost per
time-evaluation point that remains uniformly bounded for arbitrarily
long times. This is achieved by leveraging an asymptotic expansion of
the quantity $\mathcal{I}_2(\br,t)$ constructed on the basis of the
associated complex resonances.  Finally, Section~\ref{sec:SEM}
concerns the singularity expansion of the full scattered field
$u(\br,t)$: it provides a non-rigorous but reasonable rationale for the
suggestion made in Section~\ref{intro} that such expansions remain
valid even for scattering problems involving obstacles with arbitrary
trapping characteristics.

\begin{remark}\label{rmk:using_true_pols}
  The singularity-subtraction (SS) method is introduced below on the
  basis of the exact resonances $\sigma_n$~\eqref{eqn:inc_res} and the
  corresponding residues $c_{\bp,n}(\br)$~\eqref{eqn:res} in the
  region $\mathcal{M}_h^I$~\eqref{eq:box}.  For a given incident
  field, however, a modified version of the SS approach can be
  employed, in which the exact resonances and residues are replaced by
  the incidence-excited resonances $\rho_n$ and associated residues
  $d_{\bp,n}$ (see Section~\ref{sec:realleftscal}) corresponding to
  the same incident field.  This modified formulation still yields a
  decomposition into a smooth component
  $U_{\bp,h}^{\mathrm{s,e}}(\br,\omega)$---free of sharp
  frequency-domain variations---and a second component that captures
  the sharp polar contributions.  This fact has been consistently
  observed across a wide range of test cases (see, e.g.,
  Section~\ref{sec:SingSubEx} and, in particular,
  Figure~\ref{fig:subtraction_regularizing_example}).  Although the
  incidence-excited approach does not capture all of the resonances
  $\rho_n$ located within the region $\mathcal{M}_h^I$, it
  successfully identifies all resonances responsible for the spikes in
  the frequency-domain solution along the real frequency axis; see
  also Remark~\ref{rmk:phIe_is_numer}.
\end{remark}

%%%%%%%%%%%%%%%%%%%%%%%%%%%%%%%%%%%%%%%%%%%%%%%%%%%%%%%%%%%%%%%%%%%%%%%%%%%%%%%%%%%%%%%%%%%%%%%%
\subsection{Singularity subtraction and the integrals
  $\mathcal{I}_{1}(\br,t)$ and
  $\mathcal{I}_{2}(\br,t)$}\label{sec:SingSubIntrod}
In order to improve the convergence of the numerical inverse Fourier
transform~\eqref{eqn:ukw} in the presence of complex resonances near
the real axis, the singularity-subtraction method utilizes the resonances $\sigma_n$~\eqref{eqn:inc_res} and corresponding residues $c_{\bp,n}(\br)$~\eqref{eqn:res}. Since, per equations~\eqref{wind_uk}
and~\eqref{eqn:ukw}, $u_k(\br,t)$ is produced from
$U^\mathrm{slow}_k(\br,\omega)$, in view of~\eqref{eq:ukslow} we see
that the poles $\sigma_n$ of $U_{\bp}(\br,\omega)$ and corresponding residues
account for all of the near singular behavior in the integrand
of~\eqref{eqn:ukw}.

In detail, with reference to equations~\eqref{eqn:inc_res} and~\eqref{eqn:res}, and with $I$ as defined in~\eqref{eq:fr_int}, we use all complex resonances
$\sigma_n \in P_h^{I}$ for $1 \le n \le N_h^{I}$, together with their
corresponding residues $c_{\bp,n}(\br)$, to define the regularized,
singularity subtracted function $U_{\bp,h}^\mathrm{s}(\br,\omega)$ by
\begin{equation}\label{eqn:ghs}
  U_{\bp,h}^\mathrm{s}(\br,\omega) = U_\bp(\br,\omega) - \sum_{n=1}^{N_h^{I}}\frac{c_{\bp,n}(\br)}{\omega - \sigma_n}.
\end{equation}
In view of~\eqref{eq:ukslow}, the singularity-subtracted Fourier
transform~\eqref{eqn:ukw} is then defined by
\begin{equation}\label{eqn:sing_sub_smooth}
    \mathcal{I}_{1,k}(\br,t) = \frac{1}{2\pi} \int_{W_1}^{W_2}  A_k^{\mathrm{slow}}(\omega)U_{\bp,h}^\mathrm{s}(\br,\omega)e^{-\mathrm{i}\omega (t - s_k)}d\omega,
\end{equation}
and, thus, letting the ``singularity integrals'' be given by
\begin{equation}\label{eqn:sing_sub_sing}
    \mathcal{I}_{2,k}(\br,t)  =  \frac{1}{2\pi} \sum_{n=1}^{N_h^{I}} c_{\bp,n}(\br) \int_{W_1}^{W_2}  \frac{A_k^{\mathrm{slow}}(\omega)}{\omega - \sigma_n}e^{-\mathrm{i}\omega (t - s_k)}d\omega,
\end{equation}
we re-express~\eqref{eqn:ukw} in the form
\begin{equation}~\label{eqn:sing_sub}
u^{I}_k(\br,t) =  \mathcal{I}_{1,k}(\br,t)  + \mathcal{I}_{2,k}(\br,t).
\end{equation}
Therefore, the field $u^{I}(\br,t)$ in~\eqref{eqn:uab} is given by
\begin{equation}\label{eqn:sing_sub_nok}
u^{I}(\br,t) =  \mathcal{I}_{1}(\br,t)  + \mathcal{I}_{2}(\br,t).
\end{equation}
where
\begin{equation}\label{eqn:I1I2tok}
  \mathcal{I}_{1}(\br,t) = \sum_{k=1}^K \mathcal{I}_{1,k}(\br,t) \quad \text{and} \quad \mathcal{I}_{2}(\br,t) = \sum_{k=1}^K \mathcal{I}_{2,k}(\br,t).
\end{equation}
Utilizing~\eqref{eqn:BandBk} and~\eqref{eqn:bkslow}, further, we obtain
\begin{equation}\label{eqn:I2}
  \mathcal{I}_{2}(\br,t) = \frac{1}{2\pi} \sum_{n=1}^{N_h^{I}} c_{\bp,n}(\br) \int_{W_1}^{W_2}  \frac{A(\omega)}{\omega - \sigma_n}e^{-\mathrm{i}\omega t}d\omega.
\end{equation}

%%%%%%%%%%%%%%%%%%%%%%%%%%%%%%%%%%%%%%%%%%%%%%%%%%%%%%%%%%%%%%%%%%%%%%%%%%%%%%%%%%%%%%%%%%%%%%

\subsection{Numerical evaluation of $\mathcal{I}_{1}(\br,t)$ and numerical/asymptotic evaluation of $\mathcal{I}_{2}(\br,t)$}\label{sec:I2eval}

The quantity $\mathcal{I}_1(\br,t)$ is obtained as the sum of the
singularity-subtracted integrals
$\mathcal{I}_{1,k}(\br,t)$~\eqref{eqn:sing_sub_smooth}, and can
therefore be integrated effectively for all times using the FTH
quadrature scheme described in Section~\ref{sec:inv_four_tran}. As
noted in Section~\ref{sec:SEM}, further, $\mathcal{I}_{1,k}(\br,t)$
has consistently been observed to decay exponentially, at a faster
exponential rate than the quantity $\mathcal{I}_2(\br,t)$---as may be
expected by construction---and therefore needs only be calculated at
pre-asymptotic times. Thus, the accurate and efficient numerical
evaluation of the solution $u$ hinges upon the evaluation of the
quantity $\mathcal{I}_{2}(\br,t)$, which ultimately reduces to
computing the integrals
\begin{equation}\label{eqn:base_sing_int}
\int_{W_1}^{W_2} \frac{A(\omega)}{\omega - \sigma_n}e^{-\mathrm{i}\omega t},d\omega,
\quad 1 \le n \le N^I_h,
\end{equation}
which appear on the right-hand side of~\eqref{eqn:I2}. As shown in
what follows, these integrals can be evaluated accurately and
efficiently by combining numerical quadrature at pre-asymptotic times
(Section~\ref{pre-asym}) with an asymptotic expansion for large times
(Sections~\ref{sec:cont_def}--\ref{asym-exp}). Substituting the asymptotic expansion of~\eqref{eqn:base_sing_int} into~\eqref{eqn:I2}, yields the singularity expansion
%~\eqref{eqn:res} and~\eqref{eq:all_poles}
\begin{equation}
  \label{eq:SExp}
  \mathcal{E}_h^{I}(\br,t) = -\mathrm{i}\sum_{n = 1}^{N_h^{I}} c_{\bp,n}(\br) A(\sigma_n)e^{-\mathrm{i}\sigma_nt}
\end{equation}
associated with poles in the set $P^I_h$ for the quantity
$\mathcal{I}_2(\br,t)$. The following theorem, whose proof is
presented in Section~\ref{asym-exp}, provides an estimate on the
accuracy and validity of the singularity expansion~\eqref{eq:SExp}. A
discussion concerning the two main assumptions in
Theorem~\ref{thm:sing_exp} together with the validity of the
corresponding theorem for three-dimensional obstacles is presented in
Remark~\ref{rmk:growth-assump} below.
\begin{theorem}\label{thm:sing_exp}
  Let $\mu(W_1,W_2)$, $\varepsilon(\mu(W_1,W_2))$ and $N_h^I$, be
  defined as in~\eqref{mu}, \eqref{eqn:eps_m_def} and~\eqref{eq:nih},
  respectively.  Further, assume that the relation
  \begin{equation}
    \label{eq:pole-number}
    N_h^I = O(\left(\mu(W_1,W_2)\right)^2)    
  \end{equation}
  holds, and that there exists a constant $D > 0$ such that, for any
  $W_1 < 0$ and $W_2 > 0$, the residues
  $c_{\bp,n}(\br)$~\eqref{eqn:res} of the poles contained in the set
  $\mathcal{M}_h^I$~\eqref{eq:box} associated with the interval
  $I = I(W_1,W_2)$~\eqref{eq:fr_int} satisfy
  \begin{equation}
    \label{eq:residue-size}
    |c_{\bp,n}(\br)| < D.
  \end{equation}
  Then there exists a constant $M > 0$ such that the error of the
  approximation of $\mathcal{I}_2(\br,t)$~\eqref{eqn:I2} by the
  singularity expansion $\mathcal{E}_h^{I}$~\eqref{eq:SExp} satisfies
  the bound
\begin{equation}\label{eqn:I2_Asym_err_thm}
  |\mathcal{I}_2(\br,t) - \mathcal{E}_h^{I}(\br,t)| \le M (\mu(W_1,W_2))^3 (e^{-h(t - T^\mathrm{inc})} + \varepsilon(\mu(W_1,W_2))) \quad \text{for} \quad t > T^\mathrm{inc}.
\end{equation}
\end{theorem}

\begin{proof}
The proof is provided in Section~\ref{asym-exp} using results established in Sections~\ref{sec:cont_def}--\ref{sec:C2C4}.
\end{proof}

\begin{remark}\label{rmk:growth-assump}
  The 2D bound~\eqref{eq:pole-number}, which is one of the assumptions
  in Theorem~\ref{thm:sing_exp} (and whose validity has been
  confirmed, even for highly trapping obstacles, by means of numerical
  experiments in the course of this work), is established
  in~\cite{vodev1994sharp} under certain conditions concerning the
  growth of the characteristic values of the resolvent; see
  also~\cite[Sec. 4.3]{dyatlov2019mathematical}. The corresponding 3D
  bound, namely, $N^I_h = O((\mu(W_1,W_2))^3)$, was established, for
  any obstacle, in~\cite{melrose1984polynomial}; on the basis of this
  3D result a corresponding version of Theorem~\ref{thm:sing_exp} can
  be established, with the right-hand common factor $\mu(W_1,W_2)^3$
  in equation~\eqref{eqn:I2_Asym_err_thm} replaced by
  $\mu(W_1,W_2)^4$. The hypothesis~\eqref{eq:residue-size} in
  Theorem~\ref{thm:sing_exp}, on the other hand, has been
  computationally verified in the course of this work, even in cases
  involving strongly trapping obstacles (see Section~\ref{sec:Num_Res}
  and in particular Section~\ref{sec:bounded_residues}). Theoretical
  studies establishing such uniform residue bounds for problems of scattering
  by smooth potentials can be found in~\cite{tang2000resonance}.
\end{remark}

It is important to note that, since by
definition~\eqref{eq:box}, \eqref{eqn:inc_res}, we have
$|\Im{\sigma_n}| < h$ for $\sigma_n\in\mathcal{M}_h^I$, the
exponential term on the right hand side of~\eqref{eqn:I2_Asym_err_thm}
decays at a faster exponential rate than the singularity expansion
$\mathcal{E}_h^{I}(\br,t)$. Thus in view of Theorem~\ref{thm:sing_exp}
and equation~\eqref{eqn:eps_m_def}, selecting $W_1$ and $W_2$
sufficiently large so that the second right-hand summand
in~\eqref{eqn:I2_Asym_err_thm} is smaller than a prescribed error
tolerance $\tau$ (for instance, values of $\tau$ close to machine
precision were used in all examples presented in this paper), the
bound~\eqref{eqn:I2_Asym_err_thm} implies that the singularity
expansion $\mathcal{E}_h^{I}(\br,t)$ yields an accurate large-time
approximation,
\begin{equation}\label{eqn:I2_asym_exp}
\mathcal{I}_2(\br,t) \approx \mathcal{E}_h^{I}(\br,t),
\end{equation}
with errors that are exponentially smaller than the singularity
expansion itself, until the first term on the right-hand side
of~\eqref{eqn:I2_Asym_err_thm} reaches the numerical tolerance $\tau$.

The remainder of the present section~\ref{sec:I2eval} proceeds as
follows. Section~\ref{pre-asym} describes a simple algorithm for the
numerical evaluation of~\eqref{eqn:base_sing_int} at pre-asymptotic
times. On the basis of contour deformation,
Sections~\ref{sec:cont_def}--\ref{sec:C2C4} then derive a large time
asymptotic expansion for~\eqref{eqn:base_sing_int}. Using these
elements, the proof of Theorem~\ref{thm:sing_exp} is presented in
Section~\ref{asym-exp}.

\subsubsection{Numerical evaluation of the integral~\eqref{eqn:base_sing_int} at pre-asymptotic times\label{pre-asym}}
  
It is important to note that, unlike
$A_k^{\mathrm{slow}}$~\eqref{eqn:bkslow}, the function $A$ that
appears in the expression~\eqref{eqn:base_sing_int} is generally not a
slowly oscillatory function of $\omega$. However, unlike the
integrands in~\eqref{wind_uk}, the integrands in~\eqref{eqn:I2} do not
require the solution of Helmholtz PDEs and are independent of both $k$
and $\br$. Consequently, the cost required by the direct numerical
evaluation of the corresponding integrals is significantly lower than
the cost required by~\eqref{wind_uk}. Nevertheless, the evaluation
cost for the integrals~\eqref{eqn:base_sing_int} does increase with
$t$, which motivates the development of the asymptotic methods in
Sections~\ref{sec:cont_def}--\ref{sec:C2C4}. In the pre-asymptotic
regime considered in this section, the numerical evaluation of these
integrals requires particular care due to the near singularity that
occurs for poles $\sigma_n$ close to the real axis. To handle this
near-singularity, we employ the equivalent representation
\begin{equation}\label{eqn:ss_second_int}
\int_{W_1}^{W_2} \frac{A(\omega)}{\omega - \sigma_n} e^{-\mathrm{i}\omega t} d\omega
= \int_{W_1}^{W_2} \frac{A(\omega)e^{-\mathrm{i}\omega t} - A(\sigma_n)e^{-\mathrm{i}\sigma_n t}}{\omega - \sigma_n} d\omega
+ A(\sigma_n)e^{-\mathrm{i}\sigma_n t} \int_{W_1}^{W_2} \frac{1}{\omega - \sigma_n} d\omega.
\end{equation}
The first right-hand integrand in~\eqref{eqn:ss_second_int} does not
exhibit sharp variations, regardless of the proximity of $\sigma_n$ to
the integration interval, and it can therefore be accurately
evaluated, for sufficiently small $t$, using standard quadrature
methods. In this paper, we employ the Clenshaw–Curtis
rule~\cite{waldvogel2006fast} for this purpose. The second integral,
in turn, can be evaluated analytically: it equals
$\log\frac{W_2-\sigma_n}{W_1-\sigma_n}.$

\subsubsection{Contour deformation of the integral~\eqref{eqn:base_sing_int}}\label{sec:cont_def}
For large times, the first right-hand integrand
in~\eqref{eqn:ss_second_int} becomes a highly oscillatory function of
$\omega$, rendering standard quadrature methods ineffective. In this
regime we therefore discard the
decomposition~\eqref{eqn:ss_second_int} and  employ instead a
large-time asymptotic approximation of~\eqref{eqn:base_sing_int} that
is obtained by deforming the corresponding integration contour into
the complex plane. In detail, using certain values $\delta_1 >0$ and
$\delta_2 > 0$, the deformed contour
\begin{equation}
  \label{eq:contour}
 C  = \bigcup_{j=1}^5 C_j
\end{equation}
connects the points $W_2 + 0\mathrm{i}$ and $W_1 + 0\mathrm{i}$ via a
sequence of five segments $C_j$, $j = 1, \dots, 5$. Specifically,
using the same parameter $h$ as in the IE algorithm, $C_1$ joins $W_2$
to $W_2 + \delta_1$; $C_2$ connects $W_2 + \delta_1$ to
$W_2 + \delta_1 - (h + \delta_2)\mathrm{i}$; $C_3$ connects
$W_2 + \delta_1 -(h + \delta_2)\mathrm{i}$ to
$W_1 - \delta_1 - (h + \delta_2)\mathrm{i}$; $C_4$ proceeds from
$W_1 - \delta_1 - (h + \delta_2)\mathrm{i}$ to $W_1 - \delta_1$; and
finally, $C_5$ connects $W_1 - \delta_1$ to $W_1$.  Taking into
account that, in view of~\eqref{eq:all_poles} we have
$\sigma_n \in \mathcal{M}^I_h$, the values of the parameters $\delta_1$
and $\delta_2$ should be selected so as to guarantee that $\sigma_n$ is
sufficiently far from the vertical segments $C_2$ and $C_4$ and the
horizontal segment $C_3$, but they are otherwise arbitrary. The
desired asymptotic approximation is obtained in terms of the residues
that emerge as the interval $I = [W_1,W_2]$ is deformed into the
contour $C$ in~\eqref{eq:contour}.

Using the contour $C$ and applying the residue theorem we obtain
\begin{equation}\label{eqn:base_sing_int_cont}
  \int_{W_1}^{W_2} \frac{A(\omega)}{\omega - \sigma_n}e^{-\mathrm{i}\omega t}d\omega = -2\pi\mathrm{i} A(\sigma_n)e^{-i\sigma_nt} + \sum_{j=1}^5\int_{C_j} \frac{A(\omega)}{\omega - \sigma_n}e^{-\mathrm{i}\omega t}d\omega.
\end{equation}
As shown in what follows, the integral terms on the right-hand side of
this equation either decay exponentially---at a rate faster than the
residue term as $t \to \infty$---or become super-algebraically small,
uniformly in time, as $\mu(W_1,W_2) \to +\infty$~\eqref{mu}, with the
restriction $t > T^\mathrm{inc}$ for the integrals over $C_2$ and
$C_4$.
%(see also Remark~\ref{rmk:horizontal_contours}).

The quantities
\begin{equation}\label{eqn:mjn_def}
  M_{j,n} = M_{j,n}(W_1,W_2) =\int_{C_j} \frac{1}{|\omega - \sigma_n|}|d\omega| \quad (1\le j \le 5)
\end{equation}
are used in Sections~\ref{sec:C1C5} through~\ref{asym-exp} to estimate
the asymptotic character of the right-hand integrals
in~\eqref{eqn:base_sing_int_cont}.

\subsubsection{Uniform-in-time super-algebraic decay of $C_1$ and $C_5$ integrals in~\eqref{eqn:base_sing_int_cont} as $\mu(W_1,W_2) \to +\infty$}\label{sec:C1C5}
In view of~\eqref{eqn:AepsBnd} and~\eqref{eqn:mjn_def}, the $j = 1$
and $j=5$ integrals in~\eqref{eqn:base_sing_int_cont} satisfy
\begin{equation}\label{eqn:cj_real_1}
  \left|\int_{C_j} \frac{A(\omega)}{\omega - \sigma_n}e^{-i\omega t}d\omega\right|  \le M_{j,n} \varepsilon(\mu(W_1,W_2)) \quad j = 1,5.
\end{equation}
As shown in what follows the quantities  $M_{j,n}$ with $j = 1,5$ are
bounded by a constant times $\mu(W_1,W_2)$ as
$\mu(W_1,W_2)\to +\infty$~\eqref{mu}, so that the growth of $M_{j,n}$
for $j = 1,5$ is overcome in~\eqref{eqn:cj_real_1} by the
corresponding super-algebraically fast decay~\eqref{eqn:eps_m_def} of
$\varepsilon(\mu(W_1,W_2))$). We establish the necessary bound for
$M_{1,n}$; the corresponding result for $M_{5,n}$ follows
similarly. To do this we write $\sigma_n = \sigma_{n,1} + i\sigma_{n,2}$
with $\sigma_{n,1},\sigma_{n,2}\in\mathbb{R}$, and note that 
\begin{equation}\label{w2delta}
  M_{1,n} = \frac{1}{2}\Biggr( \log\Biggl(1 + \frac{\omega - \sigma_{n,1}}{\sqrt{(\omega - \sigma_{n,1})^2 + (\sigma_{n,2})^2}}\Biggr) - \log\Biggl(1 - \frac{\omega - \sigma_{n,1}}{\sqrt{(\omega - \sigma_{n,1})^2 + (\sigma_{n,2})^2}}\Biggr) \Biggr) \Bigg|_{\omega =W_2}^{\omega = W_2 + \delta_1}.
\end{equation}
Then, letting
\begin{equation*}
    f^\pm(x) = \left|\log\left(1 \pm \frac{1}{\sqrt{1 + x^2}}\right)\right|, \quad x_1 = \frac{\sigma_{n,2}}{W_2 - \sigma_{n,1}}, \quad x_2 = \frac{\sigma_{n,2}}{W_2 + \delta_1 - \sigma_{n,1}}
\end{equation*}
and applying the triangle inequality to the right-hand side
of~\eqref{w2delta} we obtain the bound
\begin{equation}\label{eqn:ns_log_bound}
       M_{1,n} \le \frac{1}{2}\left(f^+(x_1) + f^+(x_2) +f^-(x_1) + f^-(x_2) \right).
\end{equation}
Since $\sigma_{n,1} \le W_2$~\eqref{eq:all_poles} (and, thus
$W_2 -\sigma_{n,1}\geq 0$), it follows that
$f^+(x_1) + f^+(x_2) \le 2\log(2)$. To quantify the behavior of
$f^-(x_1)$ and $f^-(x_2)$, on the other hand, we appeal to the
fact~\cite{burq1998decroissance} that, for any obstacle $\Gamma$,
regardless of trapping character, we have
$|\sigma_{n,2}| > e^{-\beta |\sigma_{n,1}|}$ for some (obstacle-dependent)
constant $\beta > 0$ (see
also~\cite[Sec. 2.4]{zworski2017mathematical}). Therefore, since
$|\sigma_{n,1}|\leq \max\{-W_1, W_2\}$, it follows that
\begin{equation}\label{eqn:xbnds}
    \frac{e^{-\beta \max\{-W_1, W_2\}}}{W_2 - W_1} \le x_1 \le \infty \quad \text{and} \quad \frac{e^{-\beta \max\{-W_1, W_2\}}}{W_2 + \delta_1 - W_1} \le x_2 \le \infty.
\end{equation}
Further, it is easily verified, by means of a Taylor expansion for $x$ near zero
and a straightforward estimate for $|x|\geq 1$, that there exists a
constant $K > 0$ such that
\begin{equation}\label{eqn:xbnds2}
 0 \leq f^-(x) \leq K \, \big|\log |x|\big| 
 \quad \text{for } |x| < 1,
 \qquad 
 f^-(x) \leq K 
 \quad \text{for } |x| \geq 1.
\end{equation}
Equations~\eqref{eqn:xbnds} and~\eqref{eqn:xbnds2} provide bounds on
$f^-(x_1)$ and $f^-(x_2)$. Combining these bounds with the simple
bound established above for $f^+(x_1) + f^+(x_2)$, and using the
inequality $|\log(W_2 - W_1)| \leq |\log(W_2 + \delta_1 - W_1)|$ valid for
$W_2-W_1 > 1$, $\delta_1>0$, we obtain the estimate
\begin{equation}\label{eqn:M1nbound}
  M_{1,n} \;\leq\; 
  \log(2) + K \Big(1 + \beta \max\{-W_1,\, W_2\} 
  + \log(W_2 + \delta_1 - W_1)\Big),
  \quad W_2 - W_1 > 1.
\end{equation}
Together with~\eqref{eqn:cj_real_1}, in turn, the
bound~\eqref{eqn:M1nbound} and an analogous bound for $M_{5,n}$ show
that the integrals over the contours $C_1$ and $C_5$ on the right-hand
side of~\eqref{eqn:base_sing_int_cont} are super-algebraically small,
uniformly in time, as $\mu(W_1,W_2)\to +\infty$~\eqref{mu}, as
claimed.

\subsubsection{Simple preparation estimates for the integrals over $C_j$, $j = 2,3,4$}
To estimate the contributions from the contour segments $C_j$ with
$j = 2,3,4$, on the other hand, we set $\omega = \omega_1 +
\mathrm{i}\omega_2$. In view of
equation~\eqref{eqn:singleinc}, we then obtain
\begin{equation*}
    A(\omega) = \int_0^{T^\mathrm{inc}} a(t)\, e^{\mathrm{i}\omega t} \, dt =\int_0^{T^\mathrm{inc}} a(t)\, e^{\mathrm{i}\omega_1 t - \omega_2 t} \, dt = e^{-\omega_2 T^\mathrm{inc}} A^{\mathrm{bd}}(\omega),
\end{equation*}
where
\begin{equation*}
    A^{\mathrm{bd}}(\omega) = \int_0^{T^\mathrm{inc}} a(t)\, e^{\mathrm{i}\omega_1 t - \omega_2(t - T^\mathrm{inc})} \, dt.
\end{equation*}
Clearly, for the relevant values of $\omega_2 = \Im(\omega) \leq 0$
and for all $\omega_1\in\mathbb{R}$, the bounded quantity
$A^{\mathrm{bd}}$ satisfies
\begin{equation}\label{eqn:Abdn}
   |A^{\mathrm{bd}}(\omega)| \leq \alpha\quad\mbox{where}\quad \alpha = \int_0^{T^\mathrm{inc}} |a(t)| \, dt. 
\end{equation}
It follows that, for $j = 2,3,4$, the integrals over the contours
$C_j$ on the right-hand side of~\eqref{eqn:base_sing_int_cont} may be
expressed in the form
\begin{equation}\label{eqn:contourCJ}
  \int_{C_j} \frac{A(\omega)}{\omega - \sigma_n} \, e^{-\mathrm{i} \omega t} \, d\omega = \int_{C_j} \frac{e^{\omega_2 (t - T^\mathrm{inc})} \, A^{\mathrm{bd}}(\omega_1 + \mathrm{i} \omega_2) \, e^{-\mathrm{i} \omega_1 t}}{\omega_1 + \mathrm{i} \omega_2 - \sigma_n} \, d\omega.
\end{equation}
\subsubsection{Exponential decay of the integral over $C_3$
  in~\eqref{eqn:base_sing_int_cont} as $t \to \infty$ \label{exp-dec}}
Over the contour $C_3$ we have $\omega_2 = -(h + \delta_2) < 0$, and thus in view of~\eqref{eqn:mjn_def},~\eqref{eqn:Abdn}, and~\eqref{eqn:contourCJ} we obtain
\begin{equation}\label{eqn:contourC2}
  \left | \int_{C_3}\frac{A(w)}{\omega - \sigma_n}e^{-\mathrm{i}\omega t}d\omega\right |\le e^{-(h + \delta)(t-T^\mathrm{inc})}\int_{C_3} \left|\frac{A^\mathrm{bd}(\omega)}{\omega - \sigma_n}\right||d\omega| \le \alpha M_{3,n}(W_1,W_2) e^{-h(t-T^\mathrm{inc})}.
\end{equation}
Using a closed form expression similar to~\eqref{w2delta} we see that
the quantity $M_{3,n}(W_1,W_2)$~\eqref{eqn:mjn_def} grows at most
logarithmically as $\mu(W_1,W_2)\to+\infty$~\eqref{mu}. In particular,
for fixed $W_1$ and $W_2$, the $C_3$ integral decays exponentially as
$t \to \infty$ at the rate $e^{-h(t - T^\mathrm{inc})}$. Since
$\Im(\sigma_n) > -h$, this decay is faster than that of the first term
on the right-hand side of~\eqref{eqn:base_sing_int_cont}.

\subsubsection{Super-algebraic decay of the integrals over $C_2$ and $C_4$ in~\eqref{eqn:base_sing_int_cont} as $\mu \to +\infty$ for $t > T^\mathrm{inc}$}\label{sec:C2C4}

To estimate the integrals over the vertical segments $C_j$ with $j=2$
and $j=4$ we first integrate by parts the integral $A^\mathrm{bd}$,
which yields
\begin{equation}\label{eqn:compAdecay}
    |A^{\mathrm{bd}}(\omega_1 + \mathrm{i}\omega_2)| \le \frac{1}{|\omega_1 + \mathrm{i}\omega_2|^n} \int_0^{T^\mathrm{inc}} |a^{(n)}(t)e^{-\omega_2(t - T^\mathrm{inc})}|dt \le  \frac{1}{|\omega_1|^n}\int_0^{T^\mathrm{inc}} |a^{(n)}(t)|dt \quad  (\omega_2 \le 0)
\end{equation}
for all $\omega_1\in\mathbb{R}$ and all $n\in\mathbb{N}$---since
$0 < e^{-\omega_2(t - T^\mathrm{inc})}\leq 1$ for
$0\leq t\leq T^\mathrm{inc}$. Thus, in view of~\eqref{eqn:eps_m_def},
for $W_1<0$ and $W_2>0$ we have
\begin{equation}\label{eps-I}
  |A^{\mathrm{bd}}(W_1 - \delta_1 + \mathrm{i}\omega_2)| < \varepsilon(\mu(W_1,W_2)) \quad \text{and} \quad |A^{\mathrm{bd}}(W_2 + \delta_1 + \mathrm{i}\omega_2)| < \varepsilon(\mu(W_1,W_2)) \quad \text{for all} \quad {\omega_2 \le 0}.
\end{equation}
Since $\omega_2 \le 0$ on the contours $C_2$ and $C_4$, it follows that $|e^{\omega_2(t - T^\mathrm{inc})}| < 1$ whenever $t > T^\mathrm{inc}$. Consequently, equations \eqref{eqn:mjn_def}, \eqref{eqn:contourCJ}, and \eqref{eps-I} yield
\begin{equation}\label{eqn:vert_bound}
    \left | \int_{C_j}\frac{A(w)}{\omega - \sigma_n}e^{-\mathrm{i}\omega t}d\omega\right| \le M_{j,n}\, \varepsilon(\mu(W_1,W_2))  \quad \text{for} \quad j = 2,4 \quad\mbox{and}\quad  t > T^{\mathrm{inc}}.
\end{equation}
Here, like $M_{3,n}$, the $j=2,4$ quantities
$M_{j,n}$~\eqref{eqn:mjn_def} grow at most logarithmically as $W_1$
and $W_2$ grow without bound.

In summary, as noted in connection
with~\eqref{eqn:base_sing_int_cont},
Sections~\ref{sec:C1C5}--\ref{sec:C2C4} establish that the
contributions from the integral terms on the right-hand side of that
equation are either super-algebraically small, uniformly in time, as
$\mu(W_1,W_2) \to \infty$ ($j=1,2,4,5$), or decay exponentially as
$t \to \infty$ ($j=3$), at a rate faster than that of the
exponentially decaying residue term.

\subsubsection{Large-time asymptotic expansion of $\mathcal{I}_2(\br,t)$: Proof of Theorem~\ref{thm:sing_exp}\label{asym-exp}}
\begin{proof}[Proof of Theorem~\ref{thm:sing_exp}]
  Substituting~\eqref{eqn:base_sing_int_cont} into~\eqref{eqn:I2}
  and using~\eqref{eq:SExp} we obtain
 \begin{equation}\label{eqn:sem_error_rate}
  \mathcal{I}_2(\br,t) = \mathcal{E}_h^{I}(\br,t)  + \sum_{j = 1}^5 \mathcal{I}^{C_j}_{2}(\br,t)
\end{equation}
where
\begin{equation}\label{eqn:I2Cj}
  \mathcal{I}^{C_j}_{2}(\br,t) = \sum_{n = 1}^{N_h^{I}}c_{\bp,n}(\br)\int_{C_j} \frac{A(\omega)}{\omega - \sigma_n}e^{-\mathrm{i} \omega t} d\omega\quad,\quad 1 \le j \le 5.
\end{equation}
Then, calling
\begin{equation}\label{MjIh-def}
  M_{j}^{I,h}(\br) = \sum_{n = 1}^{N_h^{I}}|c_{\bp,n}(\br)M_{j,n}|, \quad 1 \le j \le 5,
  \end{equation}
in view of~\eqref{eqn:contourC2} we obtain
\begin{equation}\label{I2C1_bound}
  |\mathcal{I}^{C_3}_{2}(\br,t)| \le  \alpha M^{I,h}_3(\br)e^{-h(t - T^\mathrm{inc})}.
\end{equation}
Further, equations~\eqref{eqn:cj_real_1} and~\eqref{eqn:vert_bound}
tell us that
\begin{equation}\label{I2Cj_bound}
  |\mathcal{I}^{C_j}_{2}(\br,t)| \le  M^{I,h}_j(\br)\, \varepsilon(\mu(W_1,W_2)) \quad \text{for} \quad j = 1,2,4,5 \quad\mbox{and}\quad t > T^\mathrm{inc},
\end{equation}
Thus, in view of~\eqref{eqn:sem_error_rate}, it follows that
\begin{equation}\label{eqn:I2_Asym_err}
  |\mathcal{I}_2(\br,t) - \mathcal{E}_h^{I,\mathrm{e}}(\br,t)| \le  \alpha M^{I,h}_3(\br)e^{-h(t - T^\mathrm{inc})} + \varepsilon(\mu(W_1,W_2))\sum_{\substack{j=1 \\ j \ne 3}}^5M^{I,h}_j(\br) \quad \text{for} \quad t > T^\mathrm{inc}.
\end{equation}

Now, employing~\eqref{eqn:M1nbound} and similar estimates for
$M_{j,n}$~\eqref{eqn:mjn_def}, we obtain the estimate
\begin{equation*}
M_{j,n} \le K\mu(W_1, W_2), \quad 1 \le j \le 5, \quad 1 \le n \le N^I_h
\end{equation*}
for a certain constant $K$.  Combining this inequality
with~\eqref{MjIh-def} and the assumptions~\eqref{eq:pole-number}
and~\eqref{eq:residue-size} yields the bound
\begin{equation*}
M^{I,h}_j(r) \le K(\mu(W_1, W_2))^3, \quad 1 \le j \le 5.
\end{equation*}
Substituting this bound into~\eqref{eqn:I2_Asym_err} we then obtain the desired relation
\begin{equation*}
|\mathcal{I}_2(\br,t) - \mathcal{E}_h^{I}(\br,t)|
\le M(\mu(W_1, W_2))^3 \big(e^{-h(t - T^\mathrm{inc})} + \varepsilon(\mu(W_1, W_2))\big),
\quad t > T^\mathrm{inc}
\end{equation*}
and the proof of the theorem is thus complete.
\end{proof}

\subsection{Time-domain singularity expansion}\label{sec:SEM}
Although the asymptotic validity of the classical singularity
expansion~\cite{lax1969decaying,lax2005exponential,
  tang2000resonance,baum2005singularity} is not a requirement for the
validity of the Singularity Subtraction-enabled FTH method presented in this paper
(which is a numerical algorithm that produces the solution $u$ at all
times, and not only asymptotically for large times), it is relevant to
highlight certain interesting connections between the two
approaches. In view of~\eqref{eqn:sing_sub_nok}, and per arguments as
in the proof of Theorem~\ref{thm:sing_exp}, it may be expected that
(a)~$\mathcal{I}_1(\br,t)$ decays exponentially (up to the error
tolerance $O(\varepsilon(\mu(W_1,W_2)))$, at a rate faster than the
most rapidly decaying exponentials in~\eqref{eq:SExp}; and, that (b)~A
bound of the form $|u - u^I| < O(\varepsilon(\mu(W_1,W_2)))$ holds as
$\mu(W_1,W_2) \to +\infty$. Under these conditions, within an
$O(\varepsilon(\mu(W_1,W_2)))$ error tolerance, the following
asymptotic representation it is expected to hold:
\begin{equation}\label{eqn:usem}
  u(\br,t) \sim \mathcal{E}_h^I = -\mathrm{i}\sum_{n=1}^{N_h^{I}} c_{\bp,n}(\br) A(\sigma_n)e^{-\mathrm{i}\sigma_nt} \quad \text{as}\quad t \to \infty.
\end{equation}
More specifically---though not rigorously established---it is plausible that
\[
|u(\br,t) - \mathcal{E}_h^I| < M (\mu(W_1,W_2))^3 (e^{-h(t - T^\mathrm{inc})} + \varepsilon(\mu(W_1,W_2))) \quad \text{for} \quad t > T^\mathrm{inc}\quad
\]
for some constant $M$.

These plausible expansions and approximations are closely related to
the aforementioned asymptotic expansions of the scattered field
$u(\br,t)$, which have been widely considered in the
literature~\cite{meylan2012complex,hazard2007singularity,baum2005singularity,meylan2014singularity,lalanne2018light};
the overall approach has come to be known as as the ``Singularity
Expansion Method''~\cite{baum2005singularity}.  The validity of such
expansions has only been
established~\cite{lax1969decaying,lax2005exponential,
  tang2000resonance, martin2021time} for 3D non-trapping scatterers
(namely, scatterers for which a billiard ball bouncing off the
scatterer boundaries escapes to infinity after finitely many
bounces). In two dimensions such an expansion could only hold provided
the frequency content of the incident excitation tends to vanish
sufficiently rapidly as the frequency $\omega$ tends to zero---since,
as it is known~\cite{muraveui1979asymptotic}, 2D scattered field only
decay as $O(1/(t\log^2(t)))$ in presence of zero frequency content.

In any case, as mentioned in Section~\ref{intro}, a wide range of
numerical experiments conducted as part of this work clearly suggest
that the singularity expansion method is valid independently of the
trapping character of the scattering obstacles considered.

%%%%%%%%%%%%%%%%%%%%%%%%%%%%%%%%%%%%%%%%%%%%%%%%%%%%%%%%%%%%%%%%%%%%%%%%%

\section{Numerical implementation of FTH with Singularity
  Subtraction}\label{sec:Num_Imp}

This section presents the proposed Singularity Subtraction-enabled FTH
algorithm (FTH-SS) for the numerical solution of the initial and
boundary-value problem~\eqref{eqn:wave}. This method, which combines
the FTH methodology reviewed in Section~\ref{sec:abl_rev} with the
singularity subtraction strategy embodied in
equations~\eqref{eqn:ghs}--\eqref{eqn:sing_sub_nok}
and~\eqref{eqn:I2_asym_exp}, proceeds by first computing the
incidence-excited complex resonances and their residues by means of
the IE method (Section~\ref{sec:realleftscal}). The resulting solution $u\approx u^I(\br,t)$ is obtained via
\begin{equation}
    u^I(\br,t) = \mathcal{I}_1^\mathrm{e}(\br,t) + \mathcal{I}_2^\mathrm{e}(\br,t),
\end{equation}
where $\mathcal{I}_1^\mathrm{e}(\br,t)$ and
$\mathcal{I}_2^\mathrm{e}(\br,t)$ are defined in the same manner as
$\mathcal{I}_1(\br,t)$ and $\mathcal{I}_2(\br,t)$ in
equations~\eqref{eqn:ghs}--\eqref{eqn:sing_sub_sing}
and~\eqref{eqn:I1I2tok}, with $\sigma_n$, $c_{\bp,n}$, and $N_h^I$
replaced by their incident excited versions $\rho_n$, $d_{\bp,n}$, and
$N^{I,\mathrm{e}}_h$, respectively; see
Remark~\ref{rmk:using_true_pols}. In particular, the incident excited
version of regularized singularity subtraction function
$U_{\bp,h}^\mathrm{s}(\br,\omega)$~\eqref{eqn:ghs} is given by

\begin{equation}\label{eq:usie}
U_{\bp,h}^\mathrm{s,e}(\br,\omega) = U_{\bp}(\br,\omega) - \sum_{n=1}^{N^{I,\mathrm{e}}_h} \frac{d_{\bp,n}}{\omega - \rho_n}
  \end{equation}

To reduce computational costs, the algorithm obtains the densities
$\psi_\bp$ (that, per equations~\eqref{eq:psi_p_sol}–\eqref{eq:ukslow}
and~\eqref{eqn:u_sum_uk}–\eqref{eqn:ukw}, are needed by the FTH
algorithm to compute the solutions $U_k^{\mathrm{slow}}$), by means of
an inexpensive reprocessing step applied to the densities $\psi_\bp$
(eq.~\eqref{eq:psi_p_sol}) generated by
Algorithm~\ref{alg:real-line-adaptive} as part of the evaluation of IE
resonances. While the frequency set used during the evaluation of 
Algorithm~\ref{alg:real-line-adaptive} generally
differs from the frequency set $\mathcal{F}$~\eqref{eqn:F_freqs}
needed to compute the inverse Fourier-transform by the quadrature
rules described in Section~\ref{sec:inv_four_tran},
$\mathcal{D}_\mathcal{F}$~\eqref{eqn:dens_set} may be cheaply computed
using the rational approximants~\eqref{eqn:Rat_approx_set}. Indeed,
considering the partition~\eqref{eq:partit}, for
$\omega \in \mathcal{F}\cap I_\ell$ ($\ell = 1,\dots,L$) the
approximant $R^{I_\ell}(\omega_j)$ provides the necessary (accurate)
approximations
\begin{equation}\label{psi-R-reuse}
\psi_\bp(\br',\omega) \approx R^{I_\ell}(\omega),\quad \omega \in \mathcal{F}\cap I_\ell.
\end{equation}

% Alternative algorithm spot

Once the set $\mathcal{D}_\mathcal{F}$ has been obtained, the integral
$\mathcal{I}^\mathrm{e}_{1}(\br,t)$ is evaluated by means of the
corresponding incident excited version of~\eqref{eqn:I1I2tok} using
the quadrature rules presented in Section~\ref{sec:inv_four_tran}. The
evaluation of $\mathcal{I}_2^\mathrm{e}(\br,t)$ proceeds under two
different scenarios. At pre-asymptotic times, on one hand, this
integral is obtained using the corresponding incident excited versions
of~\eqref{eqn:I2} and~\eqref{eqn:ss_second_int} by following the
description accompanying the latter equation. For sufficiently large
times, in view of Remark~\ref{rmk:using_true_pols} and
Theorem~\ref{thm:sing_exp}, the approximation
\begin{equation}\label{eqn:usem-e}
    \mathcal{I}_2^\mathrm{e}(\br,t) \approx \mathcal{E}^{I,\mathrm{e}}_h(\br,t) \quad \text{where} \quad \mathcal{E}^{I,\mathrm{e}}_h(\br,t) =  -\mathrm{i}\sum_{n = 1}^{N_h^{I,\mathrm{e}}} d_{\bp,n}(\br) A(\sigma_n)e^{-\mathrm{i}\rho_nt}
\end{equation}
is used instead of~\eqref{eqn:I2_asym_exp}. This leads to significant
computing-time savings (as evaluation of integrals with
highly-oscillatory integrands is avoided) while capturing exponential
solution decay that, however, for highly trapping structures, can
continue to produce significant scattered fields for long times---as
illustrated in Section~\ref{sec:Whisper}. A pseudo-code for the
singularity subtraction method, provided in
Algorithm~\ref{alg:FTH_with_sinsub}, evaluates the
approximation~\eqref{eqn:uab} $u^I(\br,t)$ of $u(\br,t)$ for all $\br$
in a given set $\mathbf{R}$ of spatial observation points at which the
scattered field is to be produced.

\vspace{0.3cm}
\begin{algorithm}
\DontPrintSemicolon 
Compute the set of rational approximants~\eqref{eqn:Rat_approx_set} and corresponding poles $P_h^{I,\mathrm{e}}$~\eqref{eqn:inc_res} relevant to the incident field using Algorithm~\ref{alg:real-line-adaptive}.\;
For each pole in $P_{h}^{I,\mathrm{e}}$~\eqref{eqn:inc_res} compute the residues at all points in $\mathbf{R}$  using the method in Section~\ref{sec:res_comp}.\;
For the equally spaced discrete set of frequencies $\mathcal{F}$~\eqref{eqn:F_freqs}, compute the set of densities $\mathcal{D}_\mathcal{F}$ using the rational approximants ~\eqref{eqn:Rat_approx_set}.\;
Using $\mathcal{D}_\mathcal{F}$ evaluate $U_{\bp}(\br,\omega)$ for all frequencies in $\mathcal{F}$ and all points in $\mathbf{R}$.\;
Compute the singularity-subtracted version
$U_{\bp,h}^\mathrm{s,e}(\br,\omega)$~\eqref{eq:usie} of $U_\bp$ \;
Evaluate $u^I(\br,t) = \mathcal{I}^\mathrm{e}_1(\br,t) + \mathcal{I}^\mathrm{e}_2(\br,t)$ using by the quadrature rules discussed in Section~\ref{sec:inv_four_tran} for $\mathcal{I}^\mathrm{e}_1$ and the quadrature method discussed in Section~\ref{pre-asym} along with the asymptotic expansion~\eqref{eqn:usem-e} for $\mathcal{I}^\mathrm{e}_2$.
\caption{Singularity subtraction-enabled FTH
algorithm (FTH-SS)} \label{alg:FTH_with_sinsub}
\end{algorithm}
\vspace{0.3cm}

\section{Numerical results}\label{sec:Num_Res}
This section presents a variety of numerical illustrations of the
FTH-SS algorithm and its various elements, including illustrations of
the exponential convergence of the asymptotic
expansion~\eqref{eqn:I2_asym_exp}, demonstrations of the ability of
the IE method to regularize Fourier-transform integrals via
singularity subtraction, as well as applications of the overall FTH-SS
method in challenging configurations containing trapping obstacles. The
examples considered include test cases for both open-arc and
closed-curve scatterers, such as those shown in
Figure~\ref{fig:scatterers}, along with a closed circular geometry and
the whispering-gallery configuration depicted in
Figure~\ref{fig:whisper_sem_comp}.

\begin{figure}[H]
    \centering
    \includegraphics[scale=0.3]{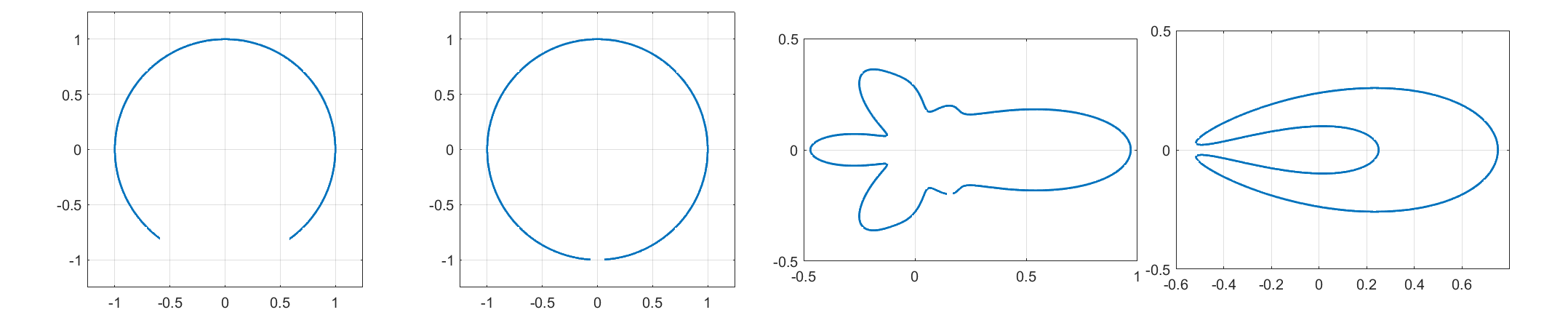}
    \caption{Scatterers used in some of the examples presented in this
      section. From left to right: large-aperture circle
      ($1.25$-radian aperture), small-aperture circle ($0.125$-radian
      aperture), open rocket-shaped cavity, closed-curve cavity.}
    \label{fig:scatterers}
\end{figure}

The first two panels in Figure~\ref{fig:scatterers} consist of circular
arcs of radius 1 with apertures spanning $1.25$ radians and $0.125$
radians, respectively.  The rocket-shaped scatterer in the third panel, is given
by the parametrization $\gamma: [0,2\pi] \to \R^2$ given by
$\gamma(s) = (C(s) \cos(s), C(s)\sin(s))$ where
\begin{equation*}
  C(s) = 0.35 + 0.1 \cos(s) + 0.12\cos(2s) + 0.15\cos(3s) + 0.1\cos(4s) + 0.1\cos(6s) + 0.05\cos(8s). 
\end{equation*}
The full (closed) rocket boundary is produced when the full span
$0\leq s\leq 2\pi$ is used, while the rocket-with-opening displayed in
Figure~\ref{fig:scatterers} is obtained by restricting the
parametrization to the complement of the interval
$ 5.338\leq s\leq 5.427$.  The closed-curve cavity presented in the
fourth panel, finally, coincides with the one given
in~\cite[Fig. 1]{dominguez2024nystrom}.

Two incident fields $u^{\mathrm{inc}}$ (see
equation~\eqref{eqn:wave_bc}), are considered in this section, namely
\begin{equation}\label{eqn:gaussian_incidence}
  u^{\mathrm{inc}} = u^{\mathrm{inc}}_1(\br,t) = \text{Fourier transform of } \ e^{-\frac{(\omega - \omega_0)^2}{\sigma^2}}e^{\mathrm{i}\kappa(\omega) \bp \cdot \br},
\end{equation}
for various choices of the parameters $\bp$, $\omega_0$, and $\sigma$; and,
\begin{equation}\label{eqn:uinc_2}
u^{\mathrm{inc}} = u^{\mathrm{inc}}_2(\br,t) = \text{Fourier transform of } (1 - w(\omega, 1)) F(\br,\omega),
\end{equation}
where $w(\omega,1)$ is the window function defined
in~\eqref{eqn:window_func} with $H=1$, and, where, using the chirp
function
\begin{equation}\label{eqn:chirp_incidence}
  a(t) = \sin\left(g(t) + \frac{1}{4000}g^2(t)\right) \quad \text{with} \quad g(t) = 4t + 6\cos\left(\frac{t}{\sqrt{12}}\right),
\end{equation}
together with the window function~\eqref{eqn:window_func} (for given
values of $s$, $H$, and $\bp$), the
function $F(\br,\omega)$ is given by the Fourier transform of
\begin{equation}\label{eqn:td_chirp}
  f(\br,t) = w(t - s - \bp \cdot \br; H) a(t - \bp \cdot \br)
\end{equation}
with respect to $t$. The values of the various parameters utilized in
each example are specified within the corresponding description. In
line with Remark~\ref{rmk:zer_freq}, the quantity $(1 - w(\omega,1))$
in~\eqref{eqn:uinc_2} is employed to eliminate zero frequency content.

In all cases the reference solution $u_\mathrm{ref}(\br,t)$ was
equated to the FTH-SS solution $u^I(\br,t)$ (equation~\eqref{eqn:uab})
with sufficiently fine discretizations and a sufficiently large
frequency interval $I$. In particular, these computations incorporate
the singularity subtraction method, but do not include the asymptotic
expansion~\eqref{eqn:I2_asym_exp}. Convergence of $u^I(\br,t)$ to
near machine precision was assessed by refining the boundary integral
equation discretization, increasing the size of the frequency interval
$I$, and enlarging the number of integration frequencies used.

This section is organized as follows. Section~\ref{sec:conv_regularized_int} illustrates the
overall impact of the singularity-subtraction technique, while
Section~\ref{sec:adaptivealgdemo} demonstrates the effectiveness of
the IE algorithm (Section~\ref{sec:realleftscal}) in capturing the
resonances associated with the solution $u$ generated by a given
incident field. The regularizing properties of the SS approach
(Section~\ref{sec:RS}), are then examined in
Section~\ref{sec:SingSubEx}, followed by  a numerical illustration of the validity of the
assumptions~\eqref{eq:pole-number} and~\eqref{eq:residue-size} in
Theorem~\ref{thm:sing_exp} for a strongly trapping scatterer. Finally, Sections~\ref{sec:Time_Domain_Res_Build} and~\ref{sec:Whisper} present
applications of the complete FTH-SS method across a range of
illustrative scenarios, including numerical validation of the
singularity expansion~\eqref{eqn:usem} in accurately representing the scattered field at late times.

\begin{figure}[H]
  \centering
  \includegraphics[width=\linewidth]{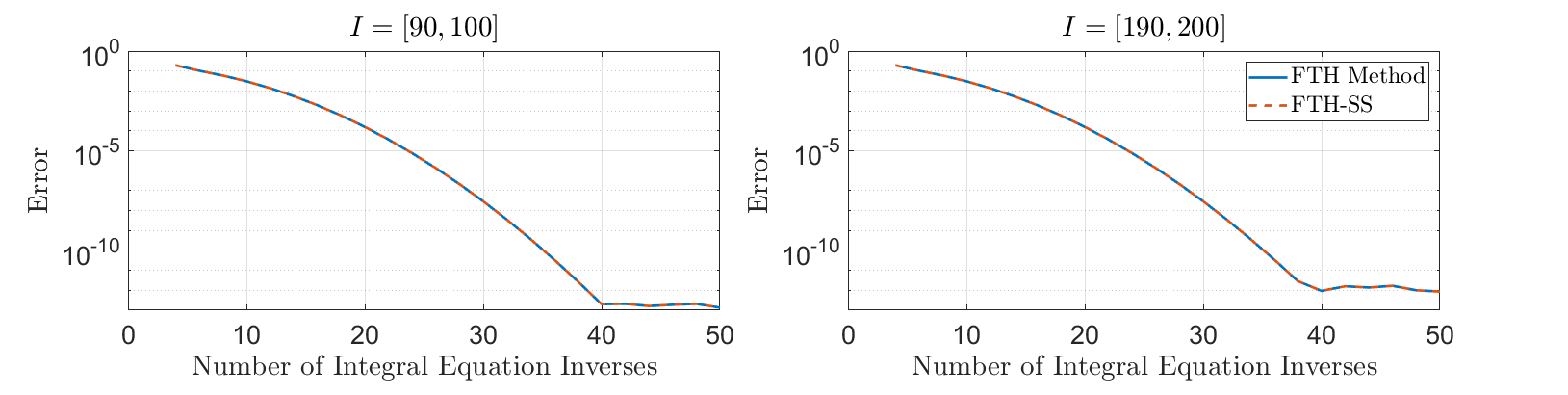}
  \caption{FTH and FTH-SS solution errors for the closed-circle
    scatterer, as a function of the number of integral equation
    inverses used. Due to its non-trapping nature, this scatterer does
    not generate complex resonances near the real axis. Consequently,
    the FTH-SS method performed no actual singularity subtraction, and
    its results coincide with those of the standard FTH method in this
    case.}
  \label{fig:FTH_Circle}
\end{figure}

%%%%%%%%%%%%%%%%%%%%%%%%%%%%%%%%%%%%%%%%%%%%%%%%%%%%%%%%%%%%%%%%%%%%%%%%%%%%%%%%%%%%%%%%%
\subsection{Comparison of FTH and
  FTH-SS}\label{sec:conv_regularized_int}
This section compares the character of the FTH and the FTH-SS methods
(Sections~\ref{sec:abl_rev} and~\ref{sec:Num_Imp}) in the contexts of
trapping and non-trapping obstacles. As expected, the FTH-SS method
provides significant advantages for trapping obstacles, but it
essentially coincides with the FTH method for non-trapping
obstacles. The test cases considered use the Gaussian incident
field~\eqref{eqn:gaussian_incidence} with incident direction
$\bp = (1,1)$ and $\sigma = 0.679$. Two center frequencies are
considered, namely $w_0 = 95$ and $w_0 = 195$; with these selections
the corresponding Gaussian functions vanish up to machine precision
outside the frequency intervals $I = [90,100]$ and $I =[190,200]$,
respectively. In each case, reference solutions
$u_\mathrm{ref}(\br,t)$ were obtained as detailed in the introduction
to Section~\ref{sec:Num_Res}.

The first example considers scattering by a closed circular obstacle
of radius 1, centered at the origin. This is a non-trapping obstacle
and therefore does not produce complex resonances near the real
axis. A reference solution $u_\mathrm{ref}(\br,t)$ at the point
$\br = (0,-1.3)$ and at $500$ equispaced times in the interval
$[0,20]$ was used for evaluation of
errors. Figure~\ref{fig:FTH_Circle} displays the maximum all-time
error for both the FTH and FTH-SS methods as functions of the number
of integral-equation solves (equivalently, the number of frequencies)
used to compute the inverse Fourier transform~\eqref{eqn:ukw}. The
left and right panels in the figure correspond to incident fields with
non-vanishing frequency content supported in the intervals
$I = [90,100]$ and $I = [190,200]$, respectively. Following the
recommendation in Section~\ref{adaptive}, a non-adaptive version of
the IE algorithm was used to produce this figure, since the number of
frequencies $J$---which reaches up to $J = 50$ in the
figure---satisfies the condition $J < 200$ and therefore does not
trigger adaptivity. In particular, the same number of integral
equation inverses was used by the FTH and FTH-SS method in this case:
as no complex resonances were found by the $IE$ algorithm, the FTH and
FTH-SS methods actually coincide in this case. As demonstrated by the
second example in this section, the situation differs markedly in the
case of scattering by a trapping obstacle: in such cases, the FTH-SS
method can significantly outperform the FTH method.

\begin{figure}[h]
    \centering
    \includegraphics[width=\linewidth]{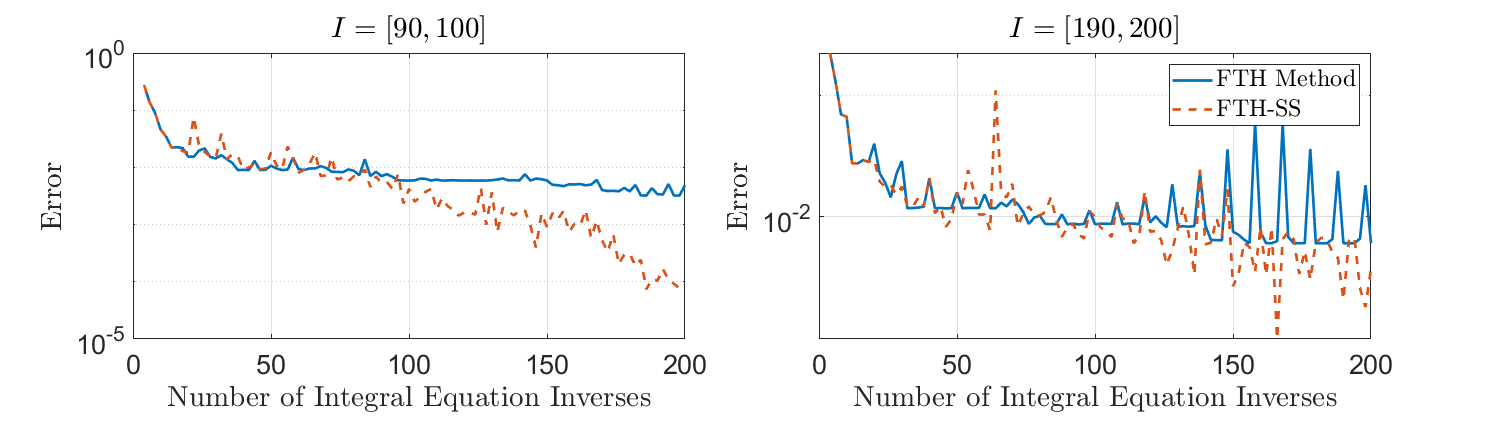}
    \caption{FTH and FTH-SS solution errors for the open circle
      scatterer depicted on the leftmost panel in
      Figure~\ref{fig:scatterers}, as a function of the number of
      integral equation inverses used. Since a number $J < 200$ of
      inverses was used for these test cases, the FTH-SS method did
      not trigger the IE-algorithm's adaptivity.}
    \label{fig:FTH_comp_no_adapt}
  \end{figure}
  
\begin{figure}[h]
    \centering
    \includegraphics[width=\linewidth]{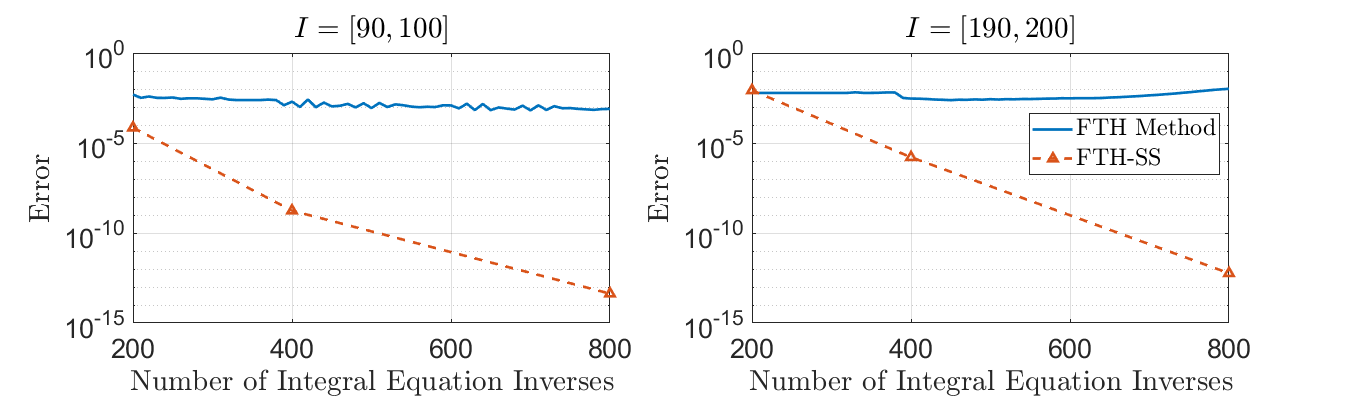}
    \caption{Same as Figure~\ref{fig:FTH_comp_no_adapt} but using a
      different range of numbers of integral equation inverses---for
      which the IE-algorithm's adaptivity was triggered. The triangles
      mark the errors corresponding to three different numbers of
      inverses actually used by FTH-SS method---which are determined
      by each one of the three adaptivity levels triggered in the
      adaptive IE method.}
    \label{fig:FTH_comp_adapt}
\end{figure}

The second example in this section concerns scattering by the open
circle shown in the leftmost panel of Figure~\ref{fig:scatterers},
using the same two incident fields as in the first example. The
absolute value of the real part of the total field corresponding to the first incident
field---associated with the frequency interval $[90, 100]$---is
displayed in Figure~\ref{fig:circ_large_opening}. A reference solution
$u_\mathrm{ref}(\br,t)$ at the point $\br = (0,0)$ and at $500$
equispaced times in the interval $[0,120]$ was used for error
evaluation.  The RE algorithm produced $130$ and $192$ complex
resonances in the box $\mathcal{M}^I_h$ with $h = 0.3$ for the
intervals $I = [90,100]$ and $I = [190,200]$, respectively. Maximum
solution errors---evaluated at $\br = (0, 0)$ over 500 equally spaced
time points in the interval $[0, 120]$---for incident fields with
frequency content in the intervals $I = [90, 100]$ and
$I = [190, 200]$ are shown on the left and right panels, respectively,
in both Figures~\ref{fig:FTH_comp_no_adapt}
and~\ref{fig:FTH_comp_adapt}.

\begin{figure}[h]
    \centering
    \includegraphics[width=\linewidth]{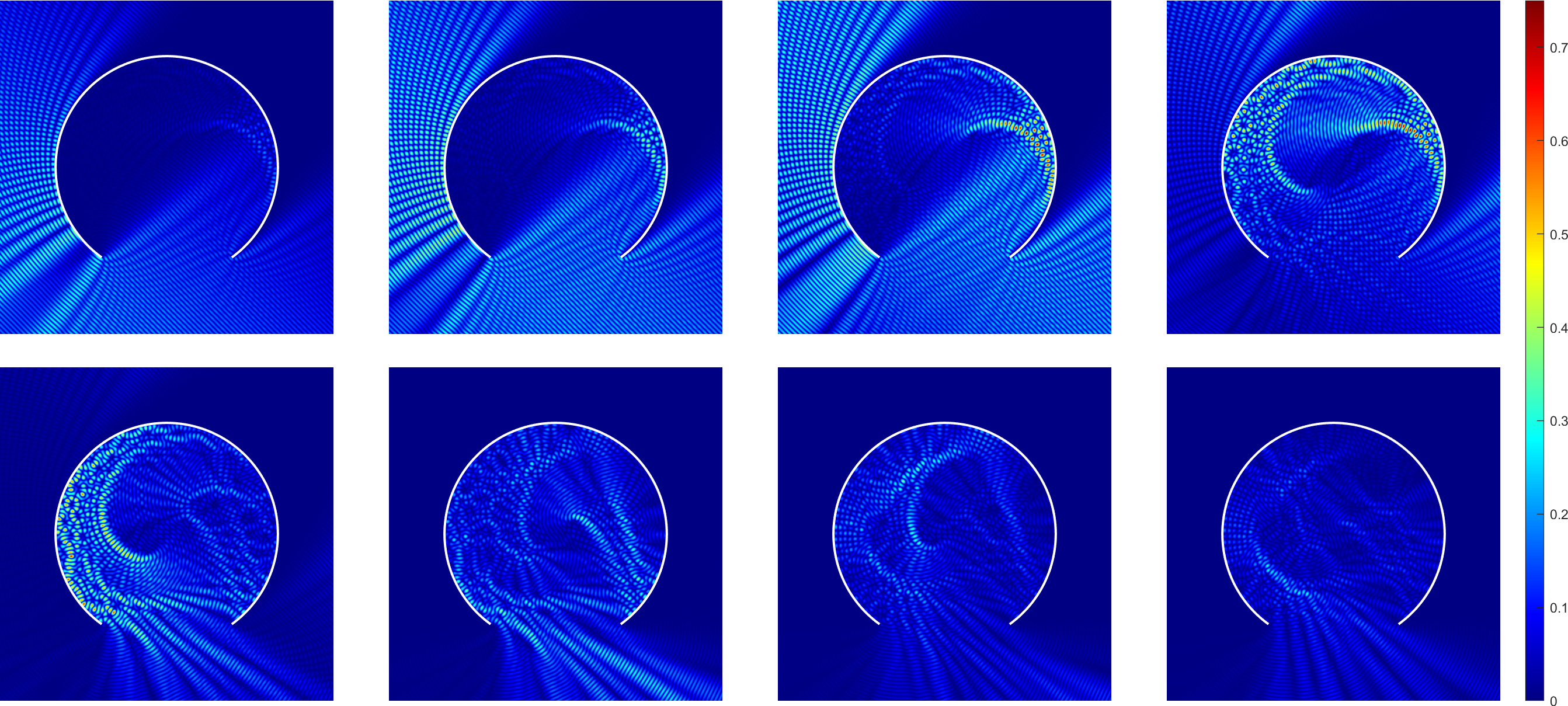}
    \caption{Temporal evolution of scattering from the large-aperture circle,
      shown at an increasing sequence of times from left to right and
      top to bottom. Each panel shows the absolute value of the real
      part of the total field.}
    \label{fig:circ_large_opening}
\end{figure}

As in the first example of this section, the results in
Figure~\ref{fig:FTH_comp_no_adapt} were obtained using the IE
algorithm without adaptivity, since only frequency numbers $J < 200$
were used in this case. In particular, this figure demonstrates that
even without the adaptive version of the IE algorithm, the FTH-SS
method offers a significant advantage. As shown in
Figure~\ref{fig:FTH_comp_adapt}, an even greater improvement is
achieved when frequency numbers $J \ge 200$ are used in combination with
the fully adaptive IE algorithm. In Figure~\ref{fig:FTH_comp_adapt}
only three numbers-of-inverses, $200$, $400$ and $800$, marked as
triangular error points, were used for the FTH-SS method. These values
correspond to splitting the interval $I$ into $1$, $2$ and $4$
subintervals respectively, as part of the adaptive IE algorithm with
initial input $J = 200$ in the interval $I$. As additional reference
points we report that for the intervals $I = [90,100]$ and
$I = [190,200]$ and utilizing as many as $10,000$ inverses, the FTH
method (without singularity subtraction) produced solutions with
errors of $1.2 \cdot 10^{-8}$ and $5.0 \cdot 10^{-5}$, respectively.

%%%%%%%%%%%%%%%%%%%%%%%%%%%%%%%%%%%%%%%%%%%%%%%%%%%%%%%%%%%%%%%%%%%%%%%%%%%%%%%%%%%%%%%%%
\begin{figure}[h!]
    \centering
    \includegraphics[width=\linewidth]{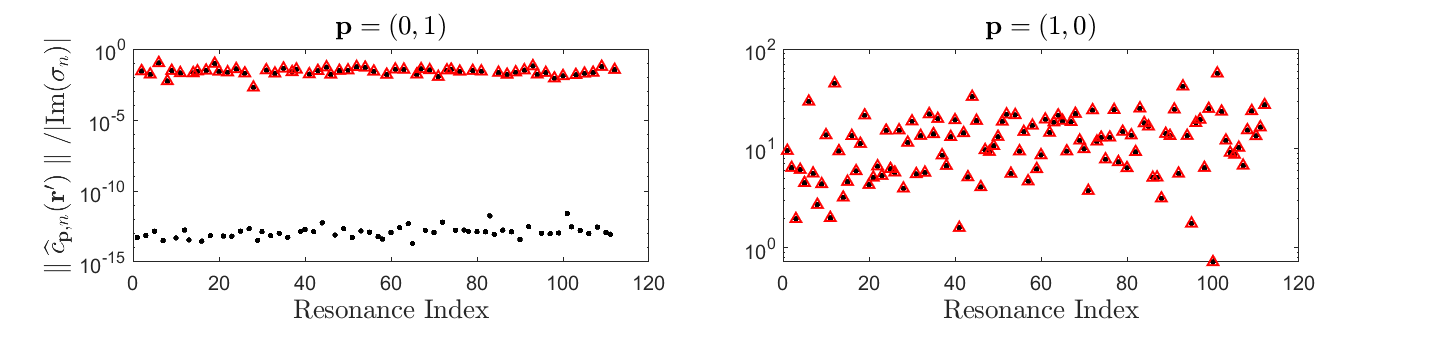}
    \caption{Value of the metric~\eqref{eq:relevance} at
      $\omega = \Re(\sigma_n)$ for complex resonances produced by the
      RE and IE algorithms (shown black dots and red triangles,
      respectively) for the scattering problem described in the text
      for incidence directions pointing into the circle opening (left
      panel) and ``tangential'' to the circle opening (right
      panel). As indicated in the text, all the resonances
      ``relevant'' for the corresponding time-domain problem are
      obtained by the IE algorithm with near machine-precision
      accuracy.}
    \label{fig:not_found}
\end{figure}

\subsection{Adaptive IE algorithm singularity-capturing character}\label{sec:adaptivealgdemo}
The examples presented in this section demonstrate, as indicated in Remark~\ref{rmk:phIe_is_numer}, the ability of the
IE method (Algorithm~\ref{alg:real-line-adaptive}), with a given
incident field $B_\bp$~\eqref{eq:BP}, to reliably capture all complex resonances
which are relevant to the time domain problem, up to the level of error inherent in the numerical evaluation of singularities and residues
themselves. As discussed in the introduction to
Section~\ref{sec:realleftscal}, a useful metric on the relevance of
a resonance pair $(\sigma_n,\widehat{c}_{\bp, n})$ is given by the
$L^2$ norm
\begin{equation}
  \label{eq:relevance}
  \parallel\widehat{c}_{\bp,n}(\br')\parallel_{L^2(\Gamma)} / |\omega - \sigma_n|
\end{equation}
of the contribution $\widehat{c}_{\bp,n}(\br')/(\omega - \sigma_n)$ of
the pair to the integral density and, thus,
via~\eqref{eqn:res}--\eqref{eqn:spat_res_dens_bound}, to the scattered
field $u$. To illustrate that the IE algorithm captures those resonances for which the metric~\eqref{eq:relevance} is not small Algorithm~\ref{alg:real-line-adaptive} was applied to obtain the complex resonances $P^{I,\mathrm{e}}_h$~\eqref{eqn:inc_res} in $\mathcal{M}^I_h$ with $I = [30,50]$, and $h = 0.2$ using the open circle scatterer displayed on the left most panel of Figure~\ref{fig:scatterers}. For these demonstrations two different
incident directions were used, namely, incidence normal to the opening
($\bp = (0,1)$), and incidence at a $45^\circ$ angle from the opening
($\bp = (1,0)$). The tolerance $\varepsilon_\mathrm{tol} = 10^{-10}$
was used for the AAA portion of the computations, and in each case the
integral operator $S_\omega^N$~\eqref{eq:discr_opers} was discretized
to an error level matching the tolerance.
\begin{figure}[htb]
    \centering
    \includegraphics[width=\linewidth]{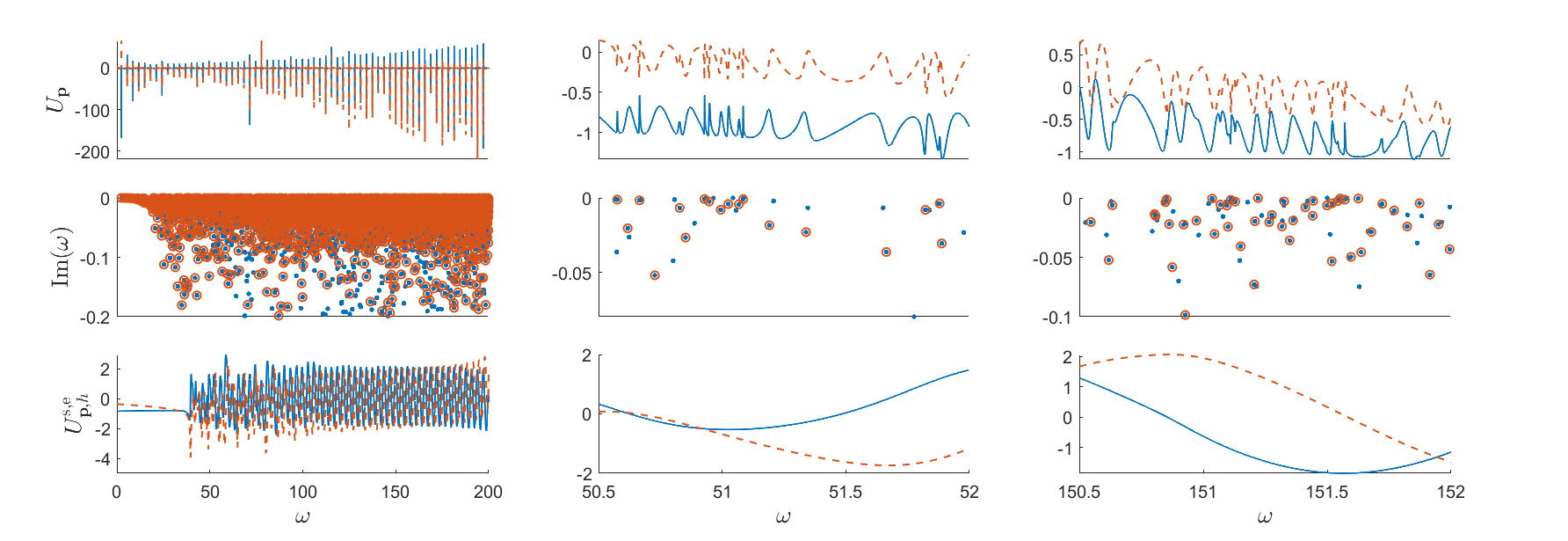}
    \caption{Top and bottom rows: Solutions $U_{\bp}(\br,\omega)$ and
      $U_{\bp,h}^{\mathrm{s,e}}(\br,\omega)$, respectively (see
      equations~\eqref{eq:Up} and~\eqref{eq:usie}), for the
      small-opening circular-arc scatterer shown in
      Figure~\ref{fig:scatterers}, over three distinct frequency
      ranges (see also Remark~\ref{rmk:using_true_pols}).  The solid
      blue and dashed orange curves represent the real and imaginary
      parts of $U_{\bp}$ (top row) and $U_{\bp,h}^{\mathrm{s,e}}$
      (bottom row).  Middle row: Resonance poles obtained from the RE
      and IE algorithms, displayed as blue dots and orange circles,
      respectively.  A total of $4833$ IE poles with real parts in
      $[0,200]$ and imaginary parts in $[-0.2,0]$ were computed and
      used to construct the regularized solution
      $U_{\bp,h}^{\mathrm{s,e}}$.}
    \label{fig:subtraction_regularizing_example}
\end{figure}

Noting that the largest value of the relevance
metric~\eqref{eq:relevance} for $\omega \in I$ is achieved at
$\omega = \Re(\sigma_n)$, for each complex resonance
$\sigma_n \in P^I_h$ computed by the RE algorithm (which, as
discussed in Section~\ref{sec:adaptiveres}, is expected to produce with high accuracy all
the resonance pairs in the box $\mathcal{M}^I_h$), the quantity
$\parallel\widehat{c}_{\bp,n}(\br')\parallel_{L^2(\Gamma)} /
|\Im(\sigma_n)|$ is plotted as a black dot in each panel of
Figure~\ref{fig:not_found} with resonances ordered according
increasing real part. The first and second panels in
Figure~\ref{fig:not_found}, which were obtained for the incidence
directions $\bp = (0,1)$ and $\bp = (1,0)$, respectively, (i.e.,
pointing into the circle opening and ``tangential'' to the circle
opening, respectively) also display a red triangle for resonances obtained by the IE algorithm. As shown in the in the figure,
Algorithm~\ref{alg:real-line-adaptive} captures all resonances whose
relevance metric is not small.

%%%%%%%%%%%%%%%%%%%%%%%%%%%%%%%%%%%%%%%%%%%%%%%%%%%%%%%%%%%%%%%%%%%%%%%%%%%%%%%%

\subsection{Singularity-subtraction regularization effect}\label{sec:SingSubEx}
In order to study the regularizing effect that results from the
singularity subtraction method proposed in
Section~\ref{sec:SingSubIntrod}, in what follows we consider the
problem of scattering by the open circle displayed on the second panel
of Figure~\ref{fig:scatterers}, with boundary conditions given by a
plane wave with incident direction $\bp = (0,1)$ (normal to the
opening). The left, center and right panels in the top row of images
in Figure~\ref{fig:subtraction_regularizing_example} display the real
(solid) and imaginary (dotted) parts of
$U_{\bp}(\br,\omega)$~\eqref{eq:Up} at the point $\br = (0,0)$, for
$\omega$ in the ranges $[0,200]$, $[50.5,52],$ and $[150.5,152]$
respectively. In all cases a large number of sharp spikes in
$U_{\bp}(\br,\omega)$ can be seen. The second row displays
corresponding poles produced by the RE and IE algorithms as
blue dots and orange circles, respectively. (Per
Remark~\ref{rmk:Interval_Selection_Alg}, to achieve an efficient
computation, Algorithm~\ref{alg:real-line-adaptive} was applied to a
set of 20 intervals of length $5$ and $50$ intervals of length $2$ in
the ranges $[0,100]$ and $[100,200]$ respectively.) Comparison of the
first and second rows in
Figure~\ref{fig:subtraction_regularizing_example} reveals a clear
correspondence between the relevant complex resonances and the spikes
in the solution. Finally, the bottom row of
Figure~\ref{fig:subtraction_regularizing_example} presents the
singularity-subtracted field $U_{\bp,h}^\mathrm{s,e}(\br,\omega)$
defined in~\eqref{eq:usie}. This result clearly illustrates the
regularizing effect of the subtraction procedure and explains how the
use of the regularized field $U_{\bp,h}^\mathrm{s}$ significantly
improves the convergence of the FTH Fourier transform method employed
for the evaluation of $u^{I}(\br,t)$, as observed in
Section~\ref{sec:conv_regularized_int}.

\begin{figure}[h]
    \centering
    \includegraphics[width=\linewidth]{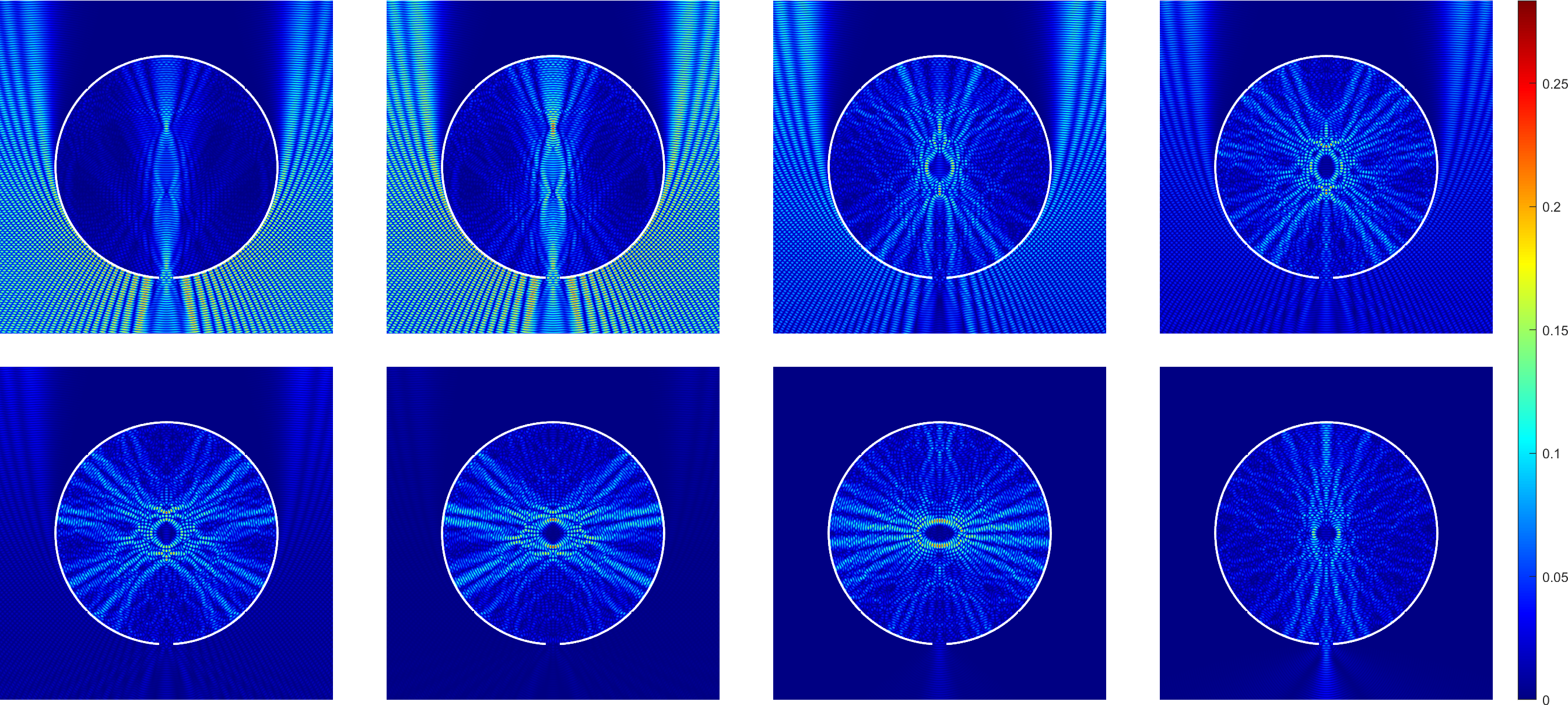}
    \caption{Temporal evolution of scattering from the open circle,
      shown at an increasing sequence of times from left to right and
      top to bottom. Each panel shows the absolute value of the real
      part of the total field. Distinct resonant structures are
      visible both during and after excitation by the incident field.}
    \label{fig:circ_so}
\end{figure}

For reference, Figure~\ref{fig:circ_so} displays time-dependent
scattered fields related to the configuration used in the examples
discussed earlier in this section. These fields were generated using
the Gaussian incident field~\eqref{eqn:gaussian_incidence} with
parameters $\omega_0 = 147.5$, $\sigma^2 = 0.1696$ and $\bp = (0,1)$,
with $I = [145,150]$, and with subtraction of the resonances produced
by the adaptive IE algorithm within the box $\mathcal{M}^I_h$ for
$h = 0.2$. The figure presents snapshots of the absolute value of the
real part of the total field, revealing several time-varying resonant
structures and illustrating the slow decay of the scattered field
after the incident field vanishes on the scattering boundary.

\subsection{Numerical verification of the pole and residue assumptions
  in Theorem~\ref{thm:sing_exp}}\label{sec:bounded_residues}
This section provides a numerical illustration of the validity of the
assumptions~\eqref{eq:pole-number} and~\eqref{eq:residue-size} in
Theorem~\ref{thm:sing_exp} concerning the number
$N^I_h = N^I_h(W_1,W_2)$ of poles~\eqref{eq:nih} and the boundedness
of the corresponding residues, respectively, for the strongly trapping
small-aperture circle shown in the second panel of
Figure~\ref{fig:scatterers} (see Remark~\ref{rmk:growth-assump}). For
this illustration, the corresponding poles $P^I_h$ were obtained to
high accuracy, for the domain $\mathcal{M}_h^I$ with $I = [0,200] $
and $ h = 0.2$, by employing the RE algorithm described in
Section~\ref{sec:adaptiveres}.

\begin{figure}[h]
    \centering
    \includegraphics[width=\linewidth]{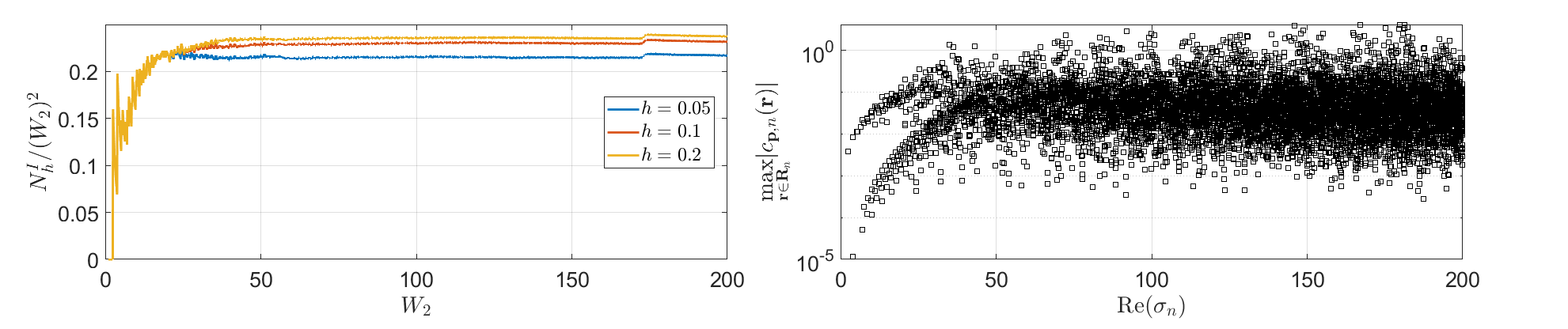}
    \caption{Numerical illustration of the assumptions underlying
      Theorem~\ref{thm:sing_exp} for the strongly trapping geometry
      depicted in the second panel of
      Figure~\ref{fig:scatterers}. Left: Scaled resonance counts
      $N^I_h / (W_2)^2$ for $I = [0, W_2]$ and three values of the
      singularity-box depth $h$, demonstrating the validity of the
      bound~\eqref{eq:pole-number} for the scattering configuration
      considered. Right: Maximum residue magnitude
      $\max_{\br \in \mathbf{R_n}} |c_{\bp,n}(\br)|$ for
      $I = [0,200]$, $h = 0.2$, and $\bp = (1,1)$, illustrating the
      boundedness of the residues in agreement
      with~\eqref{eq:residue-size}.}
    \label{fig:bnd_res}
\end{figure}

The left panel of Figure~\ref{fig:bnd_res} displays the quantity
$ N^I_h / (W_2)^2 $ for $ I = [0, W_2] $ and three values of $ h $:
$ h = 0.05 $, $ h = 0.1 $, and $ h = 0.2 $. Because complex resonances
are symmetric about the imaginary frequency axis~\cite[Corollary
7.12]{taylorbook}, these results confirm that the
bound~\eqref{eq:pole-number} holds for the highly trapping
small-aperture circle considered. Using the incident direction
$\bp = (1,1)$,  the
right panel of Figure~\ref{fig:bnd_res} displays the quantity
\begin{equation*}
  \max_{\br \in \mathbf{R}_n} |c_{\bp,n}(\br)|, \quad 1 \le n \le N^I_h \quad \text{with} \quad I = [0,200]\quad\mbox{and}  \quad h = 0.2,
\end{equation*}
where $ \mathbf{R}_n$ denotes a uniform-grid discretization of the
square domain $ [-1.5, 1.5]^2$ using six points per wavelength
$\lambda_n = 2\pi / \Re(\sigma_n)$ in each direction; clearly, the
residues $ |c_{\bp,n}(\br)| $ remain bounded as $W_2 \to +\infty$, (and by
symmetry, as $\mu \to +\infty$~\eqref{mu}), in line with the
assumption~\eqref{eq:residue-size} in Theorem~\ref{thm:sing_exp}.

%%%%%%%%%%%%%%%%%%%%%%%%%%%%%%%%%%%%%%%%%%%%%%%%%%%%%%%%%%%%%%%%%%%%%%%%%%%%%%%%%%%%%%%%%

\begin{figure}[h]
  \centering
  \includegraphics[width=\linewidth]{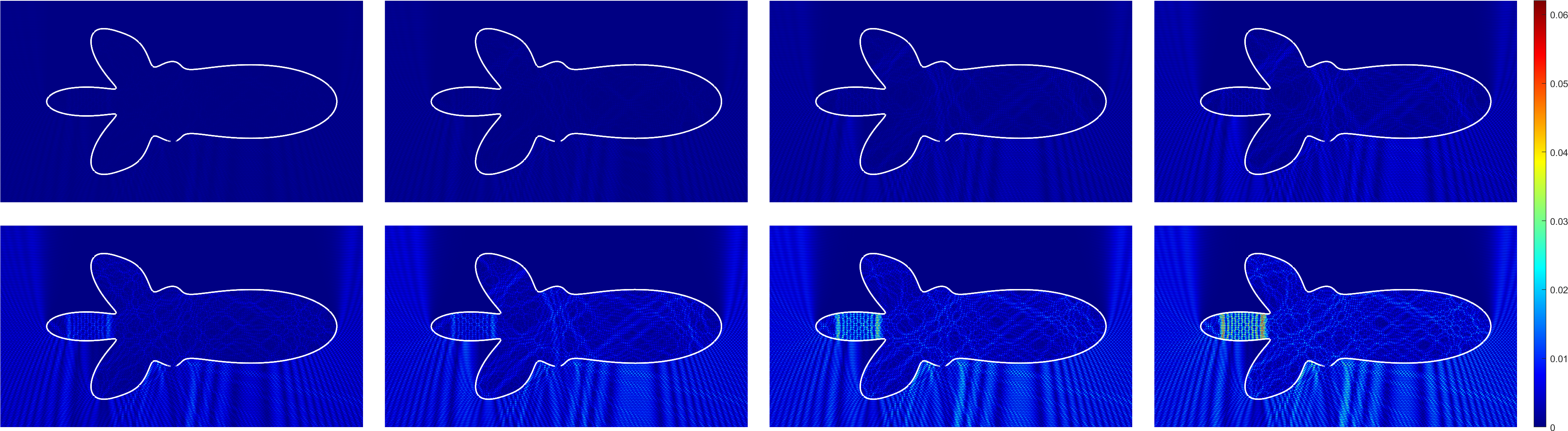}
  \caption{Same as in Figure~\ref{fig:circ_so} but for the
    rocket-shaped structure depicted in Figure~\ref{fig:scatterers}. A
    pronounced localized resonance becomes clearly visible as it
    develops in the left finger of the rocket structure.}
  \label{fig:rocket excitation}
\end{figure}

\subsection{Time-domain resonance build up}\label{sec:Time_Domain_Res_Build}
The FTH-SS algorithm’s ability to deliver accurate solutions over long
time intervals makes it well-suited for studying the buildup of
resonances in highly trapping cavities. The example in this section demonstrates the time domain
excitation of a localized resonance in the rocket-shaped scatterer
depicted in the third panel in Figure~\ref{fig:scatterers}, using the
Gaussian incident field~\eqref{eqn:gaussian_incidence} with
$\omega_0 \approx 399.969$, (a selection that corresponds to the
eigenfunction displayed in~\cite[Fig. 8]{bst}), $\sigma^2 = 0.0011$,
$I = [399.7695, 400.1695]$, and $h = 0.01$. Figure~\ref{fig:rocket excitation}
displays the the absolute value of the real part of the total field
for various times. A strong localized resonance is seen to build up in
the rocket's left finger.

\begin{figure}[h]
    \centering
    \includegraphics[width=\linewidth]{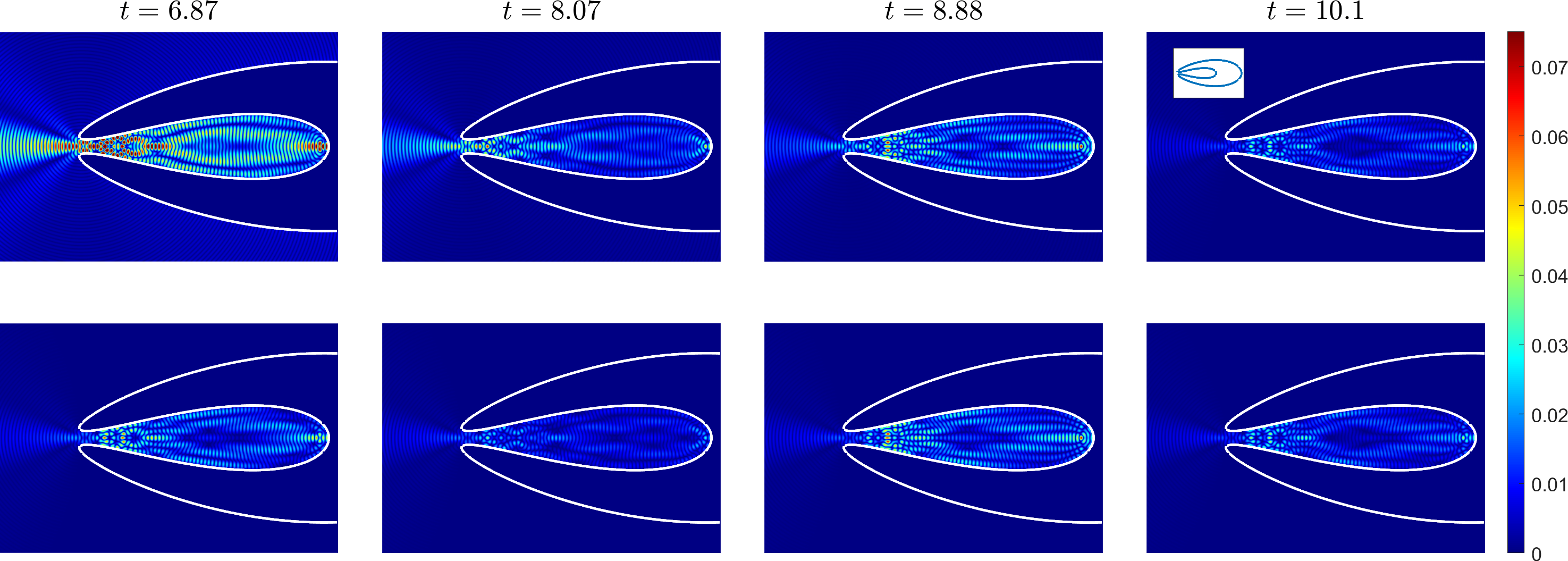}
    \caption{Top row: Temporal evolution of fields scattered from a
      highly trapping closed curve at an increasing sequence of
      times. Bottom row. Corresponding values of the singularity
      expansion $\mathcal{E}^{I,\mathrm{e}}_h$~\eqref{eqn:usem-e}. For reference the closed scatterer is displayed on the top-right panel.}
    \label{fig:closed_scatterer}
\end{figure}

%%%%%%%%%%%%%%%%%%%%%%%%%%%%%%%%%%%%%%%%%%%%%%%%%%%%%%%%%%%%%%%%%%%%%%%%%%%%%%%%%%%%%%%%%
\subsection{Asymptotic validity of the singularity expansion}\label{sec:Whisper}
This section presents a variety of numerical results illustrating the
discussion in Section~\ref{sec:SEM}, with a focus on the asymptotic
validity of the singularity expansion~\eqref{eqn:usem}, even for
highly trapping geometries. To this end we consider scattering
problems for each one of the scatterers depicted in
Figure~\ref{fig:scatterers} as well as a whispering-gallery structure
depicted in Figures~\ref{fig:whisper_sem_comp}
and~\ref{fig:mult_scat}.

\begin{figure}[h]
    \centering
    \includegraphics[width=\linewidth]{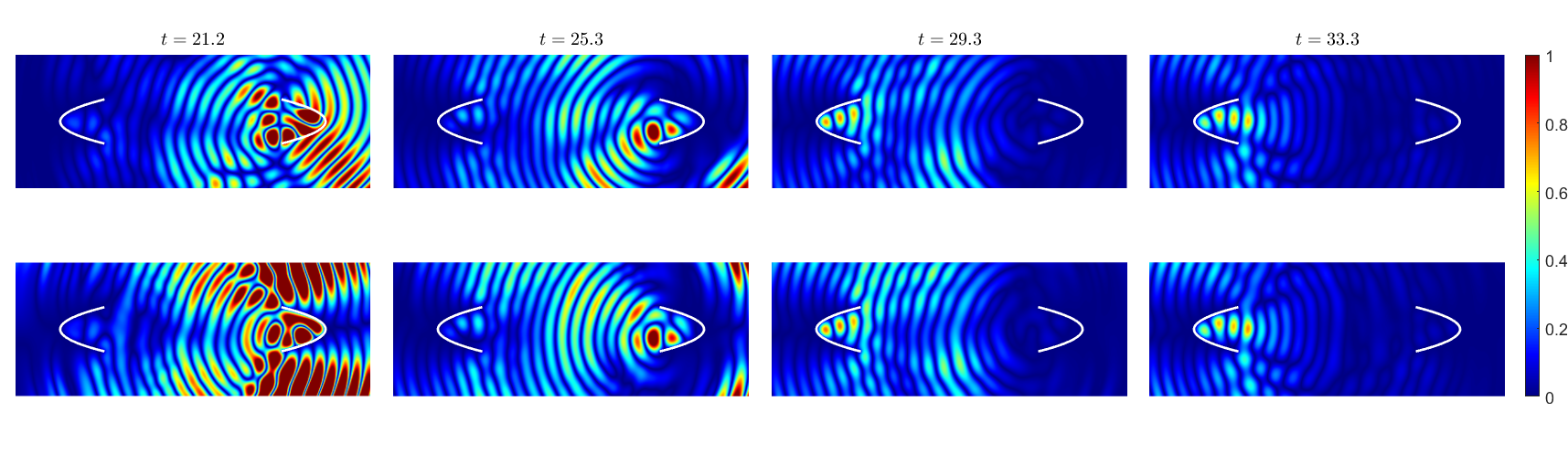}
    \caption{Same as Figure~\ref{fig:closed_scatterer} but for a
      whispering-gallery scattering structure. In agreement with the
      analysis in Section~\ref{sec:I2eval}, the singularity expansion
      suffers from significant errors before the $t = 25$
      incident-field extinction time, and the error decrease rapidly
      at later times.}
    \label{fig:whisper_sem_comp}
\end{figure}

The first example concerns scattering of the incident field defined in
equation~\eqref{eqn:gaussian_incidence}, with parameters
$\bp = (1, 0)$, $\sigma^2 = 0.2443$, and $\omega_0 = 300$, by the
closed scatterer depicted in the fourth panel of
Figure~\ref{fig:scatterers}. Both the scattered field
$u(\br,t)$~\eqref{eqn:uab} and a corresponding singularity expansion
$\mathcal{E}^{I,\mathrm{e}}_h$~\eqref{eqn:usem-e} are displayed, where
the latter is constructed by incorporating all incidence-excited
singularities $P^{I,\mathrm{e}}_h$~\eqref{eqn:inc_res} with
$I = [297, 303]$ and $h = 0.3$. The absolute values of the scattered field
is displayed at various times in the top row of
Figure~\ref{fig:closed_scatterer}, and the corresponding asymptotic
singularity-expansion approximations are presented in the bottom row
of that figure. Comparison of the top and bottom rows clearly
demonstrates the rapid convergence of the singularity expansion to the
true solution as time increases; see also Figure~\ref{fig:sem_err}.

\begin{figure}[h]
    \centering
    \includegraphics[width=\linewidth]{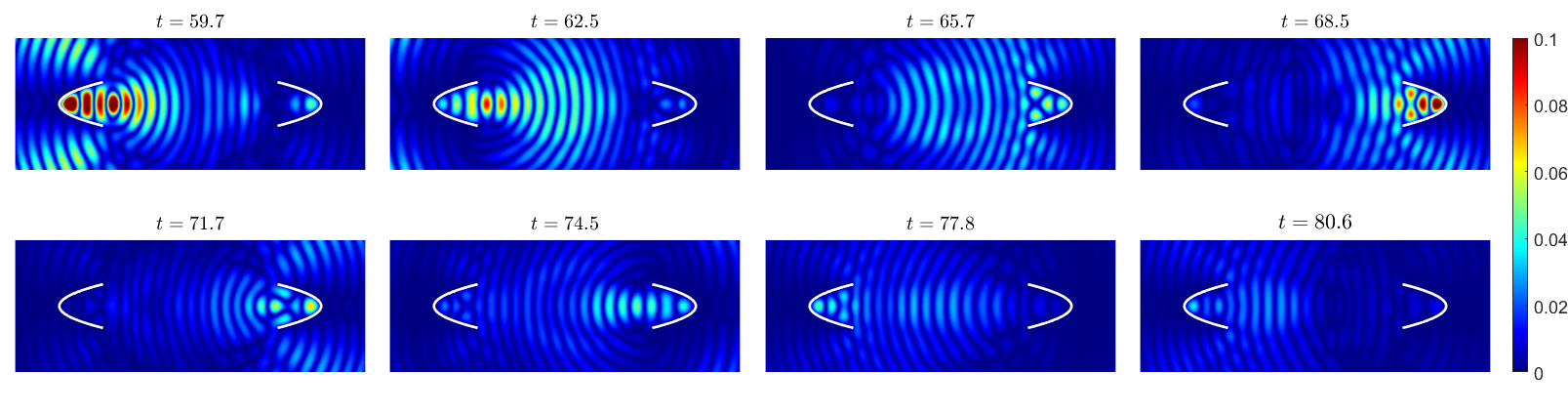}
    \caption{Late-time multiple-scattering events in the
      whispering-gallery structure. The absolute value of the
      singularity expansion approximation is displayed at several time
      points, demonstrating the expansion's ability to capture complex
      late-time scattering behavior. (Note: The color scale used here
      differs from the one employed in
      Figure~\ref{fig:whisper_sem_comp}.)  }
    \label{fig:mult_scat}
\end{figure}

A whispering-gallery example, in turn, is considered in
Figures~\ref{fig:whisper_sem_comp} and~\ref{fig:mult_scat}.  This
scattering structure, which consists of two parabolic open curves, is
illuminated by a chirp incident field~\eqref{eqn:uinc_2} with
parameters $\bp = (1,-1)$, $s = 17.5$, and $H = 7.5$. For this choice
of parameters the chirp profile~\eqref{eqn:chirp_incidence} is
supported in the time interval $10\leq t \leq 25$.  The top row of
Figure~\ref{fig:whisper_sem_comp} displays the absolute value of the
scattered field at various times, while the bottom row presents the
corresponding values of the asymptotic singularity
expansion $\mathcal{E}^I_h(\br,t)$~\eqref{eqn:usem-e} with frequency interval $I = [-40,40]$ and singularity-box depth $h = 0.3$. Comparison of these two rows of
images shows that, in agreement with the discussion in
Section~\ref{sec:I2eval}, the singularity expansion suffers from
significant errors before the time $t = T^{\mathrm{inc}} = 25$, and
that the errors decrease rapidly at later
times. Figure~\ref{fig:mult_scat} demonstrates once again the ability
of the singularity expansion to correctly capture the late multiple
scattering whispering gallery events.
 
In order to quantify the difference between the scattered field and
the asymptotic expansion~\eqref{eqn:usem} more precisely, for our final examples we consider the quantities
\begin{equation}\label{eqn:asym_err}
   \varepsilon^{I}_h(\br,t) = \Big |u(\br,t) - \mathcal{E}_h^{I}(\br,t) \Big | \quad \text{and} \quad \varepsilon^{I,\mathrm{e}}_h(\br,t) = \Big |u(\br,t) - \mathcal{E}_h^{I,\mathrm{e}}(\br,t) \Big |
\end{equation}
where exact pole and residues used in computation of $\mathcal{E}_h^{I}(\br,t)$ are approximated to high accuracy using RE method reviewed in Section~\ref{sec:adaptiveres}.

\begin{figure}[h]
    \centering
    \includegraphics[width=\linewidth]{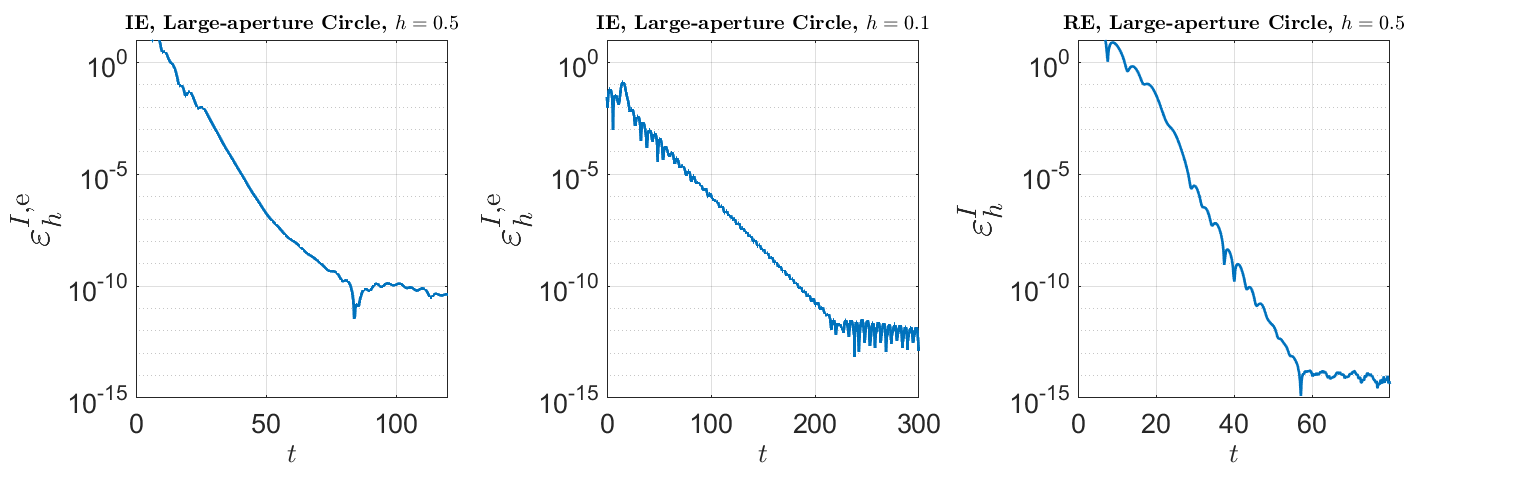}
    \caption{Errors $\varepsilon^{I}_h(\br,t)$ and
      $\varepsilon^{I,\mathrm{e}}_h(\br,t)$
      (equation~\eqref{eqn:asym_err}) at
      $\br = (0,0)$ for the large-aperture circle problem. Left and
      center panels: $\varepsilon^{I,\mathrm{e}}_h(\br,t)$ with
      $h = 0.5$ and $h = 0.1$, respectively. As discussed in the text,
      a slower error decay (resp. a reduced accuracy) is observed for
      the smaller (resp. larger) $h$ value in the context of the IE
      pole and residue evaluation method.  Right panel:
      $\varepsilon^{I}_h(\br,t)$ with $h = 0.5$---which yields near
      machine-precision accuracy by incorporating highly accurate
      poles and residues obtained by means of the RE method.}
    \label{fig:sem_h_exp}
\end{figure}

Figure~\ref{fig:sem_h_exp} displays the quantities
$\varepsilon^{I}_h(\br,t)$ and $\varepsilon^{I,\mathrm{e}}_h(\br,t)$,
as a function of $t$ and at the point $\br = (0,0)$, for the
large-aperture circle scattering problem introduced in connection with
Figure~\ref{fig:circ_large_opening}. The first and second panels
display the quantity $\varepsilon^{I,\mathrm{e}}_h(\br,t)$ with
$h = 0.5$ and $h = 0.1$, respectively. In view of Remark~\eqref{rmk:using_true_pols}, the results indicate that the asymptotic expansion~\eqref{eqn:usem} holds, however a reduction in accuracy is observed when resonances farther from the real axis are
used. This decline is attributed to the IE method’s reliance on data
at real frequencies only, which impacts upon the accuracy of the poles
and residues obtained; it should be noted, however, that, as expected,
for the smaller values of $h = 0.1$, the asymptotic expansion error
$\varepsilon^{I,\mathrm{e}}_h$ exhibits a slower decay. The third
panel in Figure~~\ref{fig:sem_h_exp} shows that, for the larger
$h = 0.5$ box depth, use of the RE method pole and residue evaluation
method results in near–machine-precision accuracy and fast
asymptotic-error decay.

Finally, Figure~\ref{fig:sem_err} shows the quantity
$\varepsilon^{I,\mathrm{e}}_h(\br,t)$ for several problems of
scattering by highly trapping obstacles considered previously in this
paper, evaluated at selected points $\br$ and plotted as functions of
time. From left to right, the panels correspond to: the small-aperture
circular cavity problem of Figure~\ref{fig:circ_so} with
$\br = (0,0)$; the rocket-shaped scatterer problem in
Figure~\ref{fig:rocket excitation} with $\br = (-0.3,0)$; the
closed-curve cavity problem of Figure~\ref{fig:closed_scatterer} with
$\br = (0,0)$; and the whispering-gallery problem of
Figure~\ref{fig:whisper_sem_comp} with $\br = (-9.7,0.1)$. The
examples in Figures~\ref{fig:sem_h_exp} and~\ref{fig:sem_err} clearly
suggest the asymptotic validity of asymptotic
expansion~\eqref{eqn:usem}, with exponentially
small asymptotic errors up to the error levels inherent in the pole
and residue evaluations.

\begin{figure}[h]
    \centering
    \includegraphics[width=\linewidth]{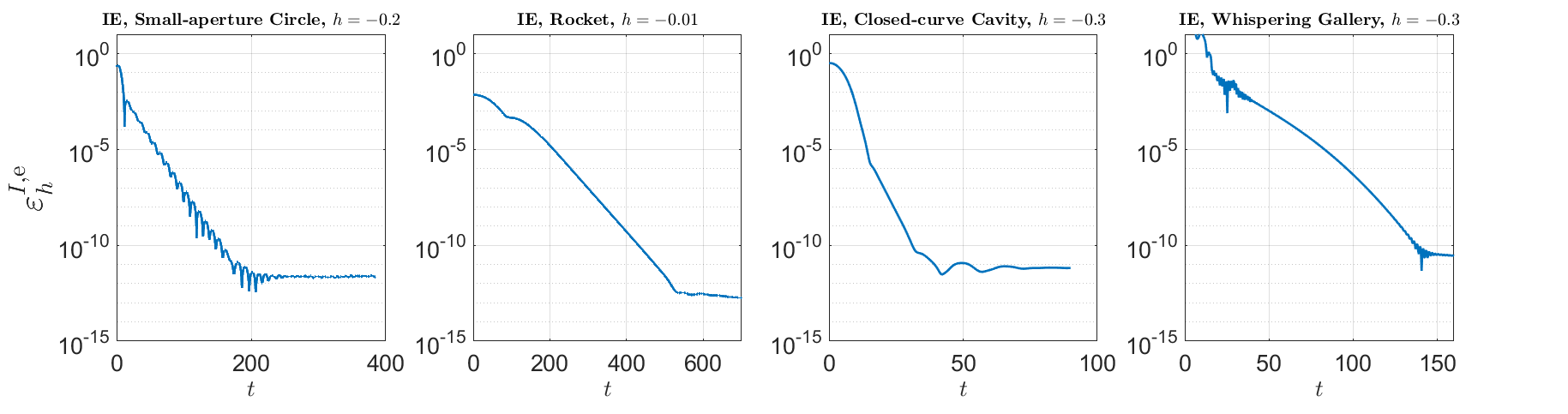}
    \caption{Errors $\varepsilon^{I,\mathrm{e}}_h(\br,t)$ at
      representative points $\br$ for various highly-trapping
      scattering configurations. Together with
      Figure~\ref{fig:sem_h_exp}, these results suggest the asymptotic
      validity of asymptotic expansion~\eqref{eqn:usem}, with exponentially small asymptotic errors up to the error levels inherent in the pole and residue
      evaluations.}
    \label{fig:sem_err}
\end{figure}

%%%%%%%%%%%%%%%%%%%%%%%%%%%%%%%%%%%%%%%%%%%%%%%%%%%%%%%%%%%%%%%%%%%%%%%%%%%%%%%%%%%%%
\section{Conclusions}

This paper has presented a singularity-subtraction technique which,
building on Fourier-transform-based methods~\cite{abl}, enables the
efficient computation of time-domain scattering from trapping
obstacles. At the core of the approach is a novel Incidence Excitation
(IE) algorithm that, using only real-frequency scattering solutions,
allows for the efficient evaluation of all complex resonances and
residues {\em relevant} to the subtraction procedure. The method is
completed by computing the fields associated with the subtracted
singularities through a combination of a simple and inexpensive
numerical scheme and a large-time asymptotic expansion of the
subtracted singularity terms. Numerical experiments show that a
related and well-known ``singularity expansion'' generally provides an
accurate description of the late-time behavior of the scattered
fields---even in the context of trapping structures wherein no
theoretical justification is currently available. A broad set of
examples confirms the method’s high efficiency and accuracy.

\section*{CRediT authorship contribution statement}
\textbf{Oscar Bruno}: Conceptualization, Methodology, Supervision, Formal analysis, Writing -- review, \& editing. \textbf{Manuel Santana}: Conceptualization, Methodology, Formal analysis, Software development, Writing -- review \& editing.  

\section*{Declaration of competing interest}
The authors declare that they have no known competing financial interests or personal relationships that could have appeared to influence the work reported in this paper.

\section*{Acknowledgments}
The authors gratefully acknowledge support from the Air Force Office
of Scientific Research and the National Science Foundation under
contracts FA9550-21-1-0373, FA9550-25-1-0015, DMS-2109831, and NSF
Graduate Research Fellowship No. 2139433.

\appendix
\section{Appendix: Complex resonances via the combined field
  formulation }\label{app:combined_form}

As is well known~\cite{taylorbook}, the analytic continuation of the
Dirichlet-Helmholtz solution operator $\mathcal{U}^\mathrm{c}_\omega$ to the
domain $\Im\omega < 0$ can be constructed using the inverse of the
operator $C_{\omega,0}$ (i.e., the inverse of $C_{\omega,\eta}$ with
$\eta = 0$). Unfortunately, however, the operator
$(C_{\omega,0})^{-1}$ has certain poles on the real axis that do not
correspond to poles of the solution operator
$\mathcal{U}^\mathrm{c}_\omega$.  These real poles can be avoided by
utilizing the operator $(C_{\omega,\eta})^{-1}$ with $\eta\ne 0$,
instead. But, like $(C_{\omega,0})^{-1}$, the operator
$(C_{\omega,\eta})^{-1}$ has complex poles that do not correspond to
the poles of $\mathcal{U}^\mathrm{c}_\omega$. Fortunately, as shown in
Theorem~\ref{thm:app} below, for any $\eta$ whose sign differs from
that of $\omega$, the poles of the inverse operator
$(C_{\omega,\eta})^{-1}$ in the lower half-plane $\Im\omega \leq 0$
coincide with the poles of $\mathcal{U}^\mathrm{c}_\omega$.

\begin{lemma}\label{lem:Injectivty}
  Let $\omega$ and $\eta$ satisfy $\Im(\omega) < 0$, $\Re(\omega) > 0$
  (resp. $\Re(\omega) < 0$), and $\eta < 0$ (resp. $\eta > 0$)
  . Then $\mathcal{C}_\eta : H^{1/2}(\Gamma) \to H^1_{loc}(\Omega^e)$
  is an injective operator.
\end{lemma}
\begin{proof}
  The proof relies on the fact that, for a given $\psi \in H^{1/2}(\Gamma)$
  and defining the function
  $U(\br,\omega)= \mathcal{C}_\eta[\psi](\br,\omega)$ for
  $\mathbf{r}\not\in\Gamma$, then if $U$ vanishes identically in
  $\Omega^\mathrm{ext}$ then $U$ and $\psi$ satisfy the relation
      \begin{equation}\label{eqn:imaggreenident}
        -\frac{2}{c^2}\Re(\omega)\Im(\omega)\int_{\Omega^\mathrm{i}} |U|^2dx = \eta \int_{\partial \Omega}|\psi|^2ds,
    \end{equation}
    where
    $\Omega^{\mathrm{i}} \coloneqq \R^2 \setminus \Omega \cup
    \Gamma$. This can be established as in~\cite[Theorem
    3.33]{COLTON:1983} by noting that the necessary jump relations are
    valid~\cite[Theorem 6.11]{mclean2000strongly} in the functional
    setting considered here.  In the case $\Re(\omega) > 0$ the
    left-hand term in~\eqref{eqn:imaggreenident} is non-negative,
    which, since $\eta < 0$, implies that $\psi$ vanishes identically,
    and the injectivity of $\mathcal{C}_\eta$ follows in this
    case. The case $\Re(\omega) < 0$, $\eta > 0$ follows similarly.
\end{proof}

\begin{theorem}\label{thm:app}
  Let $\omega \in \mathbb{C}$ such that $\Re(\omega) > 0$
  (resp. $\Re(\omega) < 0$), and let $\eta < 0$ (resp. $\eta >
  0$). Then, for $\Im(\omega) \leq 0$, the set of poles of
  $(C_{\omega,\eta})^{-1}$ coincides with the set of poles of
  $\mathcal{U}_{\omega}^\mathrm{c}$.
\end{theorem}
\begin{proof}
  Since the double- and single-layer
  operators~\eqref{eqn:sl_dl_operator} are compact, the operator
  $(C_{\omega,\eta})^{-1}$ is a meromorphic function of $\omega$ in
  the entire complex plane~\cite[Proposition 7.4]{taylorbook}, except
  for a logarithmic branch cut joining $\omega = 0$ and
  $\omega =\infty$. In view of the
  representation~\eqref{eqn:sol_op_comb} of the solution operator
  $\mathcal{U}^\mathrm{c}_\omega$ it follows that the set of poles of
  $\mathcal{U}_\mathrm{c}(\omega)$ is contained within the set of
  poles of $(C_{\omega,\eta})^{-1}$.

  To show that the converse is also true assume
  $(C_{\omega,\eta})\inv$ has a pole of order $m$ at
  $\omega = \omega_0$ with $\Im(\omega_0) < 0$.  Then there exists an
  element $B \in H^{1/2}(\Gamma)$ such that
    \begin{equation*}
      (C_{\omega,\eta})\inv [B] = (\omega - \omega_0)^{-m}\left(B_m + B_{m+1}(\omega - \omega_0) + \cdots\right)
      \end{equation*}
      for a certain sequence $B_j\in H^{1/2}(\Gamma)$, $j\geq m$, with
      $B_m\ne 0.$ Letting $u_m = \mathcal{C}_\eta[B_m]$ it follows
      that
     \begin{equation*}
        \mathcal{U}^\mathrm{c}_\omega [B] = (\omega - \omega_0)^{-m}u_m + O((\omega - \omega_0)^{-m +1}) \quad \mbox{as}\quad \omega \to \omega_0.
     \end{equation*}
     By Lemma~\ref{lem:Injectivty} $u_m \ne 0$, and, therefore
     $\omega_0$ is a pole of $\mathcal{U}^\mathrm{c}_\omega$. The
     proof is now complete.
\end{proof}

\bibliographystyle{plain}
\bibliography{main}

\end{document}